\documentclass[11pt]{article}

\usepackage{amsmath}
\usepackage{amsthm}
\usepackage{amssymb}
\usepackage{bm}
\usepackage{mathtools}
\usepackage{mdframed}
\usepackage{microtype}
\usepackage{thmtools}
\usepackage{float}
\usepackage[obeyFinal,textsize=scriptsize,shadow]{todonotes}
\usepackage[Symbolsmallscale]{upgreek}
\usepackage{xcolor}
\usepackage{tikz}

\usepackage{custom}

\usepackage{comment,url,algorithm,algorithmic,graphicx,subcaption,relsize}
\usepackage{amssymb,amsfonts,amsmath,amsthm,amscd,dsfont,mathrsfs,mathtools,microtype,nicefrac,pifont}
\usepackage{float,psfrag,epsfig,color,url,hyperref}
\usepackage{upgreek}
\usepackage{epstopdf,bbm,mathtools,enumitem}
\usepackage[toc,page]{appendix}
\usepackage[mathscr]{euscript}
\usepackage[giveninits=true, style=alphabetic, maxbibnames=999, maxalphanames=999]{biblatex}

\usepackage[top=1in, bottom=1in, left=1in, right=1in]{geometry}

\def\balign#1\ealign{\begin{align}#1\end{align}}
\def\baligns#1\ealigns{\begin{align*}#1\end{align*}}
\def\balignat#1\ealign{\begin{alignat}#1\end{alignat}}
\def\balignats#1\ealigns{\begin{alignat*}#1\end{alignat*}}
\def\bitemize#1\eitemize{\begin{itemize}#1\end{itemize}}
\def\benumerate#1\eenumerate{\begin{enumerate}#1\end{enumerate}}

\newenvironment{talign*}
 {\csname align*\endcsname}
 {\endalign}
\newenvironment{talign}
 {\csname align\endcsname}
 {\endalign}

\def\balignst#1\ealignst{\begin{talign*}#1\end{talign*}}
\def\balignt#1\ealignt{\begin{talign}#1\end{talign}}

\let\originalleft\left
\let\originalright\right
\renewcommand{\left}{\mathopen{}\mathclose\bgroup\originalleft}
\renewcommand{\right}{\aftergroup\egroup\originalright}

\def\tinycitep*#1{{\tiny\citep*{#1}}}
\def\tinycitealt*#1{{\tiny\citealt*{#1}}}
\def\tinycite*#1{{\tiny\cite*{#1}}}
\def\smallcitep*#1{{\scriptsize\citep*{#1}}}
\def\smallcitealt*#1{{\scriptsize\citealt*{#1}}}
\def\smallcite*#1{{\scriptsize\cite*{#1}}}

\def\<{\left\langle} 
\def\>{\right\rangle}

\DeclareMathOperator{\Tr}{tr} 

\DeclareSymbolFont{rsfs}{U}{rsfs}{m}{n}
\DeclareSymbolFontAlphabet{\mathscrsfs}{rsfs}

\providecommand{\tr}{\mathop\mathrm{tr}}

\ifdefined\nonewproofenvironments\else
\ifdefined\ispres\else
\newtheorem{theorem}{Theorem}
\newtheorem{lemma}[theorem]{Lemma}
\newtheorem{corollary}[theorem]{Corollary}
\newtheorem{definition}[theorem]{Definition}

\renewenvironment{proof}{\noindent\textbf{Proof.}\hspace*{.3em}}{\qed \vspace{.1in}}
\newenvironment{proof-sketch}{\noindent\textbf{Proof Sketch}
  \hspace*{1em}}{\qed\bigskip\\}
\newenvironment{proof-idea}{\noindent\textbf{Proof Idea}
  \hspace*{1em}}{\qed\bigskip\\}
\newenvironment{proof-of-lemma}[1][{}]{\noindent\textbf{Proof of Lemma {#1}}
  \hspace*{1em}}{\qed\\}
\newenvironment{proof-of-theorem}[1][{}]{\noindent\textbf{Proof of Theorem {#1}}
  \hspace*{1em}}{\qed\\}
\newenvironment{proof-attempt}{\noindent\textbf{Proof Attempt}
  \hspace*{1em}}{\qed\bigskip\\}

\newenvironment{remark}{\noindent\textbf{Remark.}
  \hspace*{0em}}{\smallskip}

\fi

\newtheorem{proposition}[theorem]{Proposition}

\theoremstyle{definition}

\fi
\numberwithin{equation}{section}
\makeatletter
\@addtoreset{equation}{section}
\makeatother

\hypersetup{
  colorlinks,
  linkcolor={red!50!black},
  citecolor={blue!50!black},
  urlcolor={blue!80!black}
}

\newcommand{\BE}{\mathbb E}

\newcommand{\BR}{\mathbb R}

\newcommand{\eps}{\varepsilon}

\usepackage{crossreftools}

\usepackage{tcolorbox}
\usepackage[normalem]{ulem}

\makeatletter
\renewcommand{\paragraph}{%
  \@startsection{paragraph}{4}%
  {\z@}{1.25ex \@plus 1ex \@minus .2ex}{-1em}%
  {\normalfont\normalsize\bfseries}%
}
\newcommand{\itemlabel}[1]{\item[{\crtcrossreflabel{\normalfont{(#1)}}[it:#1]}]}
\makeatother

\mathtoolsset{showonlyrefs}

\newcommand{\calA}{\mc A}

\begin{document}

\title{Query lower bounds for log-concave sampling}

 \author{\!\!\!\!\!
  Sinho Chewi\thanks{
  School of Mathematics at
  Institute for Advanced Study, \texttt{schewi@ias.edu}.
  Part of this work was done while SC was a research intern at Microsoft Research.
 }
 \ \ \
 Jaume de Dios Pont\thanks{
  Department of Mathematics at
  University of California, Los Angeles, \texttt{jdedios@math.ucla.edu}.
 }
 \ \ \
 Jerry Li\thanks{
 Microsoft Research, \texttt{jerrl@microsoft.com}.
 }
 \ \ \
 Chen Lu\thanks{
 Department of Mathematics at Massachusetts Institute of Technology, \texttt{chenl819@mit.edu}.
 }
 \ \ \
 Shyam Narayanan\thanks{
 Department of Electrical Engineering and Computer Science at Massachusetts Institute of Technology, \texttt{shyamsn@mit.edu}.
  Part of this work was done while SN was a research intern at Microsoft Research.
 }
}

\pagenumbering{gobble}
\maketitle

\begin{abstract}
  Log-concave sampling has witnessed remarkable algorithmic advances in recent years, but the corresponding problem of proving \emph{lower bounds} for this task has remained elusive, with lower bounds previously known only in dimension one.
  In this work, we establish the following query lower bounds: (1) sampling from strongly log-concave and log-smooth distributions in dimension $d\ge 2$ requires $\Omega(\log \kappa)$ queries, which is sharp in any constant dimension, and (2) sampling from Gaussians in dimension $d$ (hence also from general log-concave and log-smooth distributions in dimension $d$) requires $\widetilde \Omega(\min(\sqrt\kappa \log d, d))$ queries, which is nearly sharp for the class of Gaussians.
  Here $\kappa$ denotes the condition number of the target distribution.
  Our proofs rely upon (1) a multiscale construction inspired by work on the Kakeya conjecture in geometric measure theory, and (2) a novel reduction that demonstrates that block Krylov algorithms are optimal for this problem, as well as connections to lower bound techniques based on Wishart matrices developed in the matrix-vector query literature.
\end{abstract}

\newpage

\tableofcontents

\clearpage
\pagenumbering{arabic}

\section{Introduction}

We study the problem of sampling from a target distribution on $\R^d$ given query access to its unnormalized density.
This is a fundamental algorithmic primitive arising in diverse fields, such as Bayesian inference, numerical simulation, and randomized algorithms~\cite{robcas2004mc}.
Recently, there has been considerable progress in developing faster algorithms for this problem, particularly in the case where the target distribution is log-concave.
In large part, these results have been achieved by exploiting the rich interplay between optimization and sampling~\cite{jko1998, wibisono2018sampling}, leading to novel sampling schemes inpsired by classical optimization methods~\cite{bernton2018langevin, chewi2020exponential, zhangetal2020mirror, lee2021structured, ma2021there}, as well as new quantitative convergence guarantees for sampling \cite{dalalyan2017theoretical, durmusmajewski2019lmcconvex}. 

In light of such results, many prior works (e.g.,~\cite{chengetal2018underdamped, lee2021lower, chatterjibartlettlong2022oraclesampling}) have raised the foundational question of whether the algorithmic upper bounds are tight.
However, there is still a dearth of lower bounds for log-concave sampling.
This lies in stark contrast to the analogous setting of convex optimization, 
 in which the query complexity has been tightly characterized for a plethora of function classes~\cite{nemirovskij1983problem, nesterov2018cvxopt}.
Such lower bounds yield important insights into the limitations of our existing algorithms and provide guidance towards identifying optimal ones.

Given the deep connections between the two fields, it is natural to ask why optimization lower bounds cannot be converted into sampling lower bounds.
One way to do so is to directly reduce from optimization, as was done in~\cite{gopleeliu2022expomech}.
However, as we are interested in the intrinsic complexity of sampling, we make the standard assumption that the mode of the target distribution is zero to remove the optimization component of the sampling task, which rules out this approach.
Another avenue is to borrow the techniques used for optimization lower bounds, but there are several obstructions to doing so.
First, most optimization lower bounds hold against (classes of) deterministic algorithms and proceed by constructing specific adversarial functions~\cite{bubeck2015convex, nesterov2018cvxopt}.
In contrast, lower bounds for randomized algorithms are relatively recent and still not fully understood~\cite{woosre2017randomizedfirstorder}, which poses a major challenge for sampling algorithms, since they are inherently randomized.
Second, whereas optimization constructions can employ local perturbations to hide the minima, sampling constructions need to hide the bulk of the mass of the target distribution, making them surprisingly delicate.
 
We now describe the problem in more detail.
We consider the canonical setting in which target distribution $\pi$ on $\R^d$ is $\alpha$-strongly log-concave and $\beta$-log-smooth, with its mode located at the origin.
Namely, we assume $\pi\propto \exp(-V)$, where the potential $V:\R^d\to \R$ is twice continuously differentiable, $\alpha$-strongly convex, $\beta$-smooth, and $\nabla V(0) = 0$. 
We let $\kappa \deq \beta / \alpha$ denote the \emph{condition number} of $\pi$.
We study algorithms in which the sampler is given query access to $V$ and $\nabla V$, and the goal is to produce a sample whose law is close to $\pi$ in total variation distance. 
The complexity of the algorithm is measured by the number of queries made. Note that this oracle model captures the majority of sampling algorithms used in practice, including the unadjusted Langevin algorithm, Hamiltonian Monte Carlo, Metropolized random walks, and hit-and-run.

Despite the intense research activity centered on log-concave sampling, only a handful of works address the lower bound question, and the majority of them are either algorithm-specific or pertain to auxiliary problems such as estimation of the normalizing constant; see Section~\ref{scn:related} for related work.
To the best of our knowledge, currently the only general log-concave sampling lower bound is that of~\cite{chewietal2022logconcave1d}, which establishes a sharp query lower bound of order $\Omega(\log\log \kappa)$ in dimension one.
However, that work leaves open the question of obtaining stronger lower bounds in higher dimension, which is the more relevant case for applications.
Even beyond the log-concave setting, we are aware of only one other work that obtains query lower bounds for sampling: the recent result of~\cite{chewietal23fisherlower} is incomparable to the present work, as it considers a different setting, and we discuss it further in Section~\ref{scn:related}.
Overall, the lack of sampling lower bounds points to a lack of tools for addressing this problem and motivates the present work.

\subsection{Our contributions}

In this paper, we make significant progress on this problem by proving new lower bounds for sampling which reach beyond the one-dimensional setting considered in~\cite{chewietal2022logconcave1d}.
In fact, for some settings of interest, our lower bounds match existing upper bounds up to constants, and we therefore obtain some of the first \emph{tight} complexity results for sampling from log-concave distributions in dimension $d > 1$.
We obtain lower bounds in two regimes:

\paragraph{Lower bounds in low dimension.}  Our first lower bound gives a tight characterization of the complexity of log-sampling in any constant dimension $d \ge 2$.
We show:

\begin{theorem}[informal, see Theorem~\ref{thm:sampling-is-logk-hard}]\label{thm:logk_informal}
    For any dimension $d \geq 2$, any sampler for $d$-dimensional log-concave distributions with condition number $\kappa$ requires $\Omega (\log \kappa)$ queries.
\end{theorem}
Note that this result is exponentially stronger than the $\Omega(\log\log \kappa)$ lower bound in the univariate case~\cite{chewietal2022logconcave1d}. Moreover, when the dimension $d$ is held fixed, we obtain a matching $O(\log \kappa)$ algorithmic upper bound, based on folklore ideas from the classical literature on sampling from convex bodies (Theorem~\ref{thm:low_d_upper}).
Together with the result of~\cite{chewietal2022logconcave1d} for $d = 1$, this settles the complexity of log-concave sampling in constant dimension.

On a technical level, the lower bound is based on a novel construction inspired by work on the Kakeya conjecture in geometric measure theory, which we believe may be of independent interest.
We give a detailed description of the construction in Section~\ref{sec:2d_lwr}.

\paragraph{Lower bounds in high dimension.}
Our second set of lower bounds applies to the high-dimensional setting and implies that when the dimension is sufficiently large, a polynomial dependence on the condition number $\kappa$ is unavoidable (in contrast to Theorem~\ref{thm:logk_informal}, which only gives a logarithmic dependence on $\kappa$ in low dimension).
In fact, our lower bounds hold for the special case of sampling from Gaussians, for which they are nearly tight.
We first prove the following theorem.

\begin{theorem}[informal, see Corollary~\ref{cor:wishart_lower_bd}]\label{thm:high_d_informal_1}
    Any sampler for centered $d$-dimensional Gaussians with condition number $\kappa$ requires $\Omega (\min (\sqrt{\kappa}, d))$ queries.
\end{theorem}

We emphasize the fact that in our setting, the Gaussians are centered. Note that if the Gaussians were allowed to have varying means, then one can deduce a sampling lower bound by reducing the optimization task of minimizing a convex quadratic function $x\mapsto \langle (x-x_\star), \Sigma^{-1} \, (x-x_\star)\rangle$ to the task of sampling from the corresponding Gaussian $\NN(x_\star,\Sigma)$.
However, as previously alluded to, this does not address the inherent difficulty of the sampling problem.

The proof of Theorem~\ref{thm:high_d_informal_1} rests upon an elegant technique developed in the literature on the matrix-vector query model (see Section~\ref{scn:related}) in which the  conditioning properties and sharp characterizations of the eigenvalue distribution of Wishart matrices are used to produce difficult lower bound instances for various tasks.
We adapt this method to our context by reducing the task of inverse trace estimation to sampling (see Theorem~\ref{thm:reduce_inv_trace_sampling}).

As we show in Appendix~\ref{sec:gaussian_upper}, the lower bound is nearly tight over the class of Gaussians, as it is possible to sample from a Gaussian using $O(\min(\sqrt{\kappa} \log d, d))$ queries using the block Krylov method.
However, note that the lower bound from Theorem~\ref{thm:high_d_informal_1} does not match the block Krylov upper bound, and the lower bound of Theorem~\ref{thm:high_d_informal_1} is vacuous when $\kappa$ is constant.
In particular, it leaves open the possibility that the complexity of sampling from well-conditioned Gaussians is dimension-free.
While such dimension-free rates are possible in convex optimization, our next  result shows that the same is in fact not possible for log-concave sampling:

\begin{restatable}{theorem}{highd} \emph{(informal, see Theorem~\ref{thm:lower_bound_main})}
\label{thm:high_d_informal_2}
    Let $d$ be sufficiently large, and let $\kappa \leq d^{1/5-\delta}$. Then, any sampler for $d$-dimensional Gaussians with condition number $\kappa$ requires $\Omega_\delta (\sqrt{\kappa} \log d)$ queries.
\end{restatable}

In the regime for which Theorem~\ref{thm:high_d_informal_2} is valid, the lower bound matches the block Krylov upper bound up to constant factors, and hence we settle the complexity of sampling from Gaussians in this regime.
Moreover, Theorems~\ref{thm:high_d_informal_1} and~\ref{thm:high_d_informal_2} together imply the \emph{first} dimension-dependent lower bounds for general log-concave sampling. We conjecture that Theorem~\ref{thm:high_d_informal_2} holds for all $\kappa$ for which $\sqrt\kappa \log d \le d$, and we leave this question for future work.

Although Theorem~\ref{thm:high_d_informal_2} may appear to only be a mild improvement over Theorem~\ref{thm:high_d_informal_1}, analyzing this regime is quite delicate, and we believe that the tools based on Wishart matrices employed in the proof of Theorem~\ref{thm:high_d_informal_1} may be insufficient to reach Theorem~\ref{thm:high_d_informal_2}.
Instead, we prove Theorem~\ref{thm:high_d_informal_2} by first establishing sharp lower bounds on the performance of block Krylov algorithms for the sampling task, and then providing a novel reduction (Lemma \ref{lem:chen_block_krylov}) which shows that block Krylov algorithms are optimal for this task.
This reduction is quite general, and as the block Krylov algorithm and the matrix-vector query model are of wide interest in scientific computing and numerical linear algebra, we believe that our reduction may be broadly useful for tackling other problems in this space.

We remark that a concise way of summarizing Theorems~\ref{thm:high_d_informal_1} and~\ref{thm:high_d_informal_2} if we do not care about lower order terms is that 
 sampling from Gaussians requires $\widetilde{\Omega} (\min (\sqrt{\kappa} \log d, d))$ queries, where we write $f = \widetilde{\Omega} (g)$ to mean $f = \Omega(g \log^{-O(1)}(g))$.

\subsection{Related work}\label{scn:related}

There is a vast literature on from sampling log-concave (and non-log-concave) distributions, and a full survey is beyond the scope of this paper.
For a detailed exposition, see e.g.~\cite{chewi2022logconcave}.

\paragraph{Lower bounds for log-concave sampling.} As previously mentioned, the only unconditional lower bound against log-concave sampling is by~\cite{chewietal2022logconcave1d} for the one-dimensional setting, where the tight bound is $\Theta (\log \log \kappa)$.
Other prior work on sampling lower bounds has fallen largely into one of several categories.
One line of work studies lower bounds against a specific class of algorithm such as underdamped Langevin~\cite{caoluwang2021uldlowerbd} or MALA~\cite{chewi2021optimal,lee2021lower, wuschche2022minimaxmala}.
However, these lower bounds techniques are tailored to the restricted class of algorithms that they consider and are not suitable for proving general query lower bounds.
Another line of work considers lower bounds against computing normalizing constants~\cite{radvem2008vollower, ge2020estimating}. 
The work~\cite{tal2019computationalsampling} also investigates the computational complexity of sampling.

We mention two further lower bounds in different settings.
The work of~\cite{chatterjibartlettlong2022oraclesampling} proves a lower bound against stochastic gradient oracles, and the work of~\cite{gopleeliu2022expomech} proves a lower bound on the number of individual function value (i.e., zeroth-order) queries needed to sample from a density of the form $\exp(-\sum_{i\in I} f_i + \mu \,\norm\cdot^2)$, where each $f_i$ is convex, Lipschitz, and whose domain is the unit ball.
In contrast, we consider deterministic, first-order oracle access.
Moreover, their considerations are somewhat orthogonal to ours:~\cite{chatterjibartlettlong2022oraclesampling} focuses more on the role of noise, whereas we consider exact gradient access; and the lower bound of~\cite{gopleeliu2022expomech} applies a direct reduction from optimization, which is also not in the spirit of the present work (in particular, we explicitly set the mode of the target distribution to zero).

Finally, we also mention the recent work~\cite{chewietal23fisherlower}, which proves query lower bounds for non-log-concave sampling in a different metric (the Fisher information).
This work is inspired by the corresponding upper bounds of~\cite{balasubramanian2022towards} and can be viewed as lower bounds against \emph{local mixing}.

\paragraph{Upper bounds for log-concave sampling.}
Starting with the seminal papers of~\cite{dalalyantsybakov2012sparseregression, dalalyan2017theoretical,durmus2017nonasymptotic}, there has been a flurry of recent work on proving non-asymptotic guarantees for log-concave sampling, with iteration complexities that scale polynomially in the condition number and dimension.
This includes analyses for the classical Langevin dynamics~\cite{wibisono2018sampling,dalalyan2019user, 
 durmusmajewski2019lmcconvex, vempala2019ulaisoperimetry, balasubramanian2022towards, chewietal2022lmcpoincare, AltTal23Langevin}, mirror and proximal methods~\cite{wibisono2019proximal,chewi2020exponential, salric2020proximallangevin, zhangetal2020mirror, ahnchewi2021mirrorlangevin, jia2021mlmc, lee2021structured, chenetal2022proximalsampler, cheeld2022localization, khavem2022riemannianlangevin, lietal2022mirrorlangevin, fanyuanchen23improvedproximal}, the Metropolis-adjusted Langevin algorithm (MALA)~\cite{dwivedi2018log,chen2020fast,lee2020logsmooth, chewi2021optimal, wuschche2022minimaxmala, altche23warm}, and many others~\cite{chengetal2018underdamped, shen2019randomized, dalalyanrioudurand2020underdamped, ding2021random,ma2021there}.

Our upper bound for sampling from Gaussians (Theorem~\ref{thm:gaussian_upper}) is closely related to the use of the conjugate gradient algorithm for sampling from Gaussians~\cite{nissuc2022conjgradgibbs}.
Also, our $O(\log \kappa)$ upper bound algorithm is closely related to rounding procedures which have been previously used in the convex body sampling literature (see, e.g.,~\cite{lovvem2006rounding}).

\paragraph{Matrix-vector product query model.} While matrix-vector queries have been studied in scientific computing for decades (e.g.,~\cite{BaiFG96}), they have only been studied in the theoretical computer science literature recently, with a fully formalized model described in \cite{SunWYZ19}. The most relevant works to ours are those that study the matrix-vector query complexity of spectral properties, such as estimating top eigenvectors \cite{SimchowitzAR18, braverman2020regressionLB}, trace and matrix norms \cite{Hutch90, WimmerWZ14, RashtchianWZ20, DharangutteM21, MeyerMMW21}, the full eigenspectrum \cite{CohenKSV18, BravermanKM22}, and low-rank approximation \cite{MuscoM15, BakshiCW22}. We remark that the non-adaptive matrix-vector product model is closely related to sketching, which has enjoyed a large body of work (see, e.g.,~\cite{Woodruff14} for a survey).

\section{Technical overview}
Here we summarize the main technical ideas used to prove our lower bounds. For details, see Section~\ref{sec:2d_lwr} for Theorem~\ref{thm:logk_informal}, Section~\ref{sec:wishart} for Theorem~\ref{thm:high_d_informal_1}, and Section~\ref{sec:krylov} for Theorem~\ref{thm:high_d_informal_2}.

\subsection{Geometric construction in low dimension}

Theorem~\ref{thm:logk_informal} is proved with a construction in dimension two. For convenience, in this section we use radial coordinates to denote points in $\R^2$, so $\omega \deq (x,y) = (r, \theta)$, where $r\in \R_+$ and $\theta\in [0, 2\pi)$. We denote sectors of $\R^2$ enclosed by angles $\theta_1$ and $\theta_2$ as $S(\theta_1, \theta_2) \deq \{(r, \theta)\in \R^2 : \theta\in [\theta_1, \theta_2]\}$, and denote bounded sectors as $\SB(\theta_1, \theta_2, r) \deq \{(r', \theta)\in \R^2 : \theta\in [\theta_1, \theta_2],\; r' \le r\}$.

The argument is information-theoretic in nature. We will construct a family of strongly log-concave and log-smooth distributions $\{\pi_1, \dots, \pi_m\}$, where each $\pi_b \propto \exp(-V_b)$, which satisfies two key properties.
First, different distributions $\pi_b$ and $\pi_{b'}$ are well separated in total variation distance; and second, if $b$ is chosen uniformly at random from $[m]$, then querying the potential $(V_b(\omega), \nabla V_b(\omega))$ at any $\omega\in \R^2$ will reveal $O(1)$ bits of information about $b$. The lower bound in Theorem~\ref{thm:logk_informal} follows readily from the existence of such a family, provided that $m$ and $\kappa$ are polynomially related. On the one hand, because the distributions are well-separated in total variation, if we can sample well from the distribution $\pi_b$ using queries, we can identify the index $b$ with high probability. On the other hand, because there are $m$ distributions and every query reveals $O(1)$ bits of information about $b$, we need at least $\Omega(\log m) = \Omega(\log \kappa)$ queries to identify $b$, which results in a $\Omega(\log \kappa)$ query lower bound for log-concave sampling.

How do we construct such a family? A first attempt is to consider distributions supported on thin convex sets that have no overlap. For $b = \frac{1}{\kappa}, \frac{2}{\kappa}, \dots, 1$, let $\pi_b = \unif(\mc Z_b)$, where $\mc Z_b = \SB(\frac{\pi}{2}\,b, \frac{\pi}{2}\,(b + \frac{1}{2\kappa}), 1)$, and the size of the family is $m = \lfloor\kappa\rfloor$.
The potential $V_b$ is the convex indicator of $\mc Z_b$, i.e., it is $0$ on $\mc Z_b$ and $+\infty$ outside.
Morally, the distributions $\pi_b$ can be thought of as having condition number $\kappa$. 

This family does satisfy the two properties needed for the lower bound: different distributions are certainly well-separated because they have disjoint supports; and when we query any potential $V_b$ at a point $\omega\in \R^2$, we always receive one bit of information: whether or not $\omega$ lies in the support of $\pi_b$. However, the distributions in this family are neither strongly log-concave nor log-smooth. It is easy to make them strongly log-concave while still satisfying the desired properties: we can adjust the distributions by adding the same quadratic function $\frac{\norm{\cdot}^2}{2}$ to all of the potentials $V_b$. But it is much harder to make this family log-smooth.

One way to make this construction log-smooth is to let the potentials $V_b$ grow slowly (linearly) to infinity outside of the their zero sets $\mc Z_b$, which leads to a modified second attempt: for $m = \kappa^{\Omega(1)}$, $b = \frac{1}{m}, \dots, 1$, let $\pi_b$ have potential $V_b = \tilde V_b + \frac{\norm{\cdot}^2}{2\kappa^{O(1)}}$, where $\mc Z_b = S(\frac{\pi}{2}\,b, \frac{\pi}{2}\,(b+\frac{1}{2m}))$, and $\tilde V_b(\omega) =\kappa \dist(\omega, \mc Z_b)$.
Note that the potentials $V_b$ are in fact still not smooth at the boundaries of the sets $\mc Z_b$, but this can be fixed by mollifying $V_b$. The distributions in this family will be well-separated, because an $\Omega(1)$ fraction of the mass of $\pi_b$ will lie in $\mc Z_b$, and the sets $\mc Z_b$ are disjoint for different $b$. Unfortunately, this family no longer reveals $O(1)$ bits per query: for any $\omega\in \R^2$, we can identify $b$ with a single query to $(V_b(\omega), \nabla V_b(\omega))$, because either $\omega\in \mc Z_b$, or $\nabla V_b(\omega)$ reveals the direction of $\mc Z_b$, and in both cases the index $b$ itself is identified.

We can reduce the information revealed by queries by more carefully controlling the growth of $\tilde V_b$, so that the further away a point $\omega$ lies from $\mc Z_b$, the fewer the number of bits will be revealed by $(\tilde V_b(\omega), \nabla \tilde V_b(\omega))$. This motivates a third attempt at the construction. For $m = 2^{N} = \kappa^{\Omega(1)}$, $b = \frac{1}{m}, \dots, 1-\frac{1}{m}$, let $b = 0.b_1 \dots b_N$ be the binary expansion of $b$, and let $[b]_k = 0.b_1\dots b_k$ be the truncation of $b$ up to the $k$-th bit. For $k = 1, \dots, N$, let $\mc Z^{\text{radial}}_{k, b} = S(\frac{\pi}{2}\,[b]_k, \frac{\pi}{2}\,([b]_k + 2^{-k}))$, and let $\phi^{\text{radial}}_{k, b}(x) = \kappa^{O(1)}\, 2^{-k}\dist(x, \mc Z^{\text{radial}}_{k,b})$. Finally, let $V^{\text{radial}}_b = \frac{\norm{\cdot}^2}{2\kappa^{O(1)}} + \tilde V^{\text{radial}}_b$, where
\begin{align*}
    \tilde V^{\text{radial}}_b = \max_{k = 1, \dots, N} \phi^{\text{radial}}_{k, b}\, .
\end{align*}
The potentials $V^{\text{radial}}_b$ will again have to be mollified to be made smooth. It turns out that the potentials $\tilde V^{\text{radial}}_b$ will grow fast enough outside $\mc Z^{\text{radial}}_{N, b}$ such that the distributions will be well-separated. It also turns out that queries indeed reveal $O(1)$ bits of information on average. This can be seen as follows: note that the sets $\mc Z^{\text{radial}}_{k,b}$ are sectors such that $\mc Z^{\text{radial}}_{k, b} \supset \mc Z^{\text{radial}}_{k+1, b}$, and as $k$ increases, $\mc Z^{\text{radial}}_{k,b}$ becomes thinner around the ray $\{\theta = \frac{\pi}{2}\,b\}$; also note that as $k$ increases, the growth rate of $\phi^{\text{radial}}_{k, b}$ outside its zero set $\mc Z^{\text{radial}}_{k, b}$ is decreasing; these two properties imply that if we query a point $\omega = (r,\theta)$ that is far from the sector $\mc Z^{\text{radial}}_{i, b}$ (in the sense that $\theta \not\in [\frac{\pi}{2}\,[b]_i - 100\cdot 2^{-i}, \frac{\pi}{2}\,[b]_i + 100\cdot 2^{-i}]$), then the value of $\tilde V^{\text{radial}}_b(\omega)$ will not depend on any $\phi^{\text{radial}}_{k, b}$ for $k > i$, and hence querying $\tilde V^{\text{radial}}_b(\omega)$ will only reveal $b$ up to the $i$-th bit. As a result, if $b$ is chosen uniformly, then for a fixed query $\omega$ with high probability we will have $\omega\not\in \mc Z^{\text{radial}}_{k, b}$ for any $k = O(1)$, so the query will only reveal $O(1)$ bits of information about $b$.

Yet this construction fails because of the mollification step, which we have so far ignored. To make the potentials $V_b$ smooth, we will instead take $V_b = \chi_\delta * \tilde V^{\text{radial}}_b + \frac{\norm{\cdot}^2}{2\kappa^{O(1)}}$, where $\chi_\delta$ is supported on a ball of radius $\delta < 2^{-2N}$. We would hope that the potential $\chi_\delta * \tilde V^{\text{radial}}_b$ still satisfies the property that querying a point $\omega = (r,\theta)$ that is far from $\mc Z^{\text{radial}}_{i, b}$ only reveals $b$ up to the $i$-th bit. When $r$ is not too close to the origin (say $r > 100\cdot 2^{-i}$), this is indeed still true: if $\omega$ satisfies $\theta \not\in [\frac{\pi}{2}\,[b]_i - 200\cdot 2^{-i}, \frac{\pi}{2}\,[b]_i + 200\cdot 2^{-i}]$, then the entire $\delta$-neighbourhood of $\omega$ will satisfy $\theta \not\in [\frac{\pi}{2}\,[b]_i - 100\cdot 2^{-i}, \frac{\pi}{2}\,[b]_i + 100\cdot 2^{-i}]$, so the value of $\tilde V^{\text{radial}}_b$ on the $\delta$-neighbourhood of $\omega$ will not depend on any $\phi^{\text{radial}}_{k, b}$ for $k > i$, hence the value of $(\chi_\delta * \tilde V^{\text{radial}}_b)(\omega)$ will also not reveal any information of $b$ beyond the $i$-th bit.
But when $\omega$ is very close to the origin ($r < \delta$), the $\delta$-neighbourhood of $\omega$ will intersect $\mc Z^{\text{radial}}_{N, b}$, which means that the value of $(\chi_\delta * \tilde V^{\text{radial}}_b)(\omega)$ will depend on $\phi^{\text{radial}}_{k, b}$ for all $k$ and hence on all bits of $b$. In other words, mollification leaks information around the origin. As a result, if we query points $\delta$-close to the origin, we will again identify $b$ in a single query.

The way to resolve the leakage at the origin is to create a branching structure, such that all $V_b$ are equal near the origin so that no information is leaked at small scales, and such that far away from the origin $V_b$ is small around the ray $\{\theta = \frac{\pi}{2}\, b\}$ so that $\pi_b$ still concentrates around different sectors. We keep the choices of $m$ and $b$ from the previous construction. The potentials will be $V_b = \chi_\delta * \tilde V_b + \frac{\norm{\cdot}^2}{2\kappa^{O(1)}}$, where $\tilde V_b = \max_{k=1, \dots, N} \phi_{k, b}$, and  $\phi_{k, b}(\omega) = \kappa^{O(1)}\, 2^{-k} \dist(\omega, \mc Z_{k,b})$. The zero set $\mc Z_{k, b}$, instead of being a radial sector like $\mc Z^{\text{radial}}_{k, b}$, is now thickened adaptively.

We intuitively describe how to generate $\mc Z_{k, b}$. Each $\mc Z_{k, b}$ will be a thickening of $\mc Z_{k, b}^{\text{radial}}$, by simply including all points within some distance $d_k$ of $\mc Z_{k, b}^{\text{radial}}$. We define $\mc Z_{\le k, b} := \bigcap_{k' \le k} \mc Z_{k', b}$: note that each $\mc Z_{\le k, b}$ is getting smaller as $k$ increases, and $\mc Z_{\le N, b}$ is the zero set of $\tilde V_b$.

Consider some radii $r_0 < r_1 < r_2 < \dots$.
To generate $\mc Z_{1, b}$, we thicken $\mc Z_{1, b}^{\text{radial}}$ (corresponding to the radial sector matching on the first bit), so that it contains $\SB(0, \pi/2, r_0)$ (corresponding to the quarter-circle near the origin). This avoids leaking information near the origin, as every $x$ within radius $r$ will be in $\mc Z_{1, b}$, which means $\phi_{1, b}$ will also be $0$.
Indeed, we can thicken $\mc Z_{1, b}^{\text{radial}}$ just the right amount so that it contains $\SB(0, \pi/2, r_0).$
For the concrete example where $N = 4$, and $b = 0.1010$, we show a description of $\mc Z_{1, b}$ in Figure \ref{fig:1a}: we shade $\SB(0, \pi/2, r_0)$ in dark blue, $Z_{1, b}^{\text{radial}} = S(\pi/4, \pi/2)$ in medium blue, and the additional thickening required in light blue.

To generate $\mc Z_{k, b}$ for $k \ge 2$, we thicken a much thinner angular sector. This ensures that at large radii, the arc of $\mc Z_{k, b}$ is not too big. We will inductively thicken $\mc Z_{k, b}$ by some amount $d_k$ just enough to contain $\mc Z_{k-1, b} \cap \SB(0, \pi/2, r_{k-1})$. 
Consider one more example for $k = 2$ (again for $N = 4$, and $b = 0.1010$), in Figure \ref{fig:1b}. Note that $\mc Z^{\text{radial}}_{2,b}$ is the sector $S(\frac{\pi}{4}, \frac{3\pi}{8})$ (shaded in medium blue), and the thickened region (in light blue emanating from both sides of the sector) is just enough to capture all of $\mc Z_{1, b}$ that was within radius $r_1$. However, for larger radii, $\mc Z_{2, b}$ is much thinner than $\mc Z_{1, b}$.
In addition, if we know the first bit $b_1 = 1$, then querying $V_b$ anywhere in $\{r \le r_1\}$ will not reveal any information about the second bit $b_2$. This is because either we were in $\mc Z_{1, b}$ which only depends on $b_1$ (in which case $\phi_{1, b} = \phi_{2, b} = 0$ as we thickened to make sure $\mc Z_{2, b} \supset Z_{1, b} \cap \SB(0, \pi/2, r_1)$), or we weren't, in which case $\phi_{1, b}$ grows much more quickly than $\phi_{2, b}$.

We can also continue this process inductively for $k = 3, 4$ (Figures \ref{fig:1c} and \ref{fig:1d}): we show $\mc Z_{\le k, b}$.
The intuition for why this prevents leaking of information near the origin is that even if $k$ is large, $\mc Z_{k, b}$ in the smaller-radius regions is decided by $\mc Z_{k', b}$ for $k' \ll k,$ so we cannot learn any later bits.

The comparisons of $\mc Z^{\text{radial}}_{k,b}$ and $\mc Z_{\le k,b}$ for $b = 0.1010$ and for all $k\le 4$ are shown together in Figure~\ref{fig:angular_sectors_1}. The picture is not to scale, and the radial arcs represent the radii $r_i = 2^i r_0$, for $i=0, \dots, 4$.

\begin{center}
\begin{figure}[htbp]
    {\centering
    \begin{subfigure}{0.24\textwidth}
        \centering
        \includegraphics[width=\textwidth]{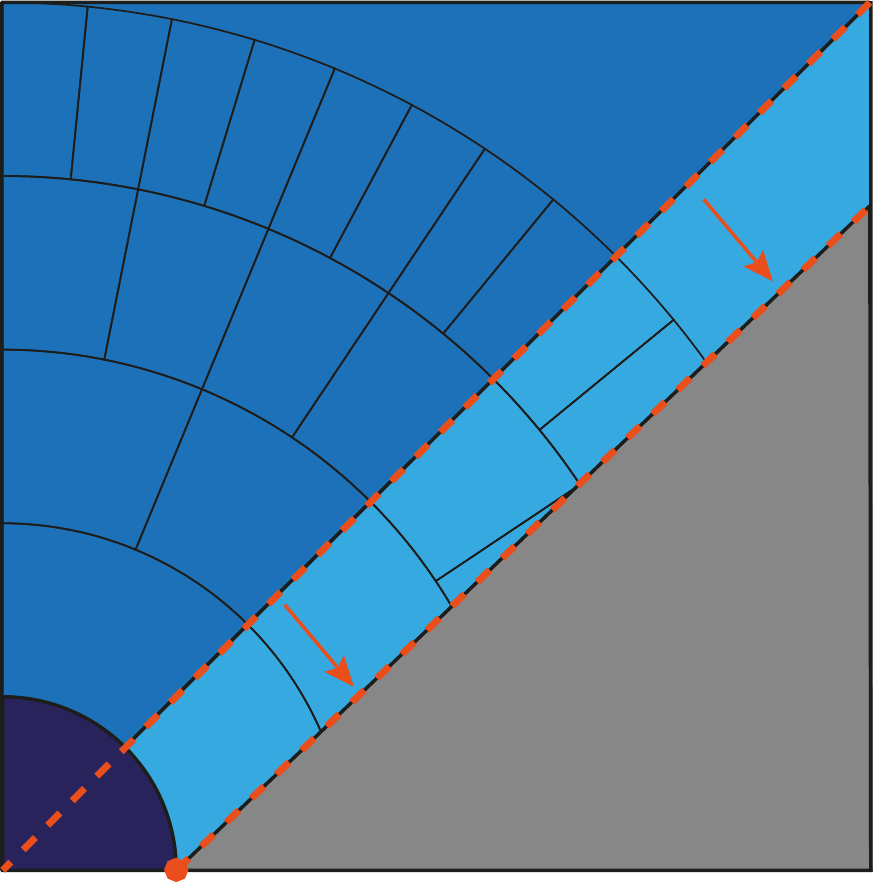}%
        \caption{$k=1$}
        \label{fig:1a}
    \end{subfigure}
    \hfill
    \begin{subfigure}{0.24\textwidth}
        \centering
        \includegraphics[width=\textwidth]{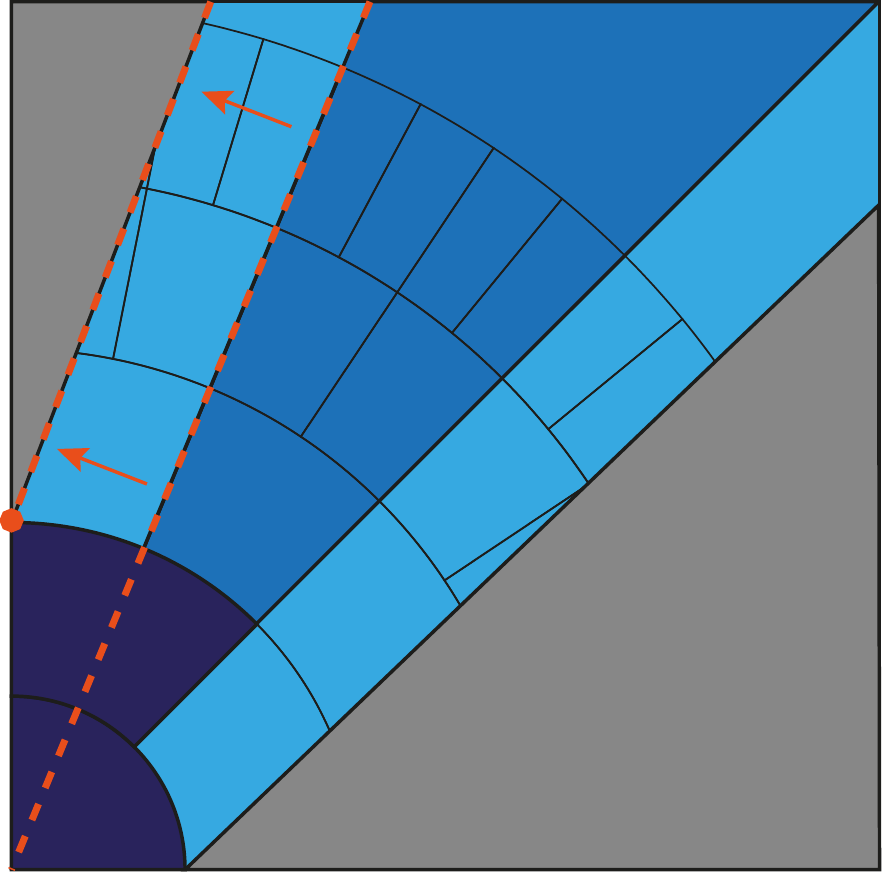}%
        \caption{$k=2$}
        \label{fig:1b}
    \end{subfigure}
    \hfill
    \begin{subfigure}{0.24\textwidth}
        \centering
        \includegraphics[width=\textwidth]{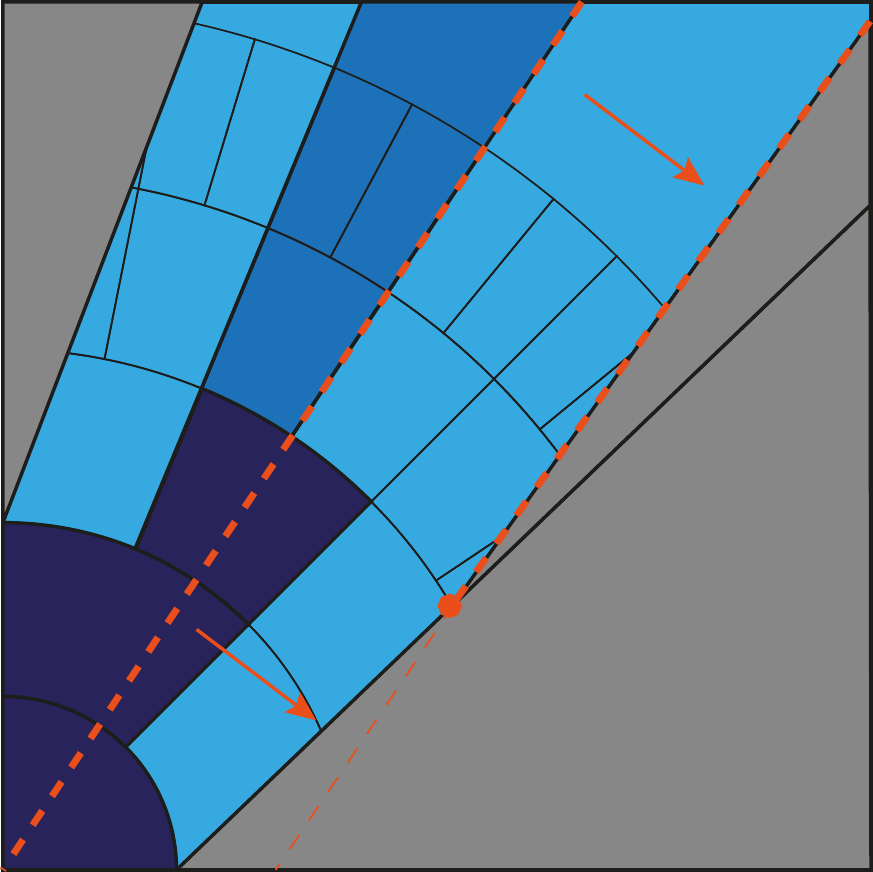}%
        \caption{$k=3$}
        \label{fig:1c}
    \end{subfigure}
    \hfill
    \begin{subfigure}{0.24\textwidth}
        \centering
        \includegraphics[width=\textwidth]{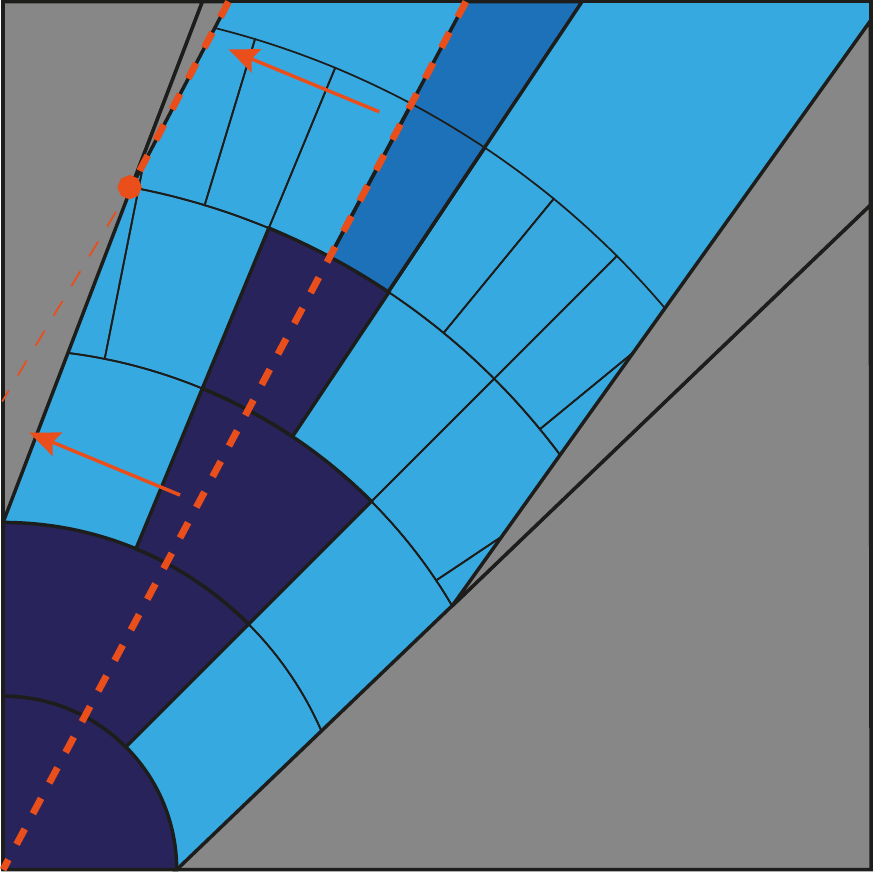}%
        \caption{$k=4$}
        \label{fig:1d}
    \end{subfigure}
    }
    \caption{Comparison of $\mc Z^{\text{radial}}_{\le k, b}$ (the sector in medium blue) with $\mc Z_{k,b}$ (union of dark, medium, and light blue), for $k = 1, 2, 3, 4$, and $b = 0.1010$. Dark blue represents the larger angular sectors closer to the origin, and light blue represents the additional fattening from taking sumsets.
    Each $\mc Z_{k,b}$ is constructed by thickening $\mc Z^{\text{radial}}_{k, b}$ enough (illustrated by the red arrows) such that no information about the $k$-th bit is revealed close to the origin, but $\mc Z_{k, b}$ continues to get thinner at large radii.}
    \label{fig:angular_sectors_1}
\end{figure}
\end{center}
\vspace{-24pt}

\noindent The construction of $\mc Z_{\le k, b}$ means that for $k>1$, querying $\phi_{k,b}$ within $\{r \le 2^{k-1}r_0\}$ will not reveal the $k$-th bit, and so even querying the mollified $\chi_\delta*\phi_{k,b}$ within $\{r \le 2^{k-2}r_0\}$ will not reveal the $k$-th bit, which stops information leaking near the origin.

Since $\tilde V_b = \max_{k=1,\dots,N} \phi_{k,b}$, the zero set of $\tilde V_b$ coincides with $\mc Z_{\le N,b}$, and for the choice of $b =0.1010$, this is shown in the first panel of Figure~\ref{fig:angular_sectors_2}. It turns out that each $\pi_b$ will concentrate around the zero set of $\tilde V_b$, and the other panels of Figure~\ref{fig:angular_sectors_2} show these zero sets for seven different values of $b$ in the set $\{\frac{1}{16},\dotsc, \frac{15}{16}\}$ at larger scales. We can see that far out from the origin the zero sets become well-separated, and hence the distributions are well-separated in total variation.

We already discussed how the thickening of $\mc Z_{k,b}$ means that querying $\phi_{k,b}$, and hence $\tilde V_b$, near the origin will not reveal the higher bits of $b$. For query points $\omega = (r, \theta)$ where $r$ is large, the same analysis on $\tilde V^{\text{radial}}_b$ tells us that $\tilde V_b(x)$ (even after mollification) will reveal $O(1)$ bits of information about $b$ when $b$ is chosen uniformly. As mentioned earlier, such a family of distributions readily leads to a sampling lower bound of $\Omega(\log m)$, where $m$ is the size of the family. Since we can choose $m = \kappa^{\Omega(1)}$, this leads to the $\Omega(\log \kappa)$ lower bound. Details of the proof can be found in Section~\ref{sec:2d_lwr}.

\begin{figure}[H]
    \centering

    \includegraphics[width=0.23\textwidth]{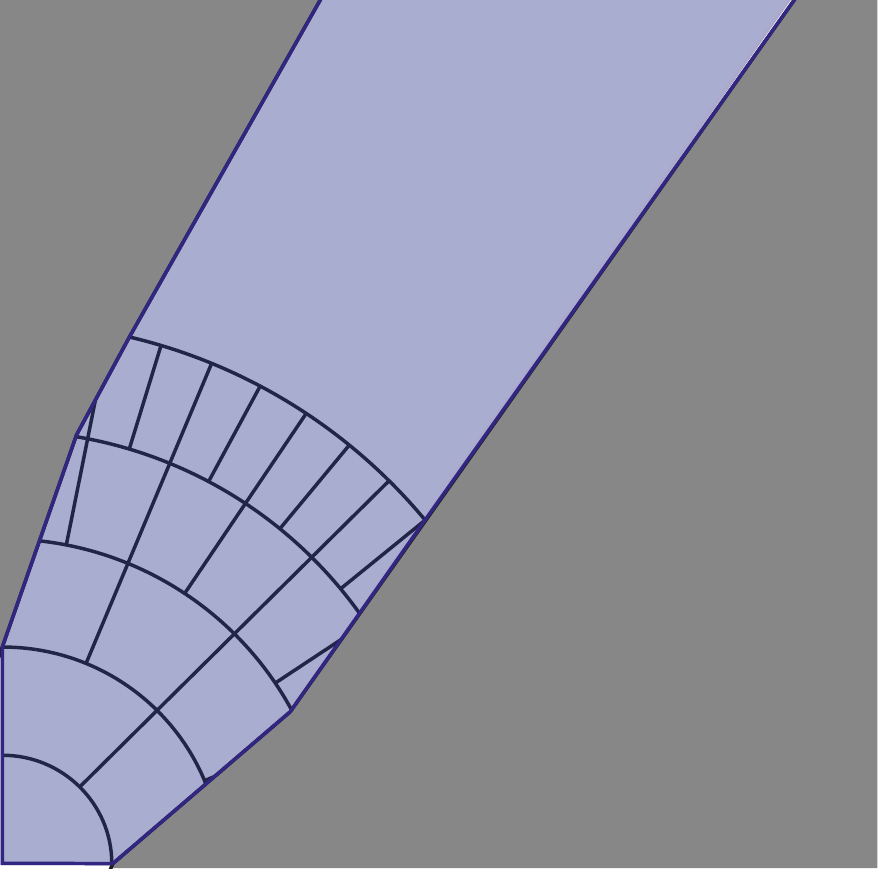}%
    \hfill
    \includegraphics[width=0.23\textwidth]{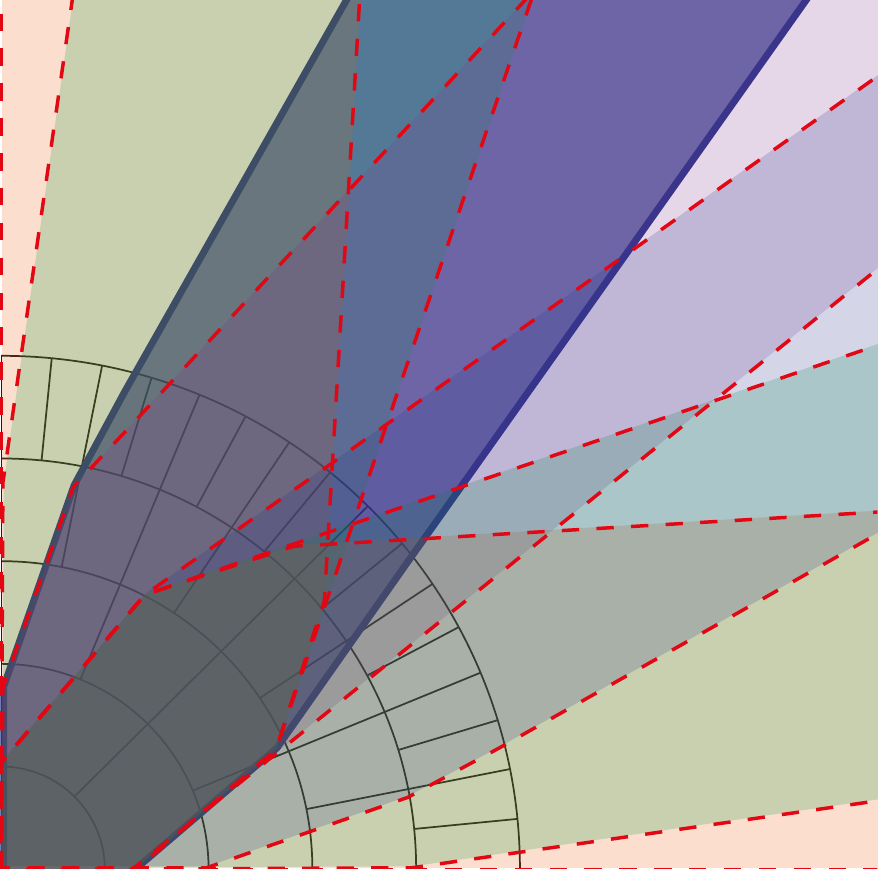}%
    \hfill
    \includegraphics[width=0.23\textwidth]{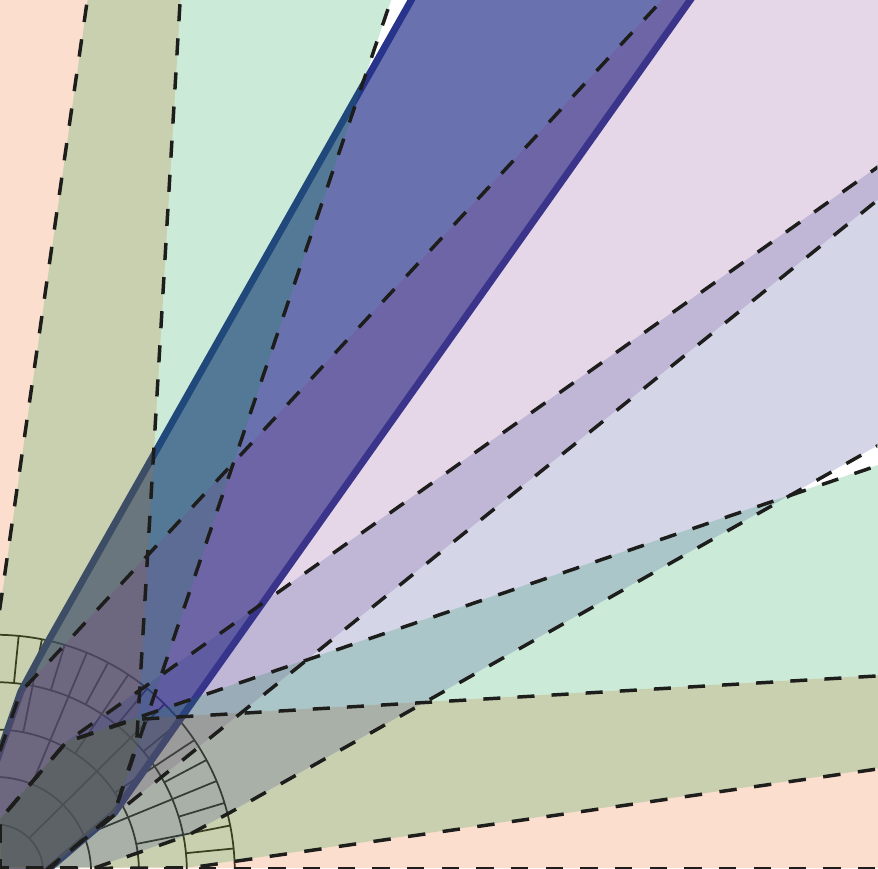}%
    \hfill
    \includegraphics[width=0.23\textwidth]{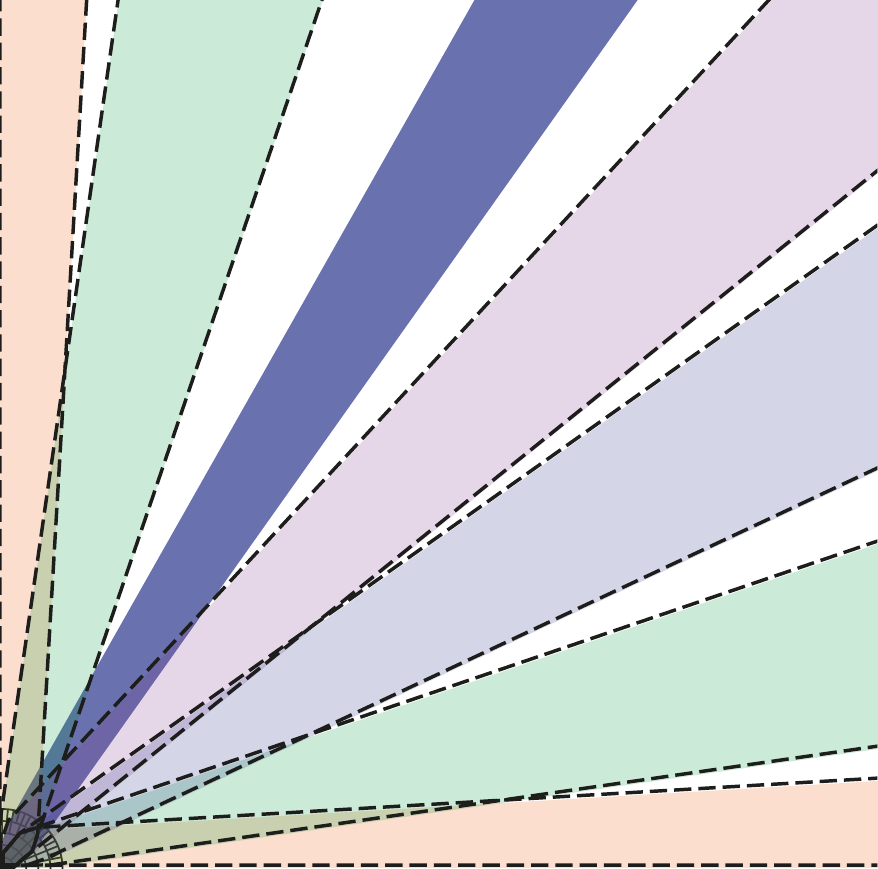}%

    \caption{Zeros sets of $\tilde V_b$. The first panel shows the zero set for $b = 0.1010$. The other panels show the zeros sets for different values of $b$ at different scales. Note that far away from the origin the zero sets become well-separated, which leads to the distributions being well-separated in total variation.
    Note that if $b, b'$ match in the first $\ell$ bits, then they will agree up to the $\ell$-th circle, as those circles only depend on $\mc Z_{\le \ell, b}$ even for $\ell$ much less than $K$.}
    \label{fig:angular_sectors_2}
\end{figure}

\vspace{-12pt}
\paragraph{Connections to Kakeya constructions.}
The construction outlined above is related to Perron's construction \cite{perron1928satz} of Besicovich (Kakeya) sets known as \emph{Perron trees}. Kakeya sets are sets with area zero that contain the translation of a unit segment in any direction.
While Kakeya sets over finite fields
have been investigated before in theoretical computer science, e.g.,~\cite{saraf2008improved, dvir2009size,jukna2011extremal}, our construction is inspired by Kakeya sets over continuous domains, namely $\R^2$.
To our knowledge, this is one of the first applications of these geometric ideas to theoretical computer science.

There are many similarities between our construction and that of Perron.
Perron's construction proceeds by the method of sprouting. Sprouting is an iterative process in which, at each step, one adds further and further smaller triangles to the pre-existing construction. The figure is then rescaled in order to have height $1$. The construction after $n$ steps contains $2^{n}$ triangles of small aperture $\Omega(2^{-n})$, and has area $O(n^{-1})$.
We do a similar process in the definition of our sets $\mc Z_{k,b}$, and indeed, ultimately our hard instance has a very similar tree-like structure.

While we were inspired by the construction of Perron trees, there are also key differences between our hard instance and Perron's construction.
Indeed, in our setting, we need to minimize overlap (so that the resulting distributions are well-separated) while simultaneously ensuring that information is not leaked by queries.
In contrast, Kakeya sets are explicitly designed to maximize overlap.
Secondly, the iterates of Perron trees are convex sets, not convex functions. 
One must turn these convex sets into convex functions somehow.
This is additionally complicated by the fact that these iterates are not nested.
In our construction, we must take great care to create nested convex sets, so that the resulting functions are convex and still maintain the structure of the sets.

\subsection{Lower bounds for sampling from Gaussians}

We now turn to our lower bounds against sampling from Gaussians.
Recall that our goal is to provide a lower bound on sampling from a Gaussian $\mc N(0, \Sigma)$, where $\Lambda \deq \Sigma^{-1}$ has condition number $\kappa$. Note that the corresponding potential is $V(x) = \frac{1}{2}\, \langle x, \Lambda \, x\rangle$, and we are allowed zeroth-order and first-order queries, which means for a query $x$, we receive $x^\top \Lambda x$ and $\Lambda x$. Hence, adaptive queries are equivalent to adaptive matrix-vector product computations with $\Lambda$.

The first observation we make is that we can reduce the problem of sampling from the Gaussian to estimating the trace of $\Sigma$.
This is because if $X$ is a sample from a distribution which is close in total variation distance to $\NN(0, \Sigma)$, then $\|X\|_2^2 \approx \tr (\Sigma)$ with high probability.
Therefore, it suffices to demonstrate a lower bound for the following problem: given matrix-vector product computations with $\Lambda$, approximately compute $\tr (\Lambda^{-1})$.

\vspace{-7pt}
\subsubsection{Lower bound via Wishart matrices}
For any $d$, let $W \in \R^{d\times d}$ have the $\Wishart (d)$ distribution.
That is, $W = XX^\top$, where $X \in \R^{d \times d}$ has i.i.d.\ $\NN(0, 1/d)$ entries.
We take $W$ to be the precision matrix, $\Lambda = W$.
Our first lower bound shows that $\Omega (d)$ matrix-vector queries with $W$ are necessary to estimate the trace of $W^{-1}$ even to constant multiplicative accuracy, with constant success probability (Theorem \ref{thm:inv_trace_lower_bd}). 
Since the condition number of $W$ is $\Theta (d^2)$ with high probability, we obtain one extreme of the claimed lower bound $\Omega (\min (\sqrt{\kappa}, d))$.
The general lower bound for all $\kappa$ then follows from a padding argument.

This lower bound approach is inspired by \cite{braverman2020regressionLB}, which proved a query lower bound for estimating the minimum eigenvalue of $W$. Their approach relies on the fact that if we condition on any sequence of $(1-\Omega(1))\,d$ adaptive queries, the posterior distribution of the remaining eigenvalues behaves similarly to the original distribution of the eigenvalues of $W$. In addition, while the smallest eigenvalue of $W$ is usually about $1/d^2$, its distribution has heavy tails: with probability $\Theta(\sqrt{\eps})$, the smallest eigenvalue of $W$ is below $\eps/d^2$. Consequently, even conditioned on $d/2$ adaptive queries, we are unable to learn the minimum eigenvalue up to a constant factor with high probability.

In our setting, we instead wish to show that learning the trace of $W^{-1}$ is hard. However, the smallest eigenvalue of the Wishart matrix is so small that with high probability, $\Tr(W^{-1}) = \Theta(\lambda_{\min}(W)^{-1})$.
While most of the time the trace is $O(d^2)$, with probability $\Theta(\sqrt\varepsilon)$ the posterior distribution of the smallest eigenvalue of $W$ after our adaptive queries may be $\eps/d^2$. Hence, we will be unable to determine whether the trace is $\le O(d^2)$ or $\ge \Omega(d^2/\eps)$ with high probability.

This lower bound technique is clean and nearly optimal, but as previously mentioned it is vacuous (of constant order) when $\kappa = O(1)$, whereas we expect the complexity of the problem to increase as $d \to \infty$. To tackle this setting, we introduce a second approach.

\vspace{-7pt}
\subsubsection{Lower bounds via reduction to block Krylov}\label{scn:block_kry_overview}

Our second technique works in two parts. 
First, we show that for a specific hard distribution over instances, any block Krylov-style algorithm requires $\Omega (\min (\sqrt{\kappa} \log d, d ))$ queries to estimate $\tr(\Sigma)$.
Then, we show a general purpose reduction which demonstrates that for this hard instance (and indeed, any rotationally invariant instance), block Krylov methods are actually optimal.

\paragraph{Lower bound for block Krylov algorithms.} 
Recall the block Krylov technique: the algorithm chooses $K$ i.i.d.\ random vectors $v_1, \dots, v_K \sim \mc N(0, I)$, and computes $\Lambda^j v_k$ for all $j \le T, k \le K$. This can be done using $KT$ adaptive queries, by querying $\Lambda^j v_k$ to learn $\Lambda^{j+1} v_k$.
For our purposes, it suffices to consider block Krylov algorithms with $K=T$ and to prove a lower bound on the smallest number $K$ needed to successfully estimate $\tr(\Sigma)$, for $\Sigma = \Lambda^{-1}$. 

We will construct two diagonal matrices $D, D'$ with all eigenvalues between $1$ and $\kappa$, such that $\Tr(D^{-1})$ and $\Tr((D')^{-1})$ are sufficiently different. In addition, if $\Lambda, \Lambda'$ are random rotations of $D, D'$, respectively, then $\{\Lambda^j v_k\}_{j,k\le K}$ and $\{(\Lambda')^j v_k\}_{j,k\le K}$ are hard to distinguish for $K \le c \sqrt{\kappa} \log d$ for a small constant $c$ (Lemma \ref{lem:block_krylov}).
Thus, unless $K\ge \Omega(\sqrt{\kappa} \log d)$, we cannot estimate the trace.

To explain the intuition behind Lemma \ref{lem:block_krylov}, we first consider what happens if we only have $\{\Lambda^j v\}_{j\le K}$ for a single random vector $v$ (i.e., power method). Letting $\lambda_1, \dots, \lambda_d$ be the eigenvalues of $\Lambda$, we have $\Lambda^j v = \sum_{i=1}^{d} \lambda_i^j \alpha_i u_i$, where $u_i$ is the $i$-th eigenvector of $\Lambda$ and $v = \sum_{i=1}^d \alpha_i u_i$. Intuitively, the only information we obtain from these vectors are their pairwise inner products, since we could have randomly rotated $\Lambda$. Therefore, the only information we have is $\langle \Lambda^j v, \Lambda^{j'} v \rangle = \sum_{i=1}^{d} \lambda_i^{j+j'} \alpha_i^2$, which is the set $\{\sum_{i=1}^d \lambda_i^j \alpha_i^2\}_{j \le 2K}$. Since $v$ is random, we may think of all of the $\alpha_i^2$ as $1$ for simplicity, and so we know $\{\sum_{i=1}^d \lambda_i^j\}_{j \le 2K}$. Our goal is to use this information to learn $\Tr(\Lambda^{-1}) = \sum_{i=1}^d \lambda_i^{-1}$. 

We connect this to the problem of estimating $1/x$ as a linear combination of $1, x, x^2, \dots, x^K$, a classic problem in approximation theory that is often tackled with \emph{Chebyshev polynomials}. Indeed, this relation to Chebyshev polynomials is the main tool in the analysis of essentially all Krylov methods. In our setting, as we desire lower bounds, we apply the fact that Chebyshev polynomials are \emph{optimal} in generating certain approximations. More concretely, suppose that there are only $K$ distinct eigenvalues $\lambda_1, \dots, \lambda_K$, with each $\lambda_i$ having some multiplicity $N_i$. Since we want to show that estimating $\tr(\Lambda^{-1})$ is hard, this amounts to showing that knowing $\sum_{i=1}^K N_i \lambda_i^j$ for $0 \le j \le K$ is insufficient to learn $\sum_{i=1}^K N_i/\lambda_i$. We express this as a linear program (if we relax the $N_i$ to be reals), the dual of which precisely captures whether $1/x$ can be approximated well by a degree-$K$ polynomial at $\lambda_1, \dots, \lambda_K$ (Proposition~\ref{prop:chebyshev_dual_lp}). If we choose the $\lambda_i$ to be the local extrema of a degree-$K$ Chebyshev polynomial, shifted so that $\lambda_1 = 1$ and $\lambda_K = \kappa$, then it is known that one cannot estimate $1/x$ up to error $d^{-\Omega(1)}$ at these points (which is needed for trace estimation), unless $K \ge \Omega(\sqrt{\kappa} \log d)$. At a high level, this is the reason why we need $\Omega(\sqrt{\kappa} \log d)$ iterations of the power method.

For general block Krylov algorithms, the algorithm obtains $\langle v_\ell, \Lambda^j\, v_k\rangle$, for $0 \le j \le K$ and $1 \le k, \ell \le K$. 
Now, the information that the algorithm sees is captured by the matrices $\{\langle v_\ell, \Lambda^j\, v_k\rangle\}_{k, \ell \le K}$, for $j = 1, \ldots, K$.
Here, we show that provided $K$ is sufficiently small compared to $d$, we can find choices of multiplicities $N_1, \dots, N_K$ and $N_1', \dots, N_K'$, such that the corresponding matrices $D, D'$ have significantly different traces (i.e., $\sum_{i=1}^K (N_i-N_i')/\lambda_i$ is large) but the information from queries is not enough to distinguish between $\Lambda$ and $\Lambda'$, which we establish via a coupling argument.

\paragraph{Reduction to block Krylov algorithms.}
The argument outlined above shows block Krylov algorithms with $K = o(\sqrt{\kappa} \log d)$ cannot distinguish between two families of randomly rotated matrices with difference traces ($\Lambda$ coming from $D$ and $\Lambda'$ coming from $D'$), and hence cannot solve the trace estimation task. 
Our next technical contribution is a reduction which allows us to simulate the output of any \emph{adaptive} algorithm with $K$ queries on our hard instance, given only the responses to a block Krylov algorithm. 
Thus, a lower bound against block Krylov methods translates into a lower bound against any query algorithm.
We now give a high-level description of the reduction.

Since we prove lower bounds based on randomized constructions, it suffices to consider adaptive deterministic algorithms, i.e., each query $v_k$ is a deterministic function of the previous queries and oracle outputs.
The difficulty of proving such a lower bound against such an algorithm is the adaptivity of the queries, which makes it difficult to reason about how much information the algorithm has learned. However, since our lower bound construction for block Krylov algorithms is rotationally invariant, intuitively the adaptivity does not help: the algorithm may as well query a random direction which it has not yet explored.

However, this intuition is not entirely correct: if the algorithm has previously queried a vector $v$ and received the information $\Lambda v$, then it may useful to query $\Lambda v$ in order to receive the information $\Lambda^2 v$, instead of querying a completely random new direction. Indeed, computing powers $v, \Lambda v, \Lambda^2 v, \dotsc$ is precisely the essence of the power method, as discussed above.
To account for this, we move to the following stronger oracle model: if the algorithm has selected vectors $v_1,\dotsc,v_{k}$, then at iteration $k$ it receives all of the information $(\Lambda^i v_j)_{i+j\le k}$ for free.
Now, there is provably no benefit to querying vectors which lie in the span of the previous queries and oracle outputs.

Recall that our goal is to argue that an adaptive deterministic algorithm can be simulated by an algorithm which simply makes i.i.d.\ Gaussian queries $z_1,z_2,\dotsc,z_K$, in the following sense.
In the stronger oracle model, at iteration $k$, the adaptive algorithm has made queries $(\valg_1,\dotsc,\valg_{k})$ and received information $(\Lambda^i \valg_j)_{i+j \le k}$ and it picks a new vector $v_{k+1}$ which lies orthogonal to its received information.
Suppose that using only the Gaussian queries $z_1,z_2,\dotsc,z_{k}$, we have simulated queries $\vssim_1,\vssim_2,\dotsc,\vssim_{k}$ which are equivalent to the execution of the adaptive algorithm in the sense that the law of the information $(\Lambda^i \vssim_j)_{i+j \le k}$ is precisely the same as the law of the algorithm's information $(\Lambda^i \valg_j)_{i+j \le k}$.
Since the algorithm is deterministic, $\valg_{k}$ is a function $v_{k}((\Lambda^i \valg_j)_{i+j < k})$ of algorithm's accumulated information.
Thus, in order to simulate the adaptive algorithm for one more step, it is natural to consider taking $\vssim_k \deq v_k((\Lambda^i \vssim_j)_{i+j < k})$.
However, we will be unable to compute $\Lambda^i \vssim_k$ for any $i \ge 1$, because the simulation must be based on the Gaussian queries $z_1,z_2,\dotsc,z_k$, whereas this definition of $\vssim_k$ requires making queries at $\vssim_1,\vssim_2,\dotsc,\vssim_{k-1}$.

Thus far, we have not invoked the rotational invariance of $\Lambda$, which is crucial to the argument.
The key is that although we cannot directly take $v_k((\Lambda^i \vssim_j)_{i+j<k})$ to be our next simulated point, we can \emph{rotate} $\vsim_k$ into $v_k((\Lambda^i \vssim_j)_{i+j<k})$ via a unitary matrix $U_k$; moreover, we can arrange that $U_k$ fixes all of the previous information $(\Lambda^i \vssim_j)_{i+j<k}$, because $v_k((\Lambda^i \vssim_j)_{i+j<k})$ lies orthogonal to this information (recall, we can assume that each deterministic function $v_k(\cdot)$ outputs a vector orthogonal to its inputs, due to our choice of oracle model).
The intuition is that due to the rotational invariance of $\Lambda$, then conditioned on the data $(\Lambda^i \vssim_j)_{i+j<k}$, the distribution of $\Lambda$ is still rotationally invariant on the orthogonal subspace of the data; hence, $U_k\vsim_k = v_k((\Lambda^i \vssim_j)_{i+j<k})$ ought to have the same law as $\vsim_k$, i.e., querying the completely random direction $\vsim_k$ is just as good as querying according to what the adaptive algorithm specifies.

Unfortunately there are further difficulties to overcome with this approach. Namely, suppose that we define each simulated point $\vssim_k$ to be the output $U_k \vsim_k$ of a rotation matrix applied to $\vsim_k$.
We would like to take $U_k$ such that $U_k\vsim_k = v_k((\Lambda^i \vssim_j)_{i+j<k})$ but this is no longer computable based on $(\Lambda^i \vsim_j)_{i+j<k}$.
However, we note that $\Lambda^i \vssim_j = \Lambda^i U_j \vsim_j = U_j \tilde \Lambda^i \vsim_j$ where $\tilde\Lambda \deq U_j^\T \Lambda U_j$.
This shows that $\Lambda^i \vssim_j$ is computed from the query of $\vsim_j$, not on the original matrix $\Lambda$ but on the modified matrix $\tilde\Lambda$, together with the matrix $U_j$.
Since we hope that $\tilde\Lambda$ has the same law as $\Lambda$, then this is good enough for the purposes of simulating the adaptive algorithm.
Actually, in order for the induction to work out, it becomes clear that we need to define a sequence of matrices $\Lambda_1,\Lambda_2,\dotsc,\Lambda_k$, where each $\Lambda_k$ is related to the previous $\Lambda_{k-1}$ via $\Lambda_k = U_k^\T \Lambda_{k-1} U_k$, and $U_k$ is chosen such that $\vssim_k = U_k\vsim_k = v_k((\Lambda_{k-1}^i \vssim_j)_{i+j<k})$.
Then, we must argue that the simulated sequence $\vssim_1,\vssim_2,\dotsc,\vssim_k$ has the same law as the algorithm's sequence $\valg_1,\valg_2,\dotsc,\valg_k$.

This last step, however, turns out to be delicate.
Indeed, although it is obvious that for a \emph{fixed} orthogonal matrix $U'$, the law of $\Lambda$ is the same as the law of $(U')^\T \Lambda U'$, the rotation matrices $U_k$ we choose in the above argument are dependent on the previous queries and oracle outputs, and are hence dependent on $\Lambda$ itself.
In the presence of such dependence, it is not obvious why the law of $\Lambda_k$ should be the same as the law of $\Lambda$, and to address this we prove a conditioning lemma in Section~\ref{scn:conditioning}.
Once the conditioning lemma is proved, the remainder of the proof follows along the lines just described, and the details of the induction are carried out in Section~\ref{scn:from_query}.

\section{A general sampling lower bound in dimension two}\label{sec:2d_lwr}

\subsection{Overview}\label{scn:2d_overview}

Our goal is to show the following theorem:

\begin{theorem}[lower bound in dimension two]\label{thm:sampling-is-logk-hard}
There is a universal constant $\varepsilon_0 > 0$ such that the following holds.
The query complexity of sampling from the class of distributions $\pi\propto\exp(-V)$ on $\R^2$ such that $V$ is $1$-strongly convex, $\kappa$-smooth, and minimized at $0$, with accuracy $\varepsilon_0$ in total variation distance, is at least $\Omega(\log \kappa)$.
\end{theorem}

The strategy to do so will be to construct a finite family $\mc S$ of potentials in the given class which satisfies the following two properties:
\begin{itemize}
    \item The potentials are \emph{hard to identify via queries} (in the sense of Definition~\ref{def:hard-to-distinguish} below), and therefore any algorithm must query $V$ at $\Omega(\log\kappa)$ points in order to identify which $V \in \mc S$ the algorithm is querying.
    \item The potentials are \emph{well-separated} (in the sense of Definition~\ref{def:well-separated} below), which loosely means that they have mostly non-overlapping support and hence (by  Proposition~\ref{prop:well-separated-sample-separated}) a single sample from $\pi\propto \exp(-V)$ suffices to identify $V \in \mc S$ with constant probability.
\end{itemize}

Before describing the potentials $\mc S$ in more detail, we note some basic definitions.

\begin{definition}
    Given two functions $f, g: \BR^d \to \BR$, the \emph{convolution} $f \ast g$ is the function defined as $(f \ast g)(x) \deq \int_{\BR^d} f(y)\, g(x-y)\, \D y$, for all $x \in \BR^d$.
\end{definition}

\begin{definition}
    For $\delta > 0$, we define $\chi_\delta$ to be the \emph{indicator function of the ball $B_\delta$} of radius $\delta$ around the origin. By this, we mean $\chi_\delta(x) = 1$ if $\|x\|_2 \le \delta$, and $\chi_\delta(x) = 0$ otherwise.
\end{definition}

The family $\mc S$ of potentials will have cardinality $\kappa^{\Omega(1)}$, so that identification of the potential requires $\Omega(\log \kappa)$ bits of information. Actually, by rescaling the potentials, it suffices for each potential $V$ to be $\kappa^{-O(1)}$-convex and $\kappa^{O(1)}$-smooth.
Our eventual construction also satisfies the following properties.

\begin{itemize}
    \item Each $V \in \mc S$ is of the form $V = \tilde V \ast \chi_{\delta}+ \norm \cdot^2/(2\kappa^{O(1)})$, where $\tilde V : \R^2\to\R$ is a convex, non-negative, and piecewise linear potential, and $\delta$ will have scale $\delta = \kappa^{-\Theta(1)}$.
    \item Each $V\in\mc S$ is zero in a small neighbourhood of a ray $\ell$ emanating from the origin, and grows fast outside of this ray; hence, the potentials are well-separated.
    \item Suppose that $\ell$, $\ell'$ are the rays corresponding to two potentials $V, V' \in \mc S$.
    At distances from $\ell$ and $\ell'$ that are much larger than the angle $\angle(\ell,\ell')$, the potentials $V$, $V'$ are exactly equal. This is the property makes the potentials hard to identify via queries.
\end{itemize}

Throughout the proof, we assume that $\kappa$ is sufficiently large, $\kappa \ge \Omega(1)$.

\subsection{Definitions and the information-theoretic argument}

\begin{definition}[density and normalizing constant]
Given a strictly convex function $V : \R^d\to\R$, we denote by $P_V$ the probability distribution with density $Z^{-1}\exp(-V)$ w.r.t.\ Lebesgue measure, where $Z \deq \int \exp(-V)$ is the normalizing constant. In an abuse of notation, we also use $P_V$ to refer to the density itself.
\end{definition}

\begin{definition}[queries and extended oracle]\label{def:extended_oracle}
For a fixed potential $V$, and given a query $x\in \R^d$, the extended oracle responds with $V(B_\delta (x_1))$, which consists of the value of $V$ for all points in the ball of radius $\delta$ centered at $x$. For a sequence of (possibly adaptive and randomized) queries $x_1, \dotsc, x_n$ and observations $V(B_\delta(x_1)), \dotsc, V(B_\delta(x_n))$, we denote the information from the $i$-th query by $\xi_i \deq \{x_i, V(B_\delta(x_i))\}$, and the information from all the queries by
\begin{align*}
    \xi_{1:n}
    &\deq \{\xi_1,\dotsc,\xi_n\}\,.
\end{align*}
\end{definition}

Note that the extended oracle in Definition~\ref{def:extended_oracle} provides more information (the set of values of the potential in some ball around the query point $x$) to the algorithm than our original first-order query model, from which the algorithm only observes $(V(x), \nabla V(x))$ at the query $x$. A lower bound for sampling in this stronger query model clearly implies a lower bound in the original query model. We consider the stronger model out of technical convenience, as this notion is robust to the mollification in the construction of the potentials.

\begin{definition}[hard to identify via queries]\label{def:hard-to-distinguish}
A finite set $\mc S$  of potentials in $\mathbb R^d$ is called $\mc I$-\emph{hard to identify with queries at scale $\delta$} if the following holds: for $V\sim \unif(\mc S)$, any sequence of queries $x_1,\dotsc,x_n$ to the extended oracle made by a deterministic adaptive algorithm satisfies
\begin{align*}
    I(\xi_{1:n}; V) \le \mc I n\, ,
\end{align*}
where $I$ denotes the mutual information.
\end{definition}

\begin{definition}[well-separated set]
\label{def:well-separated}
A set $\mc S$ of potentials is \emph{well-separated} if there is a family of measurable sets ${(\Omega_V)}_{V\in \mc S}$ where the sets $\Omega_V$ are disjoint, and a universal constant $c > 0$ such that
\begin{align*}
    P_V(\Omega_V) \ge c\,, \qquad\text{for all}~V\in \mc S\,.
\end{align*}
\end{definition}

The motivation for this definition is the following fact:

\begin{proposition} 
[one sample identifies well-separated distributions]
\label{prop:well-separated-sample-separated}
Let $\mc S$ be a well-separated set of potentials and conditionally on $V \sim \unif(\mc S)$, suppose that $X$ is a sample from a probability measure $\widehat P_V$ which is at most $\frac{c}{2}$ away from $P_V$ in total variation distance.
Then,
\begin{align*}
    \Pr\{X \in \Omega_V\}
    &\ge \frac{c}{2}\,.
\end{align*}

\end{proposition}
\begin{proof}
    By conditioning on $V$,
    \begin{align*}
        \Pr\{X\in\Omega_V\}
        = \E\Pr\{X \in \Omega_V \mid V\}
        = \E\widehat P_V(\Omega_V)
        \ge \E\bigl[P_V(\Omega_V) - \norm{P_V - \widehat P_V}_{\rm TV}\bigr]
        \ge \frac{c}{2}\,,
    \end{align*}
    which is what we wanted to show.
\end{proof}

This shows that the minimum-distance estimator
\begin{align}\label{eq:min_dist_est}
    \widehat V
    &\deq \argmin_{V\in \mc S} \inf_{z\in \Omega_V}{\norm{X-z}}
\end{align}
succeeds at estimating the randomly drawn $V$ with constant probability.
On the other hand, we have Fano's inequality from information theory.

\begin{theorem}[{Fano's inequality,~\cite[Theorem 2.10.1]{coverthomas2006infotheory}}]\label{thm:fano}
    Suppose that $\mc S$ is a finite set and $V \sim \unif(\mc S)$. Suppose that $\widehat V$ is any estimator which is based on some data $\xi$. Then,
    \begin{align*}
        \Pr\{\widehat V \ne V\}
        &\ge 1 - \frac{I(\xi; V) + \log 2}{\log{\abs{\mc S}}}\,.
    \end{align*}
\end{theorem}

Fano's inequality enables us to reduce Theorem~\ref{thm:sampling-is-logk-hard} to the following proposition:

\begin{proposition}
[well-separated set which is hard to identify via queries]
\label{prop:construction_exists}
Let $\kappa \ge \Omega(1)$. Then, there is a set $\mc S$ of potentials such that: 

\begin{enumerate}
    \item All elements of $\mc S$ are $\kappa^{-O(1)}$-convex and $\kappa^{O(1)}$-smooth, and have their minimum at zero.
    \item $\mc S$ has cardinality $\kappa^{\Omega(1)}$.
    \item $\mc S$ is well-separated with $c = \Omega(1)$.
    \item $\mc S$ is hard to identify via queries at scale $\delta = \kappa^{-\Theta(1)}$, and with $\mc I = O(1)$.
\end{enumerate}
\end{proposition}

\medskip{}

\begin{proof}[Proof of Theorem~\ref{thm:sampling-is-logk-hard}]
    Suppose that there is a sampling algorithm which, given any target distribution $\pi\propto\exp(-V)$ on $\R^2$ such that $V$ is $1$-strongly convex, $\bar\kappa$-smooth, and minimized at $0$, outputs a sample $X$ whose law is $\varepsilon_0$ close in total variation distance to $\pi$ using $n(\bar\kappa)$ queries to the extended oracle. Let $\mc S$ be the family in Proposition~\ref{prop:construction_exists}.
    By choosing $\varepsilon_0 = c/2 = \Omega(1)$ and rescaling the potentials accordingly, then Proposition~\ref{prop:well-separated-sample-separated} implies that the sampling algorithm can identify $V \sim \unif(\mc S)$ using $n(\bar\kappa)$ queries with constant probability, where $\bar\kappa = \kappa^{O(1)}$. Namely, for the estimator $\widehat V$ in~\eqref{eq:min_dist_est},
    \begin{align}\label{eq:samp_err}
        \Pr\{\widehat V = V\}
        &\ge \frac{c}{2} = \Omega(1)\,.
    \end{align}

    On the other hand, we can prove a lower bound for the error probability of any estimator $\widehat V$ constructed using adaptive queries. First we assume that the estimator is deterministic given previous queries. Because the set $\mc S$ is hard to identify, by Fano's inequality (Theorem~\ref{thm:fano}) we have
    \begin{align}\label{eq:est_err}
        \Pr\{\widehat V \ne V\}
        &\ge 1 - \frac{I(\xi_{1:n(\bar\kappa)}; V) + \log 2}{\log{\abs{\mc S}}}
        \ge 1 - \frac{\mc In(\bar\kappa) + \log 2}{\log{\abs{\mc S}}}
        = 1 - O\Bigl(\frac{n(\bar\kappa)}{\log\kappa}\Bigr)\,,
    \end{align}
    for all $n(\bar\kappa) \le c\, |\mc S| = O(\log \kappa)$. If the estimator is instead randomized, it depends on a random seed $\zeta$ that is independent of $V$. In this case, the same argument as above conditional on $\zeta$ gives
    \begin{align*}
        \Pr\{\widehat V \not= V \mid \zeta \} \ge 1 - \Omega \Bigl(\frac{n(\bar\kappa)}{\log \kappa}\Bigr)\, .
    \end{align*}
    Taking expectation over $\zeta$, we see that \eqref{eq:est_err} holds also for randomized algorithms. Combined with \eqref{eq:samp_err}, we see that $n(\bar\kappa) \ge \Omega(\log\kappa) =\Omega(\log\bar\kappa)$.
\end{proof}

\subsection{Reductions and properties of the construction}

Recall from Section~\ref{scn:2d_overview} that each $V \in \mc S$ is of the form $V = \tilde V * \chi_\delta + \norm\cdot^2/(2\kappa^{O(1)})$.
In this section, we reduce the desired properties of $\mc S$, namely that $\mc S$ is well-separated and hard to identify via queries, to geometric properties of the potentials summarized in Proposition \ref{prop:geometric_construction} below.

By increasing $\kappa$ by a factor of at most two, which will not harm the final lower bound, we can assume that $\kappa = 2^N$ for some positive integer $N$.
We also set $\delta \deq \kappa^{-5}$.
Let $B_N$ denote the set of binary strengths of length $N$.
For each $b \in B_N$ and $\ell \in [N]$, we let $[b]_\ell \deq 0.00b_1\dots b_\ell$ in binary representation, and set $[b] \deq [b]_N$.

\begin{proposition}[geometric properties]
\label{prop:geometric_construction}
There are functions $\tilde V_b$, for $b\in B_N$, such that:

\begin{enumerate}
     \itemlabel{P0} $\tilde V_b$ is convex and $\kappa^{O(1)}$-smooth on average at scale $\delta = \kappa^{-5}$, i.e., $\tilde V_b * \chi_\delta$ is $\kappa^{O(1)}$-smooth, and attains its minimum $V_b(0)=0$ at zero.

    \itemlabel{P1} The zero set $\mc Z_b \deq \{\tilde V_b = 0\}$ contains the $10^3\delta$-neighborhood of the set
    \begin{align}\label{eq:tilde_Z}
        \tilde{\mc Z}_b
        &\deq\{(x,\beta x) \in \R^2 \mid x \ge 0,\; [b]-2^{-N} \le \beta \le [b]+2^{-N}\}\,,
    \end{align}
    and is contained in the $1$-neighbourhood of $\tilde{\mc Z}_b$. 

    \itemlabel{P2}Moreover, for all $x,y\in\R^2$,
    $$
    \tilde V_b(x, y) \ge \kappa^{4}\, \bigl(\dist((x,y), \tilde{\mc Z}_b)-1\bigr)_+\,.
    $$

    \itemlabel{P3} If $b$, $b'$ coincide in the first $\ell$ bits then $\tilde V_b$ and $\tilde V_{b'}$ coincide in the set 
    $$\bigl\{(x,y)\in\R^2\bigm\vert x<\frac{1}{4}\, 2^{-3N}~\text{or}~|y-[b]_\ell\, x|>{100}
    \cdot 2^{-\ell} x\bigr\}\,.$$
\end{enumerate}
\end{proposition}

We check that these properties imply that Proposition~\ref{prop:construction_exists} holds.

\medskip{}

\begin{proof}[Proof of Proposition~\ref{prop:construction_exists}]
Let $\mc S$ be the collection of potentials $V_b  \deq  \tilde V_b\ast \chi_\delta + \norm\cdot^2/(2\kappa^{16})$ for $b \in B_N$, where $\{\tilde V_b : b\in B_N\}$ are the functions from Proposition~\ref{prop:geometric_construction}. We now verify the four properties of Proposition~\ref{prop:construction_exists}.

\underline{\textbf{Proof of 1.}}
By~\ref{it:P0}, we know that $\tilde V_b$ is convex, which implies that $\tilde V_b \ast \chi_\delta$ is also convex. Therefore, $V_b$ is $\kappa^{-16}$-strongly convex. In addition, by \ref{it:P0}, $\tilde V_b \ast \chi_\delta$ is $\kappa^{O(1)}$-smooth, which means that $V_b$ is $\kappa^{O(1)}+\kappa^{-16} \le \kappa^{O(1)}$-smooth.

\underline{\textbf{Proof of 2.}}
By construction, $\abs{\mc S} = \kappa$.

\underline{\textbf{Proof of 3.}}
We now show that $\mc S$ is $c$-separated. For any string $b$, recall the definition of $\tilde{\mc Z}_b$ from~\eqref{eq:tilde_Z}. Define the set
\begin{align*}
    \Omega_b \deq \{(x, \beta x) \in \R^2 \mid x \ge 2^{-3N},\; [b]-0.4 \cdot 2^{-N} \le \beta \le [b]+0.4 \cdot 2^{-N}\}\,.
\end{align*}
It is clear that $\{\Omega_b : b\in B_N\}$ is a family of disjoint sets. By~\ref{it:P1} we know that the zero set $\mc Z_b$ of $\tilde V_b$ contains a $10^3\delta$-neighborhood of $\tilde{\mc Z}_b$. Since $\Omega_b \subset \tilde{\mc Z}_b$, it follows that $\tilde V_b \ast \chi_\delta = 0$ on $\Omega_b$.

Let $\tilde\Omega_b \deq \{(x,y)\in \Omega_b : \norm{(x,y)} \le \kappa^8\}$.
Note that the full set of points $(x, y)$ with $\norm{(x,y)} \le \kappa^8$ has volume $\pi \kappa^{16}$, and $\Omega_b$ is a sector of the plane with arc $\Theta(2^{-N})$, minus a small set of points (specifically, the points in the sector with $x \le 2^{-3N}$, which also means $y \le O(2^{-3N})$). Therefore, the volume of $\tilde\Omega_b$ is $\Theta(\kappa^{16} \cdot 2^{-N}) = \Theta(\kappa^{15})$. In addition, all points $(x,y) \in \tilde\Omega_b$ have $V_b(x,y) = -\|(x,y)\|^2/(2\kappa^{16}) \ge -1/2$. Hence,
\begin{align}\label{eq:lower_bd_mass}
    \int_{\Omega_b} \exp(-V_b)
    &\ge \int_{\tilde\Omega_b} \exp(-V_b)
    \ge \Omega(\kappa^{15})\,.
\end{align}

Next, we bound the full integral of $\exp(-V_b)$ across $\BR^d$ by splitting $\R^d$ into four regions $\R^d = \tilde{\mc Z_b} \cup \Psi_{1,b} \cup \Psi_{2,b} \cup \Psi_{3,b}$, defined as follows:
\begin{itemize}
    \item $\Psi_{1,b} \deq \{(x,y) \in \R^2 \setminus \tilde{\mc Z}_b : \dist((x,y), \tilde{\mc Z_b}) \le 2, \; \norm{(x,y)} \le \kappa^9\}$.
    \item $\Psi_{2,b} \deq \{(x,y) \in \R^2 \setminus (\tilde{\mc Z}_b \cup \Psi_{1,b}) : \norm{(x,y)} \le \kappa^9\}$.
    \item $\Psi_{3,b} = \R^2 \setminus (\tilde{\mc Z}_b \cup \Psi_{1,b} \cup \Psi_{2,b})$.
\end{itemize}
Note that all points $\Psi_{3,b}$ have norm at least $\kappa^9$.
To show that most of the mass of $P_{V_b}$ is concentrated on $\tilde{\mc Z_b}$, we must show that the integrals over $\Psi_{1,b}$, $\Psi_{2,b}$, and $\Psi_{3,b}$ are small. In a nutshell, the integral over $\Psi_{1,b}$ is small because the $2$-neighborhood of $\tilde{\mc Z_b}$ is small (relative to the size of $\tilde{\mc Z}_b$ itself); the integral over $\Psi_{2,b}$ is small because $\tilde V_b$ increases rapidly outside $\tilde{\mc Z_b}$; and the integral over $\Psi_{3,b}$ is small because the Gaussian part of $V_b$ is small over this region.

On these four regions, we have the following bounds.
First, $\int_{\BR^2} \exp(-\norm\cdot^2/(2\kappa^{16})) = 2\pi \kappa^{16}$. Therefore, since the sector $\tilde{\mc Z}_b$ has arc $\Theta(2^{-N})$, by rotational symmetry
\begin{align*}
    \int_{\tilde{\mc Z_b}} \exp(-V_b)
    &\le \int_{\tilde{\mc Z_b}} \exp\bigl(- \frac{\norm\cdot^2}{2\kappa^{16}}\bigr)
    \le O(2^{-N}) \int_{\R^2} \exp\bigl(-\frac{\norm\cdot^2}{2\kappa^{16}}\bigr)
    \le O(\kappa^{15})\,.
\end{align*}
Note that $\Psi_{1,b}$ consists of two strips adjacent to $\tilde{\mc Z_b}$, where each strip has width $2$ and length $O(\kappa^9)$, together with a piece of area $O(1)$ near the origin.
Thus, $\vol(\Psi_{1,b}) \le O(\kappa^9)$, yielding
\begin{align*}
    \int_{\Psi_{1,b}} \exp(-V_b)
    \le \vol(\Psi_{1,b})
    \le O(\kappa^9)\,.
\end{align*}
Next, for $(x,y) \in \R^2$ such that $\dist((x,y), \tilde{\mc Z}_b) \ge 3/2$, by~\ref{it:P2} we have $\tilde V_b(x, y) \ge \kappa^4$.
After mollification at scale $\delta \le 1/2$, we conclude that $\tilde V_b * \chi_\delta \ge \kappa^4$ on $\Psi_{2,b}$.
In addition, $\Psi_{2,b}$ is contained in the ball of radius $\kappa^9$, so the volume of $\Psi_{2, b}$ is at most $\pi \kappa^{18}$. Therefore,
\begin{align*}
    \int_{\Psi_{2, b}} \exp(-V_b)
    &\le \pi \kappa^{18} \exp(-\kappa^4)\,.
\end{align*}
Finally, all points in $\Psi_{3, b}$ have $\ell_2$ norm at least $\kappa^9$, so
\begin{align*}
    \int_{\Psi_{3, b}} \exp(-V_b)
    &\le \iint_{\|\cdot\| \ge \kappa^9} \exp\bigl(-\frac{\norm{\cdot}^2}{2\kappa^{16}}\bigr)
    \le O(\kappa^8) \exp\bigl(-\Omega(\kappa^2)\bigr)\,,
\end{align*}
by standard Gaussian tail estimates.
Therefore,
\begin{align}\label{eq:upper_bd_mass}
    \int_{\R^2}\exp(-V_b)
    &\le O\Bigl(\kappa^{15} + \kappa^9 + \exp\bigl(-\Omega(\kappa^4)\bigr) + \exp\bigl(-\Omega(\kappa^2)\bigr)\Bigr)
    \le O(\kappa^{15})\,.
\end{align}

Overall,~\eqref{eq:lower_bd_mass} and~\eqref{eq:upper_bd_mass} together imply that $P_{V_b}(\Omega_b) \ge \Omega(1)$, i.e., $\mc S$ is $\Omega(1)$-well-separated.

\underline{\textbf{Proof of 4.}}
Finally, we show that $\mc S$ is hard to identify via queries at scale $\delta = \kappa^{-\Theta(1)}$ with $\mc I = O(1)$. We consider $b$ drawn uniformly at random from $B_N$.

First, however, we need to extend~\ref{it:P3} to $V_b$ (i.e., taking into account the mollification at scale $\delta$).
We claim that if $b$, $b'$ coincide in the first $\ell$ bits, then $V_b$ and $V_{b'}$ coincide in the set
\begin{align}\label{eq:modified_P3}
    \bigl\{(x,y)\in\R^2 \bigm\vert x < \frac{1}{8} \, 2^{-3N}~\text{or}~\abs{y-[b]_\ell \, x} > 200 \cdot 2^{-\ell} x\bigr\}\,.
\end{align}
In light of~\ref{it:P3}, it suffices to show that if $(x,y)$ lies in this set and $\norm{(x',y') - (x,y)} \le \delta$, then $x' < \frac{1}{4} \, 2^{-3N}$ or $\abs{y' - [b]_\ell \, x'} > 100 \cdot 2^{-\ell} \, x'$. In other words, the $\delta$-neighborhood of~\eqref{eq:modified_P3} is contained in the set in~\ref{it:P3}.
In the first case, $x' < \frac{1}{4} \, 2^{-3N}$ follows if $\delta < \frac{1}{8} \,2^{-3N}$, but since $\delta = \kappa^{-5} = 2^{-5N}$ this holds for large $\kappa$.
In the second case,
\begin{align*}
    \abs{y' - [b]_\ell \,x'}
    &\ge \abs{y - [b]_\ell \, x} - \delta - [b]_\ell \, \delta
    \ge 200 \cdot 2^{-\ell} \, x - 2\delta\,.
\end{align*}
This is greater than $100 \cdot 2^{-\ell} \, x$ provided that $2\delta \le 100 \cdot 2^{-\ell} \, x$, but this follows because $\delta = 2^{-5N}$ and $x \ge \frac{1}{8} \, 2^{-3N}$ (as we are in the negation of the first case).
In fact, by replacing $\delta$ with $2\delta$, the same argument shows that for all $(x,y)$ lying in the set~\eqref{eq:modified_P3}, we have $V_b(B_\delta(x,y)) = V_{b'}(B_\delta(x,y))$.
Note also that~\eqref{eq:modified_P3} shows that it is useless to query any points $(x,y)$ with $x < \frac{1}{8} \, 2^{-3N}$, so for the remainder of the proof we assume that the algorithm does not do so.

We now move to a stronger oracle model.
Namely, given a query point $(x, y) \in \R^2$, let $\ell$ be the largest integer such that $\abs{y-[b]_\ell \, x} \le 200\cdot 2^{-\ell} \, x$.
Then, the oracle outputs $\hat\xi \deq [b]_{\ell+1}$, i.e., the oracle reveals the first $\ell+1$ bits of $b$.
To see that this new oracle is indeed stronger, observe that we can simulate the previous oracle using the revealed bits $[b]_{\ell+1}$; namely, pick any bit string $b'$ which is consistent, in the sense that $[b']_{\ell+1} = [b]_{\ell+1}$. Then, by the choice of $\ell$, we have $\abs{y-[b]_{\ell+1}\,x} > 200 \cdot 2^{-(\ell+1)}\,x$, so that $V_b(B_\delta(x,y)) = V_{b'}(B_\delta(x,y))$, and hence we can output $V_b(B_\delta(x,y))$ given knowledge of $[b]_{\ell+1}$.
It therefore suffices to bound the mutual information $I(\hat \xi_{1:n}; b)$ where $\hat \xi_{1:n}$ denotes the output of the stronger oracle on a sequence of adaptive but deterministic queries $(x_1,y_1),\dotsc,(x_n,y_n)$.

We can then write
\begin{align}
    I(\hat\xi_{1:n}; b) &= \sum_{i=1}^n I(\hat\xi_i; b \mid \hat\xi_{1:i-1})\\
    &= \sum_{i=1}^n \{H(\hat\xi_i \mid \hat\xi_{1:i-1}) - H(\hat\xi_i, \mid \hat\xi_{1:i-1}, b)\}\\
    &\le \sum_{i=1}^n H(\hat\xi_i \mid \hat\xi_{1:i-1})\, ,\label{eq:ent_i}
\end{align}
where $H(\cdot\mid \cdot)$ denotes the conditional entropy. 
The first line follows from the chain rule for mutual information, the second line follows from definition of mutual information, and third line follows from non-negativity of conditional entropy.
Thus, we are done if we can show that $H(\hat\xi_i \mid \hat\xi_{1:i-1}) \le O(1)$, for all $i \le c\, |\mc S|$.

Conditionally on any particular realization of $\hat\xi_{1:i-1}$, let $\ell_0$ denote the number of bits of $b$ revealed thus far and let $[b_0]_{\ell_0}$ denote the revealed bits.
Clearly the bit string $b$ is uniformly distributed on the set $B_N'$ of bit strings $b'$ with $[b']_{\ell_0} = [b_0]_{\ell_0}$. 
Also, since we have assumed that the algorithm's queries are deterministic given the past history, the next query point $(x_i, y_i)$ is deterministic.
Then, the conditional probability that $\ell \ge \ell_0$ bits are revealed by the next query is
\begin{align*}
    &\Pr\{200 \cdot 2^{-\ell} \,x_i < \abs{y_i - [b]_\ell \, x_i} \le 200 \cdot 2^{-(\ell-1)} \, x_i \mid \hat \xi_{1:i-1}\}\\
    &\qquad \le \Pr\bigl\{\frac{y_i}{x_i} - 200 \cdot 2^{-(\ell-1)} \le [b]_\ell \le \frac{y_i}{x_i} + 200 \cdot 2^{-(\ell-1)} \bigm\vert \hat\xi_{1:i-1}\bigr\}\,.
\end{align*}
This is the probability that a uniformly chosen element of $B_N'$ belongs to an interval of length $\Theta(2^{-\ell})$.
Since there are $2^{N-\ell_0}$ elements of $B_N'$, and $\Theta(2^{N-\ell})$ of them belong to any fixed interval of length $\Theta(2^{-\ell})$, we conclude that the above probability is $O(2^{-(\ell-\ell_0)})$.

We then have
\begin{align}
    H(\hat\xi_i \mid \hat\xi_{1:i-1})
    &\le \E \sum_{\ell\ge \ell_0} (\ell - \ell_0) \, O(2^{-(\ell-\ell_0)})
    \le O(1)\,,
\end{align}
where the expectation is taken over $\ell_0$ (which depends on the realization of $\hat\xi_{1:i-1}$).
Substituting the above bound into~\eqref{eq:ent_i}, we conclude that $I(\xi_{1:n}; b) = O(n)$, which implies that $\mc S$ is indeed hard to identify via queries.
\end{proof}

\subsection{Construction of the distributions}

This section contains the proof of Proposition \ref{prop:geometric_construction}.

For integers $1 \le k \le N$, let $[b]_k$ be the number $0.00b_1b_2 \dots b_k$ in binary representation, and let $[b]_k \deq [b] \deq [b]_N$ for $k\ge N$. Define
\begin{align}\label{eq:def_phi}
    \phi_{k,b}(x,y) \deq  \bigl(|y - [b]_k\, x|-(2^{-k}\,x + 2^{-(3N-k)})\bigr)_+\,.
\end{align}
Here, the term $2^{-(3N-k)}$ essentially controls the thickness of the slab, and in particular, the slab becomes thicker for larger $k$; this ensures that the maximum of the $\phi_{k,b}$ will be dominated by small $k$ far away.
We also write $\phi_k \deq \phi_{k,b}$ when $b$ is clear from context.
For $x \ge 0$, the function $\phi_k$ essentially measures the distance to the set 
$$
\{(x,[b]_k\, x +\xi_k) \in \R^2 : x \ge 0,\; |\xi_k| \le 2^{-k}\,x + 2^{-(3N-k)}\}\,.
$$
Finally, we define the potential
\begin{align}\label{eq:def_potential}
    \tilde V_b(x,y) \deq  2^{7N} \max_{k=1, \dotsc, N} 2^{-k}\phi_k(x,y)\,.
\end{align}

\medskip{}

\begin{proof}[Proof of Proposition~\ref{prop:geometric_construction}]
We prove that the construction~\eqref{eq:def_potential} satisfies each of the four properties in turn.

\underline{\textbf{Proof of Property~\ref{it:P0}.}}
The convexity of $\tilde V_b$ follows because each $\phi_k$ is convex.
To check that $\tilde V_b$ is $\kappa^{O(1)}$-smooth on average, using the compositionality of the maximum (i.e., $\max(a,\max(b,c)) = \max(a,b,c)$) we see that that $\tilde V_b$ can be written as a maximum of affine functions, each of slope $\kappa^{O(1)}$; hence, $\tilde V_b$ is $\kappa^{O(1)}$-Lipschitz.
Differentiating under the integral,
\begin{align*}
    \nabla (\tilde V_b * \chi_\delta)(x,y)
    &= \iint_{B_\delta} \nabla \tilde V_b(x+u,y+v) \, \D u \, \D v
    = \iint \nabla \tilde V_b \one_{B_\delta(x,y)}\,,
\end{align*}
where the expression makes sense because $\tilde V_b$ is Lipschitz and hence differentiable a.e.\ by Rademacher's theorem, and the absolute continuity of $\tilde V_b$ ensures the validity of the fundamental theorem of calculus.
Then, by H\"older's inequality,
\begin{align*}
    \norm{\nabla (\tilde V_b * \chi_\delta)(x,y) - \nabla (\tilde V_b * \chi_\delta)(x',y')}
    &\le \bigl(\sup{\norm{\nabla \tilde V_b}}\bigr) \, \norm{\one_{B_\delta(x,y)} - \one_{B_\delta(x',y')}}_{L^1} \\
    &\le \kappa^{O(1)} \vol\bigl(B_\delta(x,y) \bigtriangleup B_\delta(x',y')\bigr)\,.
\end{align*}
By elementary considerations, the volume of the symmetric difference between the balls is bounded by $O(\kappa^{O(1)} \, \norm{(x,y) - (x',y')})$, and therefore $\nabla (\tilde V_b * \chi_\delta)$ is $\kappa^{O(1)}$-Lipschitz.

Finally, it is obvious that $\tilde V_b \ge 0$ and $\tilde V_b = 0$ at the origin.

\underline{\textbf{Proof of Property~\ref{it:P1}.}}
We only need to verify that any point $(x, y)$ which is $10^3 \delta$-close to $\tilde{\mc Z}_b$ satisfies $\tilde V_b(x, y) = 0$, as the second part of Property~\ref{it:P1} is automatically implied by Property~\ref{it:P2}. For such a point $(x, y)$, there exists $(x', y')$ such that
\begin{align*}
    x' \ge 0\,, \qquad |x'-x| \wedge |y'-y| \le 10^3 \delta\,,\qquad \text{and}\qquad |y'-[b]\, x'| \le 2^{-N} x'\,.
\end{align*}
This also implies $|y' - [b]_k\, x'| \le 2^{-k}\, x'$ for all $1 \le k \le N$, since $|[b]_k-[b]| \le 2^{-k}-2^{-N}$. Therefore, for all $1 \le k \le N$, $|y - [b]_k\, x| \le 2^{-k}\, (x + 10^3 \delta) + 2 \cdot 10^3 \delta \le 2^{-k} x + 2^{-(3N-k)},$ since $\delta = 2^{-5N}$.
By the definition~\eqref{eq:def_potential} of $\tilde V_b$ and the definition of $\phi_k$ in~\eqref{eq:def_phi}, it follows that $\tilde V_b(x,y) = 0$.

\underline{\textbf{Proof of Property \ref{it:P2}.}}
We just need to check that
\begin{align*}
2^{6N} \phi_N(x, y) \ge \kappa^4 \,\bigl(\dist((x, y), \tilde{\mc Z}_b) - 1\bigr)_+\,,
\end{align*}
or equivalently, $2^{2N} \phi_N(x, y) \ge (\dist((x, y), \tilde{\mc Z}_b) - 1)_+$. 
We first consider the case when $x \ge 0$, and we may assume that $(x, y) \not\in \tilde{\mc Z}_b$ as otherwise the claim is obvious. If $(x, y)$ has distance $\Delta$ to its closest point in $\tilde{\mc Z}_b$, then any $y'$ such that $(x, y') \in \tilde{\mc Z}_b$ must satisfy $|y-y'| \ge \Delta$. Applying this to $y' = [b]\,x \pm 2^{-N} \,x$, we obtain
\begin{align*}
    \dist\bigl((x,y), \tilde{\mc Z_b}\bigr)
    &\le \abs{y - [b] \, x + 2^{-N} x} \wedge \abs{y-[b]\, x - 2^{-N} \, x}
    = \abs{y - [b] \, x} - 2^{-N} \, x\,.
\end{align*}
In turn, it implies that $\phi_N(x, y) \ge (\dist((x, y), \tilde{\mc Z}_b) - 2^{-(3N-k)})_+ \ge (\dist((x, y), \tilde{\mc Z}_b) - 1)_+$.

If $x < 0$, then $\dist((x, y), \tilde{\mc Z}_b) \le \norm{(x,y)} \le \sqrt 2 \,\max(\abs x, \abs y)$.
Then, for $N$ large,
\begin{align*}
    2^{2N} \phi_N(x,y)
    &= 2^{2N} \,\bigl(\abs{y-[b]\,x} - 2^{-N} \, x - 2^{-(3N-k)}\bigr)_+ \\
    &= 2^{N-1/2} \, \bigl( 2^{N+1/2}\, \abs{y-[b] \, x} + \sqrt 2 \, \abs x - 2^{-(2N-k)+1/2} \bigr)_+ \\
    &\ge 2^{N-1/2} \, \Bigl( 2^{3/2} \max\bigl(0, \abs y - \frac{1}{2}\, \abs x\bigr) + \sqrt 2 \, \abs x - 1 \Bigr)_+ \\
    &\ge 2^{N-1/2} \, \bigl( \sqrt 2 \max(\abs x, \abs y) - 1 \bigr)_+
    \ge \bigl(\dist((x,y), \tilde{\mc Z}_b) - 1\bigr)_+\,.
\end{align*}
The first inequality follows because $\abs{y-[b]\,x} = \abs{\abs y - [b]\sgn(y)\,x} \ge \abs y - \frac{1}{2}\,\abs x + \frac{1}{2}\,\abs x - [b]\sgn(y)\,x \ge \abs y - \frac{1}{2}\,\abs x$ and because $N$ is sufficiently large.

\underline{\textbf{Proof of Property~\ref{it:P3}}}.
The last property follows from Proposition~\ref{prop:far-away-inequality} below, because if $b$, $b'$ agree on the first $\ell$ bits, then on the set in the statement of Property~\ref{it:P3},
\begin{align*}
    \tilde V_b
    &= 2^{7N} \max_{k=1,\dotsc,N} 2^{-k} \phi_{k,b}
    = 2^{7N} \max_{k=1,\dotsc,\ell} 2^{-k} \phi_{k,b}
    = 2^{7N} \max_{k=1,\dotsc,\ell} 2^{-k} \phi_{k,b'}
    = 2^{7N} \max_{k=1,\dotsc,N} 2^{-k} \phi_{k,b'}
    = \tilde V_{b'}\,.
\end{align*}
The second and fourth equalities invoke Proposition~\ref{prop:far-away-inequality}, and the third equality uses the fact that $\phi_{k,b}$ only depends on $b$ through $[b]_k$.
This completes the proof.
\end{proof}

\begin{proposition}[potentials agree if bits agree]
\label{prop:far-away-inequality}
Let $S_\ell(b)$ be the set 
$$S_\ell(b) \deq \bigl\{(x,y)\in\R^2 : x<\frac{1}{4} \, 2^{-3N}~\text{or}~|y-[b]_\ell\, x| \ge {100} \cdot 2^{-\ell} x
\bigr\}\,.$$
Then, for $x,y \in S_\ell(b)$,
$$ 
\max_{k=1,\dotsc, N} 2^{-k}\phi_k(x,y) = 
\max_{k=1,\dotsc,\ell} 2^{-k}\phi_k(x,y)\,.
$$
\end{proposition}

In turn, Proposition~\ref{prop:far-away-inequality} follows by induction from: 

\begin{proposition}[induction]
\label{prop:far-away-inequality-step}
If $(x,y) \in  S_\ell(b)$, and for some $k>\ell$ we have $\phi_k(x,y)>0$, then $\phi_k(x,y)\le 2\phi_{k-1}(x,y)$. 
\end{proposition}

\begin{proof} 
First, we may assume that $x > 0$. This is because if $x \le 0$,
\begin{align*}
    \phi_{k-1}(x,y) &\ge |y-[b]_k\,x|-|[b]_{k-1}-[b]_k|\, |x| - 2^{-(k-1)}\, x - 2^{-(3N-k+1)} \\
    &\ge |y-[b]_k\,x| + 2^{-k} x - 2^{-(k-1)} x - 2^{-(3N-k+1)} \\
    &= |y-[b]_k\,x| - 2^{-k} x - 2^{-(3N-k+1)} \\
    &\ge \phi_k(x, y)\,,
\end{align*}
    since we are assuming $\phi_k(x, y) > 0$.

Now, since $x > 0$, we start by estimating
\begin{align*}
    \phi_{k-1}(x,y) &\ge |y-[b]_k\,x|-|[b]_{k-1}-[b]_k|\, x - 2^{-(k-1)} x - 2^{-(3N-k+1)} \\
    &\ge |y-[b]_k\,x|- 3 \cdot 2^{-k} x - 2^{-(3N-k+1)} \\
    &= \phi_k(x, y) - 2\cdot 2^{-k} x + 2^{-(3N-k+1)}
\end{align*}
and
$$
\phi_k(x,y) = |y - [b]_k\, x|-(2^{-k}x + 2^{-(3N-k)})\,.
$$

First, suppose that $x \le \frac{1}{4} \, 2^{-3N}$. Then, $2^{-(3N-k+1)} \ge 2 \cdot 2^{-k} x$, so in fact $\phi_{k-1}(x, y) \ge \phi_k(x, y)$. 
Alternatively, if $x \ge \frac{1}{4} \, 2^{-3N}$ and $|y-[b]_\ell\, x| \ge {100} \cdot 2^{-\ell} x$, then
\begin{align*}
    2\phi_{k-1}(x,y) &\ge 2\,|y - [b]_\ell\, x| - 2\,|[b]_\ell - [b]_{k-1}|\, x - 4\cdot 2^{-k} x - 2^{-(3N-k)}\\
    &\ge 2\,|y - [b]_\ell\, x| - 6\cdot 2^{-\ell} x - 2^{-(3N-k)}\,,\\
    \phi_k(x,y) &\le |y - [b]_\ell\, x| + |[b]_\ell - [b]_k|\, x - 2^{-k} x - 2^{-(3N-k)}\\
    &\le |y - [b]_\ell\, x| + 2^{-\ell}x - 2^{-(3N-k)}\, .
\end{align*}
As a result, when $|y - [b]_\ell\, x| \ge 100 \cdot 2^{-\ell} x$, we see that $\phi_{k-1}(x,y) \ge \frac{1}{2}\,\phi_k(x,y)$. 
\end{proof}

\section{A lower bound for sampling from Gaussians via Wishart matrices}\label{sec:wishart}

We define $W \sim \Wishart(d)$ to mean $W = XX^\top$ where each entry of $X \in \BR^{d \times d}$ is $\mathcal{N}(0, \frac{1}{d})$.
We aim to prove the following two theorems, which together imply a query complexity lower bound for sampling from Gaussians.

\begin{theorem}[reducing inverse trace estimation to sampling]\label{thm:reduce_inv_trace_sampling}
    Let $\delta > 0$.
    There is a universal constant $c > 0$ (depending only on $\delta$) such that the following hold.
    Suppose that $d\ge c^{-1}$ and there exists a query algorithm such that, for any Gaussian target distribution $\pi \deq \NN(0, \Sigma)$ in $\R^d$ with $cd^{-2}\,I_d \preceq \Sigma^{-1} \preceq c^{-1}\, I_d$, outputs a sample from a distribution $\widehat \pi$ such that either $\norm{\widehat\pi-\pi}_{\rm TV} \le c$ or $\sqrt{cd^{-2}} \, W_2(\widehat\pi,\pi) \le c$, using $n$ queries to $\pi$.
    
    Then, given $W \sim \Wishart(d)$, there exists an algorithm which makes at most $c^{-1} n$ matrix-vector queries to $W$ and outputs an estimator $\trhat$ such that $\frac{1}{2} \tr(W^{-1}) \le \trhat \le  2\tr(W^{-1})$ with probability at least $1-\delta$.
\end{theorem}

\begin{theorem}[lower bound for inverse trace estimation]\label{thm:inv_trace_lower_bd}
    Let $W\sim\Wishart(d)$ for $d\ge 2$.
    For any $C > 0$, there exists $\delta > 0$ (depending only on $C$) such that any algorithm which makes $n$ matrix-vector queries to $W$ and outputs an estimator $\trhat$ such that $C^{-1} \tr(W^{-1}) \le \trhat \le C\tr(W^{-1})$ with probability at least $1-\delta$ must use $n \ge \Omega(d)$ queries.
\end{theorem}

\begin{remark}
    Suppose that we want to sample from a target distribution $\pi$ which is $\alpha$-strongly log-concave. It is straightforward to check that total variation guarantees are invariant under rescaling the target (replacing $\pi$ with $S_\# \pi$, where $S : \R^d\to\R^d$ is the scaling map $Sx \deq \zeta x$ for some $\zeta > 0$), whereas Wasserstein guarantees are not. Instead, the scale-invariant quantity is $\sqrt \alpha \, W_2$, which is what appears in Theorem~\ref{thm:reduce_inv_trace_sampling}.
\end{remark}

Consider the class of centered Gaussian distributions on $\R^d$ which are $\alpha$-strongly log-concave and $\beta$-log-smooth; let $\kappa \deq \beta/\alpha$ denote the condition number.
Let $\ms C_{\msf G, \msf d}(\kappa, d, \varepsilon)$ denote the query complexity of outputting a sample which is $\varepsilon$-close in the metric $\msf d$ to a target distribution in this class, where $\msf d$ is one of the scale-invariant distances $\msf d \in \{\text{TV},\, \sqrt \alpha \,W_2\}$. Then, Theorems~\ref{thm:reduce_inv_trace_sampling} and~\ref{thm:inv_trace_lower_bd} (with $C = 2$ and $\delta$, $c$ being universal constants) show that for $d \ge c^{-1}$,
\begin{align}\label{eq:preliminary_lower_bd}
    \ms C_{\msf G, \msf d}(c^{-2} d^2, d, c)
    &\ge \Omega(d)\,.
\end{align}

By embedding the construction into higher dimensions, we obtain the following corollary.

\begin{corollary}[query lower bound via Wishart matrices]\label{cor:wishart_lower_bd}
    For $\msf d \in \{\mathrm{TV},\, \sqrt \alpha \,W_2\}$, there is a universal constant $c > 0$ such that
    \begin{align*}
        \ms C_{\msf G, \msf d}(\kappa, d, c)
        &\ge \Omega\bigl(\sqrt \kappa \wedge d\bigr)\,.
    \end{align*}
\end{corollary}
\begin{proof}
    If $\kappa \ge c^{-2} d^2$, then~\eqref{eq:preliminary_lower_bd} yields
    \begin{align*}
        \ms C_{\msf G, \msf d}(\kappa, d, c)
        &\ge \Omega(d)
        \ge \Omega\bigl(\sqrt\kappa \wedge d\bigr)\,.
    \end{align*}
    Otherwise, if $\kappa \le c^{-2} d^2$, let $d_\star$ be the largest integer such that $\kappa \ge c^{-2} d_\star^2$. Then, by embedding the $d_\star$-dimensional construction into dimension $d$,
    \begin{align*}
        \ms C_{\msf G, \msf d}(\kappa, d, c)
        &\ge \ms C_{\msf G, \msf d}(\kappa, d_\star, c)
        \ge \Omega(d_\star)
        \ge \Omega\bigl(\sqrt \kappa \wedge d\bigr)\,,
    \end{align*}
    which concludes the proof.
\end{proof}

\subsection{Reducing inverse trace estimation to sampling}

In this section, we prove Theorem~\ref{thm:reduce_inv_trace_sampling}, which is based on the concentration of the squared norm of a Gaussian. 
We recall the following identity:

\begin{lemma}[concentration of the squared norm]\label{cor:conc_sq_norm}
    Let $Z \sim \NN(0,\Sigma)$. Then,
    \begin{align*}
        \var(\norm Z^2)
        &= 2\,\norm\Sigma_{\HS}^2\,.
    \end{align*}
\end{lemma}
\begin{proof}
    Note that since all quantities are rotationally invariant, we may assume without loss of generality that $\Sigma$ is diagonal.
    Then the equality claimed is just the variance of a non-homogenous chi-squared random variable.
\end{proof}

We now prove Theorem~\ref{thm:reduce_inv_trace_sampling}.

\medskip{}

\begin{proof}[Proof of Theorem~\ref{thm:reduce_inv_trace_sampling}]
    Let $W\sim \Wishart(d)$ and let $\Sigma \deq W^{-1}$. By Proposition~\ref{prop:wishart_smallest_eig}, there exists $c > 0$ (depending only on $\delta$) such that with probability at least $1-\delta/3$, it holds that
    \begin{align*}
        cd^{-2}\,I_d \preceq \Sigma^{-1} \preceq c^{-1}\,I_d\,.
    \end{align*}
    We work on the event $\mc E$ that this holds.

    \textbf{Case 1: total variation distance.} From Lemma~\ref{cor:conc_sq_norm} and Chebyshev's inequality, we deduce that if $Z_1,\dotsc,Z_m \simiid \NN(0,\Sigma)$ and $\trhat_\star \deq m^{-1} \sum_{i=1}^m \norm{Z_i}^2$,
    \begin{align*}
        \Pr\bigl\{\bigl\lvert\trhat_\star - \tr \Sigma\bigr\rvert \ge \frac{1}{2} \tr \Sigma\bigr\}
        &\le \frac{\var \trhat_\star}{{(\tr \Sigma)}^2/4}
        = \frac{8}{m} \cdot \frac{\tr(\Sigma^2)}{\tr(\Sigma)^2} \le \frac{8}{m}\,.
    \end{align*}
    Take $m \ge 48/\delta$ so that this probability is at most $\delta/3$.
    Conditionally on $W$, let $\widehat \pi_W$ denote the law of the sample $X$ of the algorithm when run on the target $\NN(0,\Sigma)$.
    By running the sampling algorithm $m$ times, we can obtain i.i.d.\ samples $X_1,\dotsc,X_m\simiid \widehat \pi_W$.
    Then, for $\trhat \deq m^{-1} \sum_{i=1}^m \norm{X_i}^2$,
    \begin{align*}
        &\Pr\bigl\{\bigl\lvert\trhat - \tr \Sigma\bigr\rvert \ge \frac{1}{2} \tr \Sigma\bigr\}
        \le \Pr(\mc E^\comp) +  \Pr\bigl\{\bigl\lvert\trhat - \tr \Sigma\bigr\rvert \ge \frac{1}{2}\tr\Sigma, \; \mc E\bigr\} \\
        &\qquad \le \frac{\delta}{3} + \E\bigl[\Pr\bigl\{\bigl\lvert\trhat_\star - \tr \Sigma\bigr\rvert \ge \frac{1}{2} \tr \Sigma \bigm\vert W\bigr\} \one_{\mc E}\bigr] + \E\bigl[\norm{\widehat \pi_W^{\otimes m}- {\NN(0,\Sigma)}^{\otimes m}}_{\rm TV} \one_{\mc E}\bigr] \\
        &\qquad \le \frac{\delta}{3} + \frac{\delta}{3} + cm\,.
    \end{align*}
    If we choose $c \le \delta/(3m)$, then $\trhat$ is an estimator of $\tr(W^{-1})$ with multiplicative error at most $2$ which succeeds with probability at least $1-\delta$. Note that both $c$ and $m$ depend only on $\delta$.
    
    \textbf{Case 2: Wasserstein distance.} Consider a coupling of $X$ and $Z$ such that, conditionally on $W$, we have $\E[\norm{X-Z}^2 \mid W] = \E[W_2^2(\widehat \pi_W, \NN(0,\Sigma)) \mid W]$. Let $(X_1,Z_1),\dotsc,(X_m,Z_m)$ be i.i.d.\ copies of this coupling.
    Also, let $\mc E'$ denote the event that $\lambda_{\min}(W^{-1}) \ge \bar c d^2$, where $\bar c$ is a constant depending only on $\delta$, chosen so that $\Pr(\mc E'^\comp) \le \delta/3$ using Proposition~\ref{prop:wishart_smallest_eig}.
    Then, conditionally on $W$ in the event $\mc E \cap \mc E'$,
    \begin{align*}
        \E\bigl[\abs{\trhat - \tr \Sigma} \bigm\vert W\bigr]
        &\le \E\bigl[\abs{\trhat - \trhat_\star} \bigm\vert W\bigr] + \E\bigl[\abs{\trhat_\star - \tr \Sigma} \bigm\vert W\bigr] \\
        &\le \E\bigl[\abs{\trhat - \trhat_\star} \bigm\vert W\bigr] + \frac{2\tr\Sigma}{\sqrt m}\,,
    \end{align*}
    where we used Lemma~\ref{cor:conc_sq_norm}. Using $\norm x^2 - \norm y^2 = \langle x-y, x+y\rangle$, for any $\lambda > 0$,
    \begin{align*}
        \E\bigl[\abs{\trhat - \trhat_\star} \bigm\vert W\bigr]
        &\le \E\bigl[\abs{\,\norm X^2 - \norm Z^2\,} \bigm\vert W\bigr]
        \le \E\bigl[\norm{X-Z}^2 \bigm\vert W\bigr] + 2\E\bigl[\abs{\langle X-Z, Z\rangle} \bigm\vert W\bigr] \\
        &\le (1+\lambda) \E\bigl[\norm{X-Z}^2 \bigm\vert W\bigr] + \frac{1}{\lambda} \E\bigl[\norm Z^2 \bigm\vert W\bigr] \\
        &\le (1+\lambda) \, c^3 d^2 + \frac{\tr \Sigma}{\lambda}
        \le (1+\lambda) \, \frac{c^3}{\bar c} \tr \Sigma + \frac{\tr \Sigma}{\lambda}\,.
    \end{align*}
    For the last line, recall that we are assuming $\E[W_2^2(\widehat \pi_W, \mc N(0,\Sigma)) \mid W] \le c^3 d^2$.
    If we take $\lambda = 18/\delta$, $m\ge {(36/\delta)}^2$, and if $c$ is sufficiently small (depending only on $\delta$), we obtain
    \begin{align*}
        \E\bigl[\abs{\trhat- \tr \Sigma} \bigm\vert W\bigr]
        &\le \frac{\delta \tr \Sigma}{6}\,.
    \end{align*}
    By Markov's inequality,
    \begin{align*}
        \Pr\bigl\{\bigl\lvert\trhat - \tr \Sigma\bigr\rvert \ge \frac{1}{2}\tr\Sigma\bigr\}
        &\le \Pr(\mc E^\comp) + \Pr(\mc E'^\comp) + \E\bigl[ \Pr\bigl\{\bigl\lvert\trhat - \tr \Sigma\bigr\rvert \ge \frac{1}{2}\tr\Sigma \bigm\vert W\bigr\} \one_{\mc E \cap \mc E'}\bigr] \\
        &\le \frac{\delta}{3} + \frac{\delta}{3} + \frac{\delta}{3}
        \le \delta\,.
    \end{align*}
    We conclude as before.
\end{proof}

\subsection{Lower bound for inverse trace estimation}

In this section, we prove Theorem~\ref{thm:inv_trace_lower_bd}. The idea is that due to the heavy tails of $\lambda_{\min}(W^{-1})$ implied by Proposition~\ref{prop:wishart_smallest_eig}, with some small probability $\delta$, $\tr(W^{-1})$ will be very large. An algorithm for inverse trace estimation which succeeds with probability at least $1-\delta$ must be able to detect this event, and we show that this requires making $\Omega(d)$ queries.

The key technical tools are the following propositions, due to~\cite{braverman2020regressionLB}.

\begin{proposition}[{\cite[Lemma 3.4]{braverman2020regressionLB}}] \label{prop:wishart_posterior} 
    Let $W \sim \Wishart(d).$ Then, for any sequence of $n < d$ (possibly adaptive) queries $v_1, \dotsc, v_n$ and responses $w_1 = W v_1, \dotsc, w_n = W v_n$, there exists an orthogonal matrix $V \in \BR^{d \times d}$ and matrices $Y_1 \in \BR^{n \times n}, Y_2 \in \BR^{(d-n) \times n}$ that only depend on $v_1, \dots, v_n, w_1, \dots, w_n$, such that $V W V^\top$ has the block form
    \begin{align*}
        VWV^\top = \begin{bmatrix} Y_1 Y_1^\top & Y_1 Y_2^\top \\ Y_2 Y_1^\top & Y_2 Y_2^\top + \widetilde{W} \end{bmatrix}\,.
    \end{align*}
    Here, conditionally on $v_1, \dotsc, v_n, w_1, \dotsc, w_n$, the matrix $\widetilde{W}$ has the $\Wishart(d-n)$ distribution.
\end{proposition}

\begin{proposition}[{\cite[Lemma 3.5]{braverman2020regressionLB}}] \label{prop:wishart_smaller_than_posterior} 
    For any matrices $Y_1 \in \BR^{n \times n}$, $Y_2 \in \BR^{(d-n) \times n}$, and any symmetric matrix $\widetilde{W} \in \BR^{(d-n) \times (d-n)}$, it holds that
    \begin{align*}
        \lambda_{\min}\Bigl(\begin{bmatrix} Y_1 Y_1^\top & Y_1 Y_2^\top \\ Y_2 Y_1^\top & Y_2 Y_2^\top + \widetilde{W} \end{bmatrix} \Bigr)
        \le \lambda_{\min}(\widetilde W)\,.
    \end{align*}
\end{proposition}

We are now ready to prove Theorem~\ref{thm:inv_trace_lower_bd}. Note that this result is very similar to that of~\cite{braverman2020regressionLB}, except that we work with the inverse trace rather than the minimum eigenvalue.

\medskip{}

\begin{proof}[Proof of Theorem~\ref{thm:inv_trace_lower_bd}]
    Let $\delta > 0$ be chosen later. We first argue that $\trhat$ must not be too large. Applying Proposition~\ref{prop:wishart_inverse_trace}, we conclude that there is a universal constant $C' > 0$ such that $\tr(W^{-1}) \le C'd^2$ with probability at least $1/2$. Hence,
    \begin{align*}
        \Pr\bigl\{\trhat \le CC'd^2\bigr\}
        &\ge \Pr\bigl\{\tr(W^{-1}) \le C'd^2~\text{and}~\trhat \le C\tr(W^{-1})\bigr\} \\
        &\ge \Pr\{\tr(W^{-1}) \le C'd^2\} - \Pr\{\trhat > C\tr(W^{-1})\}
        \ge \frac{1}{2} - \delta\,.
    \end{align*}
    Next, suppose for the sake of contradiction that $n \le d/2$. Let $\ms F_n$ denote the $\sigma$-algebra generated by the information available to the algorithm up to iteration $n$, that is, the queries $v_1,\dotsc,v_n$, the responses $w_1,\dotsc,w_n$, and any external randomness used by the algorithm (which is independent of $W$).
    Applying Propositions~\ref{prop:wishart_posterior} and~\ref{prop:wishart_smaller_than_posterior},
    \begin{align*}
        \Pr\bigl\{\trhat < C^{-1} \tr(W^{-1})\bigr\}
        &\ge \Pr\bigl\{\trhat \le CC' d^2~\text{and}~\lambda_{\max}(W^{-1}) > C^2 C' d^2\bigr\} \\
        &\ge \Pr\bigl\{\trhat \le CC' d^2~\text{and}~\lambda_{\max}(\widetilde W^{-1}) > C^2 C' \, d^2\bigr\} \\
        &= \E\bigl[ \one\{\trhat \le CC' d^2\} \, \Pr\{\lambda_{\max}(\widetilde W^{-1}) \ge C^2 C' d^2 \mid \ms F_n\}\bigr]\,.
    \end{align*}
    According to Proposition~\ref{prop:wishart_posterior}, conditionally on $\ms F_n$, $\widetilde W$ has the $\Wishart(d-n)$ distribution. By applying Proposition~\ref{prop:wishart_smallest_eig},
    \begin{align*}
        \Pr\{\lambda_{\max}(\widetilde W^{-1}) \ge C^2 C' d^2 \mid \ms F_n\}
        &\ge \Pr\{\lambda_{\max}(\widetilde W^{-1}) \ge 4C^2 C' \, {(d-n)}^2 \mid \ms F_n\} \\
        &= \Pr\Bigl\{\lambda_{\min}(\widetilde W) \le \frac{1}{4C^2 C' \, {(d-n)}^2} \Bigm\vert \ms F_n\Bigr\}
        \gtrsim \frac{1}{C\sqrt{C'}}\,.
    \end{align*}
    Therefore,
    \begin{align*}
        \Pr\bigl\{\trhat < C^{-1} \tr(W^{-1})\bigr\}
        &\gtrsim \Pr\bigl\{\trhat \le CC' d^2\bigr\} \, \frac{1}{C\sqrt{C'}}
        \ge \frac{1/2-\delta}{C\sqrt{C'}}\,,
    \end{align*}
    which is larger than $\delta$ provided that $\delta$ is chosen sufficiently small (depending only on $C$). This contradicts the success probability of the algorithm, and hence we deduce that $n\ge d/2$.
\end{proof}

\subsection{Useful facts about Wishart matrices}

We collect together useful facts about Wishart matrices which are used in the proofs.

\begin{proposition}[extreme singular values of a Gaussian matrix]\label{prop:wishart_smallest_eig}
    Let $W \sim \Wishart(d)$.
    For any $x \in [0, 1]$,
    \begin{align*}
        \Pr\bigl\{\lambda_{\min}(W) \le \frac{x}{d^2}\bigr\}
        \asymp \sqrt x\,.
    \end{align*}
    Also, there is a universal constant $C > 0$ such that
    \begin{align*}
        \Pr\{\lambda_{\max}(W) \ge C \, (1+t)\}
        &\le 2\exp(-dt)\,.
    \end{align*}
\end{proposition}
\begin{proof}
    See, e.g.,~\cite[Theorem 5.1]{edelman_thesis} and~\cite[Theorem 4.4.5]{vershynin_book}.
\end{proof}

\begin{proposition}[bound on the inverse trace] \label{prop:wishart_inverse_trace}
    Let $W \sim \Wishart(d)$.
    Then, for any $\delta > 0$, with probability at least $1-\delta$, it holds that $\tr(W^{-1}) \le C_\delta d^2$ where $C_\delta$ is a constant depending only on $\delta$.
\end{proposition}
\begin{proof}
    According to~\cite[Theorem 1.2]{sza1991condnumber}, there is a universal constant $C > 0$ such that for each $j=1,\dotsc,d$ and $\alpha \ge 0$,
    \begin{align*}
        \Pr\Bigl\{\frac{1}{\lambda_j(W)} \ge \frac{d^2}{\alpha^2 j^2}\Bigr\}
        &\le {(C\alpha)}^{j^2}\,.
    \end{align*}
    Let $\alpha < 1/C$ and let $E_\alpha \deq \{1/\lambda_j(W) \ge d^2/(\alpha^2 j^2)~\text{for some}~j=1,\dotsc,d\}$. By the union bound,
    \begin{align*}
        \Pr(E_\alpha)
        &\le \sum_{j=1}^d {(C\alpha)}^{j^2}
        \lesssim \frac{1}{\sqrt{\log(1/(C\alpha))}}\,.
    \end{align*}
    On the event $E_\alpha^\comp$,
    \begin{align*}
        \tr(W^{-1})
        &\le \sum_{j=1}^d \frac{d^2}{\alpha^2 j^2}
        = \frac{\uppi^2 d^2}{6\alpha^2}\,,
    \end{align*}
    which is the claimed result upon taking $\alpha$ sufficiently small.
\end{proof}

\begin{remark}
    The proof only shows that $\Pr\{\tr(W^{-1}) \ge \eta d^2\} \lesssim 1/\sqrt{\log \eta}$ for $\eta \gg 1$, which is not enough to conclude that $\E \tr(W^{-1})$ is finite.
    In fact, it holds that $\E \tr(W^{-1}) = \infty$, which can already be seen from Proposition~\ref{prop:wishart_smallest_eig}.
\end{remark}

\section{A lower bound for sampling from Gaussians via reduction to block Krylov}\label{sec:krylov}

In this section, we prove Theorem \ref{thm:high_d_informal_2}.
Our proof procedes in two parts: we first show a lower bound against the block Krylov method, and then a reduction showing that an arbitrary adaptive algorithm can be simulated via a block Krylov method.

\subsection{Preliminaries} \label{subsec:prelim_block_krylov}

We first record some important facts that we will use later on.
Throughout, let $K$ be an odd integer.
The following is a standard approximation-theoretic result:

\begin{proposition}[{\cite[Proposition 2.4, rephrased]{sachdeva_survey}}] \label{prop:chebyshev_optimal}
    Let $T_K$ be the degree-$K$ Chebyshev polynomial, and let $1 = \beta_1 > \cdots > \beta_{K+1} = -1$ be the set of real values $\beta$ such that $T_K(\beta) \in \{-1, 1\}$. Then, for any real degree-$K$ polynomial $p$ such that $|p(\beta_i)| \le 1$ for all $\beta_i$, we have $|p(x)| \le |T_K(x)| \le {(|x|+\sqrt{x^2-1})}^K$ for all $\abs x > 1$.
\end{proposition}
\noindent
Let $c_0 > 0$ be a constant to be chosen later. The above proposition immediately implies: 

\begin{corollary}[approximation error] \label{cor:chebyshev_scaled_optimal}
    Suppose that $K \le c_0 \sqrt{\kappa} \log d$. Then, there exist $\kappa = \lambda_1 > \cdots > \lambda_{K+2} = 1$ (that only depend on $K$ and $\kappa$) such that for any real degree-$K$ polynomial $P$, $\max_{1 \le i \le K+2}{\abs{\frac{1}{\lambda_i}-P(\lambda_i)}} \ge d^{-2c_0 - O(1/\sqrt\kappa)}/\kappa$.
\end{corollary}

\begin{proof}
    Set $\beta_1, \dots, \beta_{K+2}$ to be the solutions of $T_{K+1} \in \{-1, 1\}$, and for each $1 \le i \le K+2$, set $\lambda_i \deq \frac{(\kappa-1)}{2}\, (\beta_i+1)+1$; by construction, $\kappa = \lambda_1 > \cdots > \lambda_{K+2} = 1$.
    Given any polynomial $Q$ of degree at most $K+1$, note that if $|Q(\lambda_i)| \le 1$ for all $i$, then the polynomial $p$ given by $p(x)  \deq  Q(\frac{\kappa-1}{2}\, (x+1) + 1)$ satisfies $|p(\beta_i)| \le 1$ for all $i$. By Proposition~\ref{prop:chebyshev_optimal}, for $x_0 \deq -(1 + \frac{2}{\kappa-1})$,
    \begin{align*}
        |Q(0)| = |p(x_0)|
        &\le \bigl(\abs{x_0} + \sqrt{x_0^2 -1}\bigr)^{K+1}
        \le \Bigl(1 + \frac{2}{\sqrt{\kappa}} + O\bigl(\frac{1}{\kappa}\bigr)\Bigr)^{K+1} \\
        &< \exp\Bigl(\bigl(\frac{2}{\sqrt{\kappa}} + O\bigl(\frac{1}{\kappa}\bigr)\bigr) \, \bigl(c_0\sqrt\kappa \log d + 1\bigr)\Bigr)
        = d^{2c_0 + O(1/\sqrt\kappa)}\,.
    \end{align*}
    
    Next, for a degree-$K$ polynomial $P$, consider $Q(x) \deq d^{2c_0 + O(1/\sqrt\kappa)} \, (1 - x P(x))$. Note that $Q$ has degree $K+1$ and $|Q(0)| = d^{2c_0 + O(1/\sqrt\kappa)}$, which implies that $|Q(\lambda_i)| > 1$ for some $i$, which in turn shows that $|\frac{1}{\lambda_i} - P(\lambda_i)| \ge d^{-2c_0 - O(1/\sqrt\kappa)}/\kappa$.
\end{proof}

We also introduce random matrix ensembles that are used in the proof, together with basic facts and properties.

Interestingly, as in the previous section, Wishart matrices are also useful for understanding block Krylov algorithms, but for a completely different reason. This time, we will study inner products between random vectors, which is also captured by a Wishart matrix. We denote by $\Wishart(K, N)$ the law of the random matrix $XX^\top \in \R^{K\times K}$, where the entries of $X \in \BR^{K \times N}$ are i.i.d.\ \emph{standard} Gaussians. Note that this is a different convention from the previous section, in which each entry of $X$ was i.i.d.\ $\mc N(0, \frac{1}{d})$.

We also define the Gaussian orthogonal ensemble (GOE) of size $K$, denoted $\GOE(K)$. This is the law of a random symmetric matrix $G\in \BR^{K \times K}$ where each diagonal entry $G_{i, i}$ is distributed as $\mc N(0, 1)$, and each off-diagonal entry $G_{i, j} = G_{j, i}$ is distributed as $\mc N(0, \frac{1}{2})$.
Also, the entries $\{G_{i,j} : 1\le i \le K,\; j\le i\}$ are independent.

A long line of work (see, e.g.,~\cite{jiali2015betalaguerre, bubetal2016geometryrandomgraph, bubgan2018entropicclt, racric2019wishartgoe, brebrehua2021definetti, mik2022cltwishart}) shows that when $N \gg K^3$, the Wishart ensemble is well-approximated by a scaled and shifted GOE, a fact which we shall invoke in the sequel.

\begin{lemma}[equivalence of Wishart and GOE] \label{lem:wishart_goe_similar}
    Let $W \sim \Wishart(K, N)$ be drawn from the Wishart distribution, and let $W_0$ be drawn from the distribution of symmetric matrices where the diagonal and above-diagonal entries are mutually independent, each diagonal entry is drawn as $\mathcal{N}(N, 2N)$, and each above-diagonal entry is drawn as $\mathcal{N}(0, N)$. (Equivalently, we can write $W_0 = NI + \sqrt{2N}\, G$, where $G\sim \GOE(K)$.) Then,
    \begin{align*}
        \norm{\law(W) - \law(W_0)}_{\rm TV} \le O\Bigl(\frac{K^{3/2}}{N^{1/2}}\Bigr)\,.
    \end{align*}
\end{lemma}

\noindent
Finally, we also require the following basic linear algebraic fact:
\begin{proposition}[rotating the right singular vectors] \label{prop:rotation}
    Let $V, V' \in \BR^{K \times N}$ be such that $VV^\top = (V')(V')^\top$. Then, there exists an orthogonal matrix $U \in \BR^{N \times N}$ such that $VU = V'$.
\end{proposition}

\subsection{Lower bound against block Krylov algorithms}

We start with the following proposition, which will be useful in establishing the existence of matrices with different inverse traces but which generate similar power method iterates.

\begin{proposition}[polynomial approximation and duality] \label{prop:chebyshev_dual_lp}
    Suppose that $K \le c_0 \sqrt{\kappa} \log d$. Then, there exist $\kappa = \lambda_1 > \lambda_2 > \cdots > \lambda_{K+2}=1$ and non-negative real numbers $x_1, \dots, x_{K+2}; x_1', \dots, x_{K+2}'$, such that:
\begin{enumerate}
    \item For all $0 \le j \le K$, $\sum_{i=1}^{K+2} x_i \lambda_i^j = \sum_{i=1}^{K+2} x_i' \lambda_i^j$.
    \item $\sum_{i=1}^{K+2} x_i = \sum_{i=1}^{K+2} x_i' = d$.
    \item $\sum_{i=1}^{K+2} x_i/\lambda_i - \sum_{i=1}^{K+2} x_i'/\lambda_i \ge 2d^{1-2c_0-O(1/\sqrt\kappa)}/\kappa$.
\end{enumerate}
\end{proposition}

\begin{proof}
    If we fix the values of the $\lambda_i$ to be the choices in Corollary~\ref{cor:chebyshev_scaled_optimal}, this becomes a linear program in the variables $\{x_i\}_{i = 1}^{K+2}, \{x_i'\}_{i = 1}^{K+2}$. By writing $\textbf{x} = (x_1, \dots, x_{K+2}, x_1', \dots, x_{K+2}'),$ our goal is to maximize $\textbf{c}^\top \textbf{x}$ over $\textbf{x} \ge 0$ subject to $A \textbf{x} = \textbf{b}.$ In our case, we set
\[\textbf{c}  \deq  \begin{bmatrix} \lambda_1^{-1} \\ \vdots \\ \lambda_{K+2}^{-1} \\[0.25em] -\lambda_1^{-1} \\ \vdots \\ -\lambda_{K+2}^{-1} \end{bmatrix}\,, \qquad A \deq \begin{bmatrix} 1 & \cdots & 1 & 1 & \cdots & 1 \\ 1 & \cdots & 1 & -1 & \cdots & -1 \\ \lambda_1 & \cdots & \lambda_{K+2} & -\lambda_1 & \cdots & -\lambda_{K+2} \\ \vdots & \ddots & \vdots & \vdots & \ddots & \vdots \\ \lambda_1^K & \cdots & \lambda_{K+2}^K & -\lambda_1^K & \cdots & -\lambda_{K+2}^K \end{bmatrix}\,,\qquad \textbf{b} = \begin{bmatrix} 2d \\ 0 \\ \vdots \\ 0 \end{bmatrix}\,.\]

We can consider the dual linear program, and by strong duality this maximization is equivalent to minimizing $\textbf{b}^\top \textbf{y}$ over $\textbf{y}$ such that $A^\top \textbf{y} \ge \textbf{c}$. By writing $\textbf{y} = (z, y_0, y_1, \dotsc, y_K)$, this means we wish to minimize $2dz$ subject to $z + (y_0 + y_1 \lambda_i + \cdots + y_K \lambda_i^K) \ge \frac{1}{\lambda_i}$ and $z - (y_0 + y_1 \lambda_i + \cdots + y_K \lambda_i^K) \ge -\frac{1}{\lambda_i}$ for all $1 \le i \le K+2$. Equivalently, we wish to minimize $2dz$ subject to the existence of a polynomial $P$ of degree at most $K$ (with coefficients $y_0, \dotsc, y_K$) such that $z \ge \abs{\frac{1}{\lambda_i}-P(\lambda_i)}$ for all $i \le K+2$.

The minimum for the dual linear program (and thus the maximum for the primal linear program), is $2d \inf_{P\in\mc P_K} \max_{1 \le i \le K+2}{\abs{\frac{1}{\lambda_i}-P(\lambda_i)}}$, where $\mc P_K$ is the set of polynomials of degree at most $K$ with real coefficients. By Corollary~\ref{cor:chebyshev_scaled_optimal}, this quantity is at least $2d^{1-2c_0-O(1/\sqrt\kappa)}/\kappa$.
\end{proof}

We note that a slightly strengthened version of Proposition \ref{prop:chebyshev_dual_lp} holds. Let $0 < c_1 < 1$.

\begin{corollary}[existence of good solutions] \label{cor:chebyshev_dual_lp}
    Proposition~\ref{prop:chebyshev_dual_lp} holds, where we also ensure that each $x_i$ and $x_i'$ is at least $\frac{d}{2\,(K+2)}$ and $\frac{|x_i-x_i'|}{x_i} \le \frac{2c_1}{1-c_1}$, though the right-hand side of the third condition becomes $c_1 d^{1-2c_0-O(1/\sqrt\kappa)}/\kappa$.
\end{corollary}

\begin{proof}
    First, replace every $x_i$ with $\frac{1}{2}\, (x_i + \frac{d}{K+2})$ and $x_i'$ with $\frac{1}{2}\, (x_i' + \frac{d}{K+2})$. Then, we have that the replaced $x_i, x_i'$ are at least $\frac{d}{2\,(K+2)}$, and the remaining statements in Proposition~\ref{prop:chebyshev_dual_lp} hold, except the third which has the right-hand side replaced with $d^{1-2c_0 - O(1/\sqrt\kappa)}/\kappa$.
    
    Next, replace every $x_i$ with $\tilde x_i \deq \frac{1 + c_1}{2} \, x_i + \frac{1 -c_1}{2} \, x_i'$, and every $x_i'$ with $\tilde x_i' \deq \frac{1 + c_1}{2} \, x_i' + \frac{1 - c_1}{2} \, x_i$. We still have that every $\tilde x_i, \tilde x_i'$ is at least $\frac{d}{2\,(K+2)}$, the first two conditions still hold, and the right-hand side of third condition is now $c_1 d^{1-2c_0 - O(1/\sqrt\kappa)}/\kappa$. Finally, note that $|\tilde x_i-\tilde x_i'| \le c_1 \, \abs{x_i - x_i'}$, whereas $\tilde x_i \ge \frac{1-c_1}{2} \, (x_i + x_i')$.
    This implies that $\frac{\abs{\tilde x_i - \tilde x_i'}}{\tilde x_i} \le \frac{2c_1}{1-c_1}$.
\end{proof}

We now have the necessary tools to prove our lower bound against block Krylov algorithms.
Before doing so, we establish that there exist diagonal matrices $D, D'$ which have substantially different inverse traces, but block Krylov algorithms cannot distinguish between them. To prove our actual lower bound, we show the same claim holds even if $D, D'$ are randomly rotated, and the inverse trace difference is enough for a single sample to distinguish between them.

\begin{lemma}[construction of diagonal matrices] \label{lem:moment_matching}
    Suppose that $K \le c_0 \sqrt{\kappa} \log d$ and $K \le O(d)$.
    Then, there exist diagonal matrices $D, D' \in \BR^{d \times d}$ with all diagonal entries between $1$ and $\kappa$ with the following properties.
\begin{enumerate}
    \item $|\Tr(D^{-1})-\Tr(D'^{-1})| \ge c_1 d^{1-2c_0 - O(1/\sqrt\kappa)}/\kappa - 2\,(K+2)$.
    \item Consider sampling $K$ $d$-dimensional random vectors $v^{(1)},\dotsc,v^{(K)} \overset{i.i.d.}{\sim} \mathcal{N}(0, I_d)$. Then, the distributions of $\{\langle v^{(k)}, D^j\, v^{(\ell)}\rangle\}_{j \le K+2;\; k, \ell \le K}$ and $\{\langle v^{(k)}, D'^j\, v^{(\ell)}\rangle\}_{j \le K+2;\, k, \ell \le K}$ differ in total variation distance by at most $O(c_1 K^3 + K^3/d^{1/2})$.
\end{enumerate}
\end{lemma}
\begin{proof}
Choose $\{x_i\}_{i=1}^{K+2}, \{x_i'\}_{i=1}^{K+2}$, and $\{\lambda_i\}_{i=1}^{K=2}$ satisfying Corollary \ref{cor:chebyshev_dual_lp}.
Define integers $\{N_i\}_{i=1}^{K+2}$ such that each $N_i$ is either $\lfloor x_i \rfloor$ or $\lceil x_i \rceil$ and $\sum_{i=1}^{K+2} N_i = d$; define $\{N_i'\}_{i=1}^{K+2}$ similarly in terms of $\{x_i'\}_{i=1}^{K+2}$. We let $D, D'$ be diagonal matrices such that for all $i$, $D$ has $N_i$ diagonal entries equal $\lambda_i$, and $D'$ has $N_i'$ diagonal entries equal to $\lambda_i$. Now, let $v^{(1)}, \dots, v^{(K)} \in \BR^d$ be $K$ random vectors drawn i.i.d.\ from $\mathcal{N}(0, I_d)$ (and define $v^{(1) \prime}, \dots, v^{(K) \prime}$ similarly). For $1 \le i \le K+2$, we define $v^{(k, i)}$ to be the projection of $v^{(k)}$ onto the dimensions corresponding to the diagonal entry $\lambda_i$ for $D$. Note that $\{v^{(k, i)}\}_{i\le K+2, \; k\le K}$ are independent, and $v^{(k,i)} \sim \mathcal{N}(0, I_{N_i})$. Likewise, define $\{v^{(k, i) \prime}\}_{i\le K+2,\;k\le K}$ accordingly in terms of $D'$.

Note that $\Tr(D^{-1})-\Tr(D'^{-1}) = \sum_{i=1}^{K+2} N_i/\lambda_i - \sum_{i=1}^{K+2} N_i'/\lambda_i$. Since $|N_i-x_i|, |N_i'-x_i'| \le 1$, and since each $\lambda_i \ge 1$, it implies
\begin{align*}
    \tr(D^{-1}) - \tr(D'^{-1})
    &\ge \frac{c_1\, d^{1-2c_0 -O(1/\sqrt\kappa)}}{\kappa} - 2\,(K+2)\,.
\end{align*}

Next, we let $W^{(i)}$ represent the $K \times K$ matrix with entries $W^{(i)}_{k, \ell} = \langle v^{(k, i)}, v^{(\ell, i)}\rangle$ and define $W^{(i) \prime}$ similarly.
Note that the matrices $W^{(i)}, W^{(i) \prime}$ over all $i$ are independent. In addition, $W^{(i)}$ has the $\Wishart(K, N_i)$ distribution, and $W^{(i) \prime}$ has the $\Wishart(K, N_i')$ distribution. In addition, for any $k, \ell \le K$ and $j \le T$, we have that $\langle v^{(k)}, D^j\, v^{(\ell)}\rangle = \sum_{i=1}^{K+2} \lambda_i^j W^{(i)}_{k, \ell}$.

Now, we attempt to design a coupling between the matrices $\{W^{(i)}\}_{i=1}^{K+2}$ and $\{W^{(i) \prime}\}_{i=1}^{K+2}$ such that $W^{(i)}-W^{(i) \prime} = (x_i-x_i')\, I_K$ for all $i \le K+2$, with high probability.
Note that this implies our claim, due to Corollary~\ref{cor:chebyshev_dual_lp}.
To design this coupling, first note that by Lemma \ref{lem:wishart_goe_similar}, if we draw $Z^{(i)} \sim N_i \, I_K + \sqrt{2 N_i} \GOE(K)$, then $\norm{\law(W^{(i)}) - \law(Z^{(i)})}_{\rm TV} \le O(K^{3/2}/N_i^{1/2})$, and a similar statement holds if we define $Z^{(i)\prime}$ and compare its law to that of $W^{(i)\prime}$.

Note that the entries of $Z^{(i)}$ and $Z^{(i) \prime}$ are independent (apart from the requirement of symmetry), so we will attempt a coupling between the entries $Z^{(i)}_{k, \ell}$ and $Z^{(i) \prime}_{k, \ell}$. For $k < \ell$, since $Z^{(i)}_{k, \ell} \sim \mathcal{N}(0, N_i)$ and $Z^{(i) \prime}_{k, \ell} \sim \mathcal{N}(0, N_i')$, the total variation distance between their distributions is bounded up to a constant, using Corollary~\ref{cor:chebyshev_dual_lp}, by
\begin{align*}
    \bigl\lvert \frac{N_i'}{N_i} - 1 \bigr\rvert
    &\le \bigl\lvert \frac{x_i'}{x_i} - 1 \bigr\rvert + \bigl\lvert \frac{N_i' - x_i'}{N_i} \bigr\rvert + \bigl\lvert \frac{x_i' \, (N_i - x_i)}{N_i x_i} \bigr\rvert
    \le O\bigl(c_1 + \frac{K}{d}\bigr)
\end{align*}
under our assumptions.
Therefore, we can couple $Z^{(i)}_{k, \ell}$ and $Z^{(i)'}_{k, \ell}$ such that they fail to coincide with this probability. For $k = \ell$, we have $Z^{(i)}_{k, k} \sim \mathcal{N}(N_i, 2N_i)$ and $Z^{(i) \prime}_{k, k} + x_i - x_i' \sim \mathcal{N}(N_i'+x_i - x_i', 2N_i')$.
The total variation distance between their distributions is bounded by a constant times
\begin{align*}
    \bigl\lvert \frac{N_i'}{N_i} - 1\bigr\rvert + \frac{\abs{N_i' - x_i' + x_i - N_i}}{\sqrt{N_i}}
    \le O\bigl(c_1 + \frac{K^{1/2}}{d^{1/2}}\bigr)\,.
\end{align*}
Therefore, we can couple the two random variables together so that $Z_{k,k}^{(i)} = Z_{k,k}^{(i)\prime} + x_i - x_i'$ fails with the above probability.

By a union bound, the coupling $Z^{(i)} = Z^{(i)\prime} + (x_i - x_i') \, I_K$ for all $i$ fails with probability at most
\begin{align*}
    O\Bigl( K^3\,\bigl(c_1 + \frac{K}{d}\bigr) + K^2 \, \bigl(c_1 + \frac{K^{1/2}}{d^{1/2}}\bigr)\Bigr)
    = O\bigl(c_1 K^3 + \frac{K^{5/2}}{d^{1/2}}\bigr)\,.
\end{align*}
We dropped the $c_1 K^4/d$ term because of our assumption $K \le O(d)$.
Combining this with comparison between the Wishart and GOE ensembles and another union bound, we obtain the result.
\end{proof}

Finally, we are able to prove our main lower bound against block Krylov algorithms.

\begin{lemma}[lower bound against block Krylov algorithms]\label{lem:block_krylov}
    Let $\kappa, K, D, D'$ be as in Lemma \ref{lem:moment_matching}. Then, let $U$ be a uniformly random orthogonal matrix in $\BR^{d \times d}$, and let $\Lambda = U^\top D U$ and $\Lambda' = U^\top D' U$. Let $v^{(1)}, \dots, v^{(K)} \overset{i.i.d.}{\sim} \mathcal{N}(0, I_d)$.
    Then, for any $\delta > 0$, provided $K \le O_\delta(\sqrt\kappa \log d)$ and $\kappa \le d^{1/5-\delta}$, the distributions of $\{\Lambda^j v^{(k)}\}_{j \le (K+2)/2,\; k \le K}$ and $\{\Lambda'^j v^{(k)}\}_{j \le (K+2)/2;\; k \le K}$ differ in total variation distance by at most $o(1)$.
    On the other hand, drawing a sample either from $\mathcal{N}(0, \Lambda^{-1})$ or $\mathcal{N}(0, \Lambda'^{-1})$ can, with probability $1-o(1)$, distinguish between the two cases.
\end{lemma}
\begin{proof}
    The following calculations are contingent on the values of the various parameters that we will choose at the end of the proof.
    From Lemma~\ref{lem:moment_matching}, there is a coupling such that the tuples $\{\langle v^{(k)}, D^j \, v^{(\ell)}\rangle\}_{j\le K+2,\;k,\ell\le K}$ and $\{\langle v^{(k)\prime}, D'^j\,v^{(\ell)\prime}\rangle\}_{j\le K+2,\; k,\ell \le K}$ are equal with high probability. In particular, it holds that $\langle D^i \, v^{(k)}, D^j \, v^{(\ell)}\rangle = \langle D'^i\,v^{(k)\prime}, D'^j\,v^{(\ell)\prime}\rangle$ for all $i,j\le (K+2)/2$ and $k\le K$ with high probability. By Proposition~\ref{prop:rotation}, there is a unitary matrix $U_0$ such that $D'^j \, v^{(k)\prime} = U_0 D^j\, v^{(k)}$ for all $j\le (K+2)/2$ and all $k\le K$ with high probability.
    Note then that the tuples $\{U^\T D^j U\, U^\T v^{(k)}\}_{j\le (K+2)/2, \; k\le K}$ and $\{U^\T U_0^\T D'^j U_0 U\, U^\T U_0^\T v^{(k)\prime}\}_{j\le (K+2)/2, \; k \le K}$ are equal with high probability, and this is a coupling which witnesses the fact that the distributions of $\{\Lambda^j v^{(k)}\}_{j\le (K+2)/2, \; k\le K}$ and $\{\Lambda'^j v^{(k)}\}_{j\le (K+2)/2, \; k\le K}$ are at most $O(c_1 K^3 + K^3/d^{1/2})$ apart in total variation distance.
    
    Finally, we note that from a single sample it is easy to distinguish between $\mathcal{N}(0, \Lambda^{-1})$ and $\mathcal{N}(0, \Lambda'^{-1})$. This is because if $X \sim \mathcal{N}(0, \Lambda^{-1}),$ then $\BE[\|X\|^2] = \Tr(\Lambda^{-1}) = \Tr(D^{-1}) = \sum_{i=1}^{K+2} N_i/\lambda_i$, but one checks that $\var(\|X\|^2) = O(\sum_{i=1}^{K+2} N_i/\lambda_i^2) \le O(d)$. Likewise, if $X' \sim \mathcal{N}(0, \Lambda'^{-1})$, then we have $\BE[\|X'\|^2] = \sum_{i=1}^{K+2} N_i'/\lambda_i$ but $\var(\|X'\|^2) = O(d)$. So, the difference in their expectations at least $c_1 d^{1-2c_0 - O(1/\sqrt\kappa)}/\kappa - 2\,(K+2)$, whereas the standard deviations are bounded by $O(d^{1/2})$.

    To finish the proof, we must choose the values of $c_0$ and $c_1$.
    We require the following conditions:
    \begin{enumerate}
        \item $c_1 K^3 = o(1)$.
        \item $K^3/d^{1/2} = o(1)$.
        \item $d^{1/2} = o(c_1 d^{1-2c_0 - O(1/\sqrt \kappa)}/\kappa - 2\,(K+2))$.
    \end{enumerate}
    For the second condition, we can assume $\kappa \le d^{1/3}/\log^4(d)$.
    To satisfy the first condition, we can set $c_1 = 1/(\kappa^{3/2} \log^4(d))$.
    Finally, if $\kappa$ is sufficiently large and if $c_0$ is chosen depending on $\delta$, then the third condition requires $\sqrt\kappa \log d + d^{1/2} = o(d^{1-\delta}/\kappa^{5/2})$, and it suffices for $\kappa \le d^{1/5-\delta}$.
\end{proof}

\begin{remark}
    We did not attempt to optimize the exponent in the condition $\kappa \le d^{1/5-\delta}$. Indeed, by using the chain rule for the KL divergence rather than a union bound in the proof of Lemma~\ref{lem:moment_matching}, we believe that the total variation bound can be improved to $O(c_1 K^{3/2} + K^{5/2}/d^{1/2})$, and a back-of-the-envelope calculation suggests that this could improve the condition to $\kappa \le d^{2/7-\delta}$. Nevertheless, this falls short of capturing the full regime $\sqrt\kappa \log d \le O(d)$, and we leave this as an open question.
\end{remark}

\subsection{Reduction to block Krylov algorithms} \label{scn:reduction-to-block-krylov}

In this section, we show that in order to prove a lower bound for sampling from Gaussians against any query algorithm, it suffices to prove a lower bound against block Krylov algorithms.

\subsubsection{Setup}\label{scn:setup}

Let $\Lambda = U^\top D U$, where $D$ is a (possibly random) diagonal matrix, $U$ is a Haar-random orthogonal matrix, and $U$ and $D$ are independent. We consider the following model, which is a strengthening of the matrix-vector product model:

\begin{definition}[extended oracle model]\label{def:extended-oracle-model}
Given $K \in \mathbb{N}$, for all  $k \in [K]$, the algorithm chooses a new query point $v_k$, and receives the information $\{\Lambda^i v_j\}_{(i, j) \in H_k},$ where $H_k := \{(i, j): i+j \le k+1, i \ge 0, 1 \le j \le k\}$ is a set of ordered pairs of nonnegative integers. We use the following notation $\{\Lambda^i v_j\}_{S}$ for any set $S$ to denote $\{\Lambda^i v_j\}_{(i, j) \in S}$.  
\end{definition}

This is clearly a stronger oracle model than before, so a lower bound against algorithms in the extended oracle model implies a lower bound against algorithms in the original matrix-vector model.

\begin{definition}[adaptive deterministic algorithm]
\label{def:adaptive-deterministic-algorithm}
An \textit{adaptive deterministic algorithm} $\calA$ that makes $K$ extended oracle queries (see Definition~\ref{def:extended-oracle-model}) is given by a deterministic collection of functions $v_1,v_2(\cdot),\dotsc,v_{K}(\cdot)$, where $v_1$ is constant and each $v_k (\cdot)$ is a function of $\frac{k\,(k+1)}{2}-1$ inputs.
This corresponds to a sequence of queries where the $k$-th query $v_k(\{\Lambda^i v_j\}_{H_{k-1}})$ is chosen adaptively based on the information available to the algorithm at the start of iteration $k$.
(Note that $v_1$ has no inputs.)
When the choice of the inputs is clear from context, we may simply write $v_k = v_k(\{\Lambda^i v_j\}_{H_{k-1}})$.
\end{definition}

In the extended oracle model, the next lemma shows that we can assume that each $v_k$ is a unit vector orthogonal to its inputs.

\begin{lemma}[extended oracle and orthogonal queries]\label{lem:extended-oracle-queries-are-orthogonal}
For $k \in [2, K]$, let $v_k$ be as stated in Definition~\ref{def:adaptive-deterministic-algorithm} and let $\{\Lambda^{i} v_j\}_{H_{k-1}}$ be as stated in Definition~\ref{def:extended-oracle-model}. Then, without loss of generality, we may assume that $v_k$ is orthogonal to the subspace spanned by the vectors in $\{\Lambda^{i} v_j\}_{H_{k-1}}$. 
\end{lemma}
\begin{proof}
Assume for sake of contradiction that this were not the case. Then, we can decompose $v_k = \sum_{(i, j) \in H_{k-1}} c_{i,j} \Lambda^i v_j + c^\perp v_k^\perp$ where $v_k^\perp$ is a unit vector orthogonal to $\{\Lambda^i v_j\}_{H_{k-1}}$ and each $c_{i, j}$ and $c^\perp$ is a scalar.
At the end of iteration $k$, the new information obtained by the algorithm is $\{\Lambda^i v_j\}_{i+j=k+1, j \le k}$.
For all $(i,j) \ne (1,k)$, the new information does not depend on $v_k$.
Also, $\Lambda v_k = \sum_{(i, j) \in H_{k-1}} c_{i,j} \Lambda^{i+1} v_j + c^\perp \Lambda v_k^\perp$, where each $\Lambda^{i+1} v_j$ is information obtained by the algorithm at the end of iteration $k+1$ regardless (due to our extended query model).
Since $(i+1, j) \in H_k$ if $(i, j) \in H_{k-1}$, and since $(1, k) \in H_k$, this expression shows that the algorithm would receive the same amount of information (or more, if $c^\perp = 0$) if it queries $v_k^\perp$ instead of $v_k$.
Applying this reasoning inductively proves the claim.
\end{proof}

We compare to a \emph{block Krylov algorithm}, which makes i.i.d.\ standard Gaussian queries $z_1,\dotsc,z_K$ and then receives $\{\Lambda^i z_j\}$ for all $i, j \le K$. Recall that a block Krylov algorithm does not make \emph{adaptive} queries, so it is easier to prove lower bounds against block Krylov algorithms. Our goal is to now show that block Krylov algorithms can simulate an adaptive deterministic algorithm.

\subsubsection{Conditioning lemma}\label{scn:conditioning}

We start by proving a general conditioning lemma which will be invoked repeatedly in the reduction to block Krylov algorithms.
This lemma roughly shows that if the adaptive algorithm knows $\{\Lambda^i v_j\}_{H_k},$ the posterior distribution of $\Lambda$ given $\{\Lambda^i v_j\}_{H_k}$ is indeed rotationally symmetric on the orthogonal complement $\{\Lambda^i v_j\}_{H_k}$.

We will use the notation $\eqdist$ to denote that two random variables are equal in probability distribution (possibly conditioned on other information).

\begin{lemma}[conditioning lemma, preliminary version]\label{lem:random_rotation_fixed_subspace}
    Let $U$ be a Haar-random orthogonal matrix, and $\Lambda = U^\top D U$, where $D$ is a (possibly random) positive diagonal matrix. Suppose that $\calA$ is an adaptive deterministic algorithm that generates extended oracle queries $v_1, \dots, v_{K}$, and after the $k$-th query knows $\Lambda^i v_j$ for all $(i, j) \in H_k$.
    For any integer $m \ge 1$, let $k$ be the integer such that $\frac{k(k+1)}{2} \le m < \frac{(k+1)(k+2)}{2},$ i.e., $m$ is at least the $k$-th triangular number but less than the $(k+1)$-th triangular number.
    Consider the order of vectors $v_1, \Lambda v_1, v_2, \Lambda^2 v_1, \Lambda v_2, v_3, \Lambda^3 v_1, \dots$ (this enumerates $\Lambda^i v_j$ in order of $i+j$, breaking ties with smaller values of $j$ first).
    Let $W_m$ be the set of first $m$ of these vectors and $X_k$ be the set $\{v_1, \dots, v_k\}$. Let $V$ be a Haar-random orthogonal matrix fixing $W_m$ and acting on the orthogonal complement $W_m^\perp$. Then, $(X_k, U) \eqdist (X_k, U V)$.
\end{lemma}

Before proving this lemma, we note that since the algorithm is deterministic and $D$ is fixed, $W_m$ and $X_k$ are deterministic functions of $\Lambda$, and thus of $U$. Hence, we can write $v_k(U'), W_m(U'), X_k(U')$ to be the $v_k, W_m, X_k$ that would have been generated if we started with $\Lambda' = (U')^\top D U'$. (If no argument is given, $v_k, W_m, X_k$ are assumed to mean $v_k(U), W_m(U), X_k(U)$, respectively.) We note the following proposition.

\begin{proposition}[fixing the first $m$ queries and responses]\label{prop:invariant}
    Suppose that $V$ is any orthogonal matrix fixing $W_m(U)$. Then, $W_m(U) = W_m(UV)$.
\end{proposition}
\begin{proof}
    We prove $W_{m'}(U) = W_{m'}(UV)$ for all $m' \le m$. The base case of $k = 1$ is trivial, since $v_1$ is fixed. We now prove the induction step for $m'$.
    
    If $m' \le m$ is a triangular number, $m' = \frac{k(k+1)}{2}$, then the $m'$-th vector in $W_m$ is $v_k$. But note that $v_k(U)$ is a deterministic function of $W_{m'-1}(U)$, and $v_k(UV)$ is the same deterministic function of $W_{m'-1}(UV)$. Hence, if the induction hypothesis holds for $m'-1$, it also holds for $m$.

    If $m' \le m$ is not a triangular number, then the $m'$-th number in $W_m(U)$ is $\Lambda^i v_j$ for some $i \ge 1$. Likewise, the $m'$-th number in $W_m(UV)$ is $V^\top \Lambda^i V v_j(UV)$. Since $i \ge 1$, we know that $v_j(U) = v_j(UV)$, by the induction hypothesis on $\frac{j(j+1)}{2} < m'$. But, we know that $V$ fixes $W_m$, which means it fixes $v_j$ and $\Lambda^i v_j$. Thus, $V^\top \Lambda^i V v_j(UV) = V^\top \Lambda^i V v_j = \Lambda^i v_j$.
\end{proof}

We are now ready to prove Lemma \ref{lem:random_rotation_fixed_subspace}.

\medskip{}

\begin{proof}[Proof of Lemma~\ref{lem:random_rotation_fixed_subspace}]
    We prove this by induction on $m$. For the base case $m = 1$, $U$ is a random matrix and $V$ is a random matrix that fixes $v_1$. Note that $v_1$ is chosen independently of $\Lambda$ (and thus of $U$), so $U$ and $V$ are independent. Even for any fixed $V$, the distribution $U V$ is a uniformly random orthogonal matrix, so overall $U \eqdist U V$. Also, $v_1$ is deterministic, so $(v_1, U) \eqdist (v_1, U V)$.
    
    For the induction step, we split the proof into $2$ cases. The proofs in both cases will be very similar, but with minor differences.
    
    \paragraph{Case 1: $m$ is a triangular number.} This means that the $m$-th vector added is $v_k$, where $m = \frac{k(k+1)}{2}$. Let $V_1$ be a random orthogonal matrix fixing $W_{m-1}$ and $V_2$ be a random orthogonal matrix fixing $W_m$. Our goal is then to show $(X_k, U) \eqdist (X_k, U V_2)$.

    To make this rigorous, we note an order of generating the random variables. First, we generate $U$ randomly: $W_m$ and $X_k$ are deterministic in terms of $U$. Next, we define $V_1$ to be a random rotation fixing $W_{m-1}$. Finally, we define $V_2$ to be a random rotation fixing $W_m$, where $V_1, V_2$ are conditionally independent on $U$.

    First, we prove that $(X_k, U) \eqdist (X_k, UV_1)$. Note that $U \eqdist UV_1$ by our inductive hypothesis. In addition, since $V_1$ fixes $W_{m-1}(U)$, $W_{m-1}(U) = W_{m-1}(UV_1)$ by Proposition \ref{prop:invariant}. Since $m = \frac{k(k+1)}{2}$ is a triangular number, $X_k(\cdot)$ is a deterministic function of $W_{m-1}(\cdot)$, which means $X_k(U) = X_k(UV_1)$. Hence, $(X_k, U) \eqdist (X_k(UV_1), UV_1) = (X_k, UV_1)$.
        
    Next, we prove that $(X_k, U V_2) \eqdist (X_k, U V_1 V_2)$. It suffices to prove that 
\[(X_k, U, V_2) \eqdist (X_k, UV_1, V_2)\,.\]
    To do so, we first show that $V_2 = f(U, R)$, where $f$ is a deterministic function and $R$ represents a random orthogonal matrix over $d-\dim(W_m)$ dimensions that is independent of $U$. (Recall that $W_m$ is a deterministic function of $U$.) To define $f(U, R)$, we consider some deterministic map that sends each $W_m$ to a set of $d-\dim(W_m)$ basis vectors in $W_m^\perp$. We then define $V_2 = f(U, R)$ to act on $W_m^\perp$ using $R$ and the correspondence of basis vectors. Since $W_m$ and $X_k$ are deterministic in terms of $U$, this means $f(U, R)$ is well-defined.
    We will now show that 
\[V_2 = f(U, R) = f(UV_1, R) \hspace{0.5cm} \text{and} \hspace{0.5cm} X_k = X_k(UV_1)\,.\]
    Since $U \eqdist UV_1$ by our inductive hypothesis,
    \begin{align*}
        (X_k, U, V_2) \eqdist (X_k(UV_1), UV_1, f(UV_1, R)) = (X_k, UV_1, V_2)\,.
    \end{align*}
    By Proposition \ref{prop:invariant}, $W_{m-1}(U) = W_{m-1}(UV_1),$ and since $X_k(\cdot)$ is deterministic given $W_{m-1}(\cdot)$ for $m = \frac{k(k+1)}{2}$, $X_k(U) = X_k(UV_1)$. This implies $W_m(U) = W_m(UV_1),$ which means $f(UV_1, R) = f(U, R)$, since $f(\cdot, R)$ only depends on $W_m(\cdot)$ and $R$. This completes the proof.
        
    Next, we show that $(X_k, U V_1 V_2) \eqdist (X_k, U V_1)$. Since we chose the order with $U$ being defined first, we are allowed to condition on $U$. Since $X_k$ is deterministic in terms of $U$, it suffices to show that $V_1 V_2\mid U \eqdist V_1\mid U$. Since $W_{m-1}, W_m$ are also deterministic given $U$, note that $V_1$ is a uniformly random orthogonal matrix fixing $W_{m-1},$ and $V_2$ is a random orthogonal matrix fixing $W_{m} \supset W_{m-1}$. Since $V_1$ and $V_2$ are conditionally independent given $U$, this means $V_1 V_2\mid U$ is a uniformly random orthogonal matrix fixing $W_{m-1}$, so $V_1 V_2\mid U \eqdist V_1\mid U$.
    
    In summary, we have that
\begin{align*}
    (X_k, U) &\eqdist (X_k, U V_1) \\
    &\eqdist (X_k, U V_1 V_2) \\
    &\eqdist (X_k, U V_2)\,.
\end{align*}
    
    \paragraph{Case 2: $m$ is not a triangular number.} Again, let $V_1$ be a random orthogonal matrix fixing $W_{m-1}$ and $V_2$ be a random orthogonal matrix fixing $W_m$. Our goal is again to show that $(X_k, U) \eqdist (X_k, U V_2)$. 

    First, we again have $(X_k, U V_1) \eqdist (X_k, U)$ by our inductive hypothesis.
    
    Next, we show that $(X_k, U V_2) \eqdist (X_k, U V_2 V_1)$. It suffices to prove that 
\[(X_k, U, V_2) \eqdist (X_k, UV_1, V_1^\top V_2 V_1)\,,\]
    since $(U V_1) (V_1^\top V_2 V_1) = U V_2 V_1$.
    We recall the random variable $R$ and use the same function $V_2 = f(U, R)$.
    Since we have already shown that $U \eqdist UV_1$, this implies that $(X_k, U, V_2) \eqdist (X_k(UV_1), UV_1, f(UV_1, R))$.
    Since $m$ is not triangular, $X_k(\cdot)$ is contained in $W_{m-1}(\cdot)$, so by Proposition \ref{prop:invariant}, $X_k(U) = X_k(UV_1)$. So, we have 
\[(X_k, U, V_2) \eqdist (X_k(UV_1), UV_1, f(UV_1, R)) = (X_k, UV_1, f(UV_1, R))\,.\]
    Now, if we fix $U$ and $V_1$, $W_{m-1}(UV_1) = W_{m-1}(U)$ by Proposition \ref{prop:invariant}. However, since the $m$-th $(i, j)$ pair has $i \ge 1$ when $m$ is not triangular, the final vector in $W_m(UV_1)$ will be $V_1^\top \Lambda^i V_1 v_j = V_1^\top (\Lambda^i v_j)$. For fixed $U, V_1$, $f(U, R)$ is a random rotation fixing $W_{m-1}$ and $\Lambda^i v_j$, but $f(UV_1, R)$ is a random rotation fixing $W_{m-1}$ and $V_1^\top (\Lambda^i v_j)$. Since $V_1^\top$ fixes $W_{m-1}$ by how we defined $V_1$, this means that for fixed $U, V_1$, $f(U, R)$ is a random rotation fixing $W_m$ but $f(UV_1, R)$ is a random rotation fixing $V_1^\top W_m$. Therefore, conditioned on $U, V_1$, $f(UV_1, R)$ has the same distribution as $V_1^\top f(U, R) V_1$. Since $X_k$ is deterministic in terms of $U$, this means
\[(X_k, UV_1, f(UV_1, R)) \mid U, V_1 \eqdist (X_k, UV_1, V_1^\top f(U, R) V_1)\mid U, V_1\,.\]
    We can remove the conditioning to establish that $(X_k, UV_1, f(UV_1, R)) \eqdist (X_k, UV_1, V_1^\top f(U, R) V_1) = (X_k, UV_1, V_1^\top V_2 V_1),$ which completes the proof.
        
    Next, we show that $(X_k, U V_2 V_1) \eqdist (X_k, U V_1)$. The proof is essentially the same as in the case when $m$ is triangular. We again condition on $U$, and we have that $V_2 V_1\mid U \eqdist V_1\mid U$ have the same distribution as uniform orthogonal matrices fixing $W_{m-1}(U)$. Since $X_k$ is a deterministic function of $U$, this means $(X_k, U V_2 V_1)\mid U \eqdist (X_k, U V_1)\mid U,$ and removing the conditioning finishes the proof.
    
    In summary, 
\begin{align*}
    (X_k, U) &\eqdist (X_k, U V_1) \\
    &\eqdist (X_k, U V_2 V_1) \\
    &\eqdist (X_k, U V_2)\,. \qedhere
\end{align*}
\end{proof}

We now prove our main conditioning lemma, which will be a modification of Lemma \ref{lem:random_rotation_fixed_subspace}.

\begin{lemma}[conditioning lemma]\label{lem:conditioning_lemma}
    Let all notation be as in Lemma \ref{lem:random_rotation_fixed_subspace}, and let $V_0$ be a fixed orthogonal matrix fixing $W_m$. Importantly, $V_0$ is a deterministic function only depending on $W_m$ (and not directly on $U$). Then, $(X_k, U) \eqdist (X_k, UV_0)$.
\end{lemma}
\begin{proof}
    First, note that since $V_0$ is a deterministic function of $W_m$, it is also a deterministic function of $U$. We can write $V_0(\cdot)$ as this function, and $V_0 = V_0(U)$.

    Now, Lemma \ref{lem:random_rotation_fixed_subspace} proves that $(X_k, U) \eqdist (X_k, UV)$. Note that conditioned on $U$, $V$ is a random matrix fixing $W_m$ and $V_0$ is a fixed matrix fixing $W_m$, which means that $VV_0\mid U \eqdist V\mid U$. Hence, $(X_k, UV) \eqdist (X_k, UVV_0)$. But from Proposition \ref{prop:invariant}, $X_k(UV) = X_k(U)$ and $W_m(UV) = W_m(U)$, which means that $V_0(\cdot)$, which only depends on $W_m(\cdot)$, satisfies $V_0(UV) = V_0(U)$. Hence, because $U \eqdist UV$, we have $(X_k, UVV_0) = (X_k(UV), UV \cdot V_0(UV)) \eqdist (X_k(U), U \cdot V_0(U)) = (X_k, UV_0)$.

    In summary, we have that $(X_k, U) \eqdist (X_k, UV) \eqdist (X_k, UVV_0) \eqdist (X_k, UV_0)$, which completes the proof.
\end{proof}

\subsubsection{From query algorithms to block Krylov algorithms}\label{scn:from_query}

In this section, we carry out the high-level outline from Section~\ref{scn:block_kry_overview}.
We aim to prove the following result, which implies that any adaptive deterministic algorithm in the extended oracle model can be simulated by rotating the output of a block Krylov algorithm.

\begin{lemma}[reduction to block Krylov]\label{lem:chen_block_krylov}
    Suppose $\Lambda = U^\top D U$, where $U$ is a Haar-random orthogonal matrix and $D$ is a diagonal matrix drawn from some (possibly unknown) distribution.
    Let $v_1, v_2(\cdot), \dotsc, v_{K}(\cdot)$ be an adaptive deterministic algorithm that makes $K$ orthonormal queries, where $K^2 < d$. Let $\valg_1, \valg_2, \dotsc, \valg_{K}$ be recursively defined as follows: $\valg_1 = v_1$, and $\valg_k = v_k(\{\Lambda^i \valg_j\}_{H_{k-1}})$ for $k \ge 2$. Let $z_1, \dotsc, z_{K}$ be i.i.d.\ standard Gaussian vectors. Then, from the collection $\{\Lambda^i z_j\}_{H_K}$ (without knowledge of $D$ or $\Lambda$), we can construct a set of unit vectors $\vsim_1,\vsim_2, \dotsc, \vsim_{K}$, and a set of rotation matrices $\Usim_1, \Usim_2, \dotsc, \Usim_{K}$, where $\vsim_k$ and $\Usim_k$ only depend on $\{\Lambda^i z_j\}_{H_{k-1}}$ and $z_k$, and such that
\[\{(\Usim_{1:K})^\top \Lambda^i \vsim_j\}_{H_{K}} \eqdist \{ \Lambda^i \valg_j \}_{H_{K}},\]
    where $\Usim_{1:K} \deq \Usim_1 \dotsm \Usim_{K}$, and the equivalence in distribution is over the randomness of $\Lambda$ and $\{z_i\}_{i \le K}$. Moreover, $\{\Lambda^i \tilde{v}_j\}_{H_K}$ is deterministically determined by $\{\Lambda^i z_j\}_{H_K}$.
\end{lemma}

Lemma \ref{lem:chen_block_krylov} says that the knowledge of $\Lambda^i z_j$ alone is sufficient to reconstruct the distribution of any adaptive algorithm's queries and responses.
The proof of the lemma requires introducing a hefty amount of notation, but we emphasize that it follows along the lines of Section~\ref{scn:block_kry_overview}.

First, we describe how to construct $\vsim_k$. Let $\vsim_1 = \frac{z_1}{\norm{z_1}}$, and for $k \ge 2$, let $\vsim_k$ be the unit vector parallel to the component of $z_k$ that is orthogonal to the span of $\{\Lambda^i z_j\}_{H_{k-1}}$. (With probability $1$, this is well-defined.)

Because each $\vsim_k$ is a linear combination of $\{\Lambda^i z_j\}_{H_{k-1}}$ and $z_k$, we can construct the set $\{\Lambda^i \vsim_j\}_{H_K}$ from the set $\{\Lambda^i z_j\}_{H_K}$.

We now construct the rotation matrices $\Usim_k$. First, we define matrix-valued functions $U_k(\cdot)$, for $k=1, \dots, K$, as follows.
\begin{definition}[rotations fixing previous queries and responses]
    For $1 \le k \le K$, the function $U_k(\cdot)$ takes arguments $\{x_{i,j}\}_{H_{k-1}}$, $y_k$, $z_k$, where the vectors $y_k$ and $z_k$ have unit norm and are both orthogonal to the collection $\{x_{i,j}\}_{H_{k-1}}$. 
    
    To define $U_1(\cdot)$: since $H_{0}$ is empty, the first function $U_1$ only takes arguments $y_1, z_1$, and is such that $U_1(y_1, z_1)$ is a deterministic orthogonal matrix that satisfies $U_1(y_1, z_1)^\top y_1 = z_1$. Note that $U_1(\cdot)$ exists because $y_1$ and $z_1$ both have unit norm; for example, we can complete $y_1$ and $z_1$ to orthonormal bases $(y_1,y_2,\dotsc,y_{d})$, $(z_1,z_2,\dotsc,z_{d})$ and take $U_1(y_1,z_1) = \sum_{i=1}^{d} y_i z_i^\top$.

To define $U_k(\cdot)$: $U_k(\{x_{i,j}\}_{H_{k-1}}, y_k, z_k)$ is a deterministic orthogonal matrix that satisfies
\begin{align}
\begin{aligned}\label{eq:uk_func}
    U_k^\top x_{i,j} &= x_{i,j}\,, \qquad \text{for all}~ (i, j) \in H_{k-1}\,, \\
    U_k^\top y_k &= z_k\,.
\end{aligned}
\end{align}
Such a choice of $U_k$ is always possible, because $k^2 < d$, and because $y_k$ and $z_k$ are orthogonal to $x_{i,j}$; for example, we can start with the identity matrix on the subspace spanned by $\{x_{i,j}\}_{H_{k-1}}$ and add to it a sum of outer products formed by completing $y_k$ and $z_k$ to two orthonormal bases of the orthogonal complement.
\end{definition}

Next, we describe how to construct $\Usim_k$. We will define $\Usim_k$ along with an auxiliary sequence $\{\vssim_k\}_{k=1,2,\dotsc,K-1}$. 

\begin{definition}[simulated sequences]
    We let $\vssim_1 = v_1$, and $\Usim_1 = U_1(\vsim_1, \vssim_1)$. For $k \ge 2$, $\vssim_k$ and $\Usim_k$ are defined recursively as follows:
\begin{align}
\begin{aligned}\label{eq:vs_us}
    \vssim_k &= v_k\bigl(\{(\Usim_{1:(k-1)})^\top \Lambda^i \vsim_j\}_{H_{k-1}}\bigr)\\
    \Usim_k &= U_k\bigl(\{(\Usim_{1:(k-1)})^\top \Lambda^i \vsim_j\}_{H_{k-1}},\; (\Usim_{1:(k-1)})^\top \vsim_k,\; \vssim_k\bigr)\, .
\end{aligned}
\end{align}
\end{definition}

Intuitively, one can think of $\vssim_k$ as the $k$th vector the simulator thinks the algorithm is querying, and $\Usim_k$ as a rotation that corresponds $\vssim_k$ to the random unit vector known by block Krylov.

\begin{proposition}[existence of rotations]
    Each $\Usim_k$ is well-defined.
\end{proposition}

\begin{proof}
    To show that this choice of $\Usim_k$ is possible, we need to check that $(\Usim_{1:(k-1)})^\top \vsim_k$, $\vssim_k$ both have unit norm and are orthogonal to the subspace $S_k$ spanned by $(\Usim_{1:(k-1)})^\top \Lambda^i \vsim_j$ for $(i, j) \in H_{k-1}$. They both have unit norm because $\vsim_k$ and $\vssim_k$ are constructed to have unit norm, and inductively we can assume $\Usim_{1:(k-1)}$ is orthogonal. Note that $\vssim_k$ is orthogonal to $S_k$ by our assumption on the function $v_k(\cdot)$, and $(\Usim_{1:(k-1)})^\top \vsim_k$ is also orthogonal to $S_k$ because
\begin{align*}
    \langle (\Usim_{1:(k-1)})^\top \Lambda^i \vsim_j, (\Usim_{1:(k-1)})^\top\vsim_k\rangle = \langle \Lambda^i \vsim_j, \vsim_k\rangle = 0\, ,
\end{align*}
    where the second line follows from the definition of $\tilde v_k$.
\end{proof}

We summarize some additional properties of $\vssim_k$ and $\Usim_k$ in the following lemma.
\begin{lemma}[properties of the simulated sequences]\label{lem:krylov_aux}
The variables $\Usim_k$ and $\vssim_k$ for $k = 1, \dotsc, K$ defined above satisfy the following properties:
\begin{enumerate}[label=(P\arabic*), ref=(P\arabic*)]
\item \label{prop1} $\vssim_k$ depends only on $\{\Lambda^i \vsim_j\}_{H_{k-1}}$, and $\Usim_k$ depends only on $\{\Lambda^i \vsim_j\}_{i+j \le k}$.
\item \label{prop2} 
For any $k \ge j$, we have
\begin{align*}
    \vsim_j = \Usim_{1:k} \vssim_j\, .
\end{align*}
\item \label{prop3} For $k \ge 2$, $\vssim_k$ satisfies
\begin{align*}
    \vssim_k
    &= v_k\bigl(\{(\Usim_{1:(k-1)})^\top \Lambda^i \Usim_{1:(k-1)} \vssim_j\}_{H_{k-1}}\bigr)\,.
\end{align*}
\item \label{prop4} For $k \ge 2$, $\Usim_k$ satisfies
\begin{align*}
    \Usim_k &= U_k\bigl(\{(\Usim_{1:(k-1)})^\top \Lambda^i \Usim_{1:(k-1)}\vssim_j\}_{H_{k-1}},\; (\Usim_{1:(k-1)})^\top \vsim_k,\; \vssim_k\bigr)\, .
\end{align*}
\end{enumerate}
\end{lemma}

\begin{proof}
    \ref{prop1} is immediate from the definitions, since $\{(i, j): i+j \le k\} = H_{k-1} \cup \{(0, k)\}$. 
    
    To show \ref{prop2}, note that the second property of the function $U_k$ from \eqref{eq:uk_func} implies that
    \begin{align}\label{eq:tildev}
         \vssim_j = (\Usim_j)^\top (\Usim_{1:(j-1)})^\top\vsim_j = (\Usim_{1:j})^\top\vsim_j\, . 
    \end{align}
    This proves~\ref{prop2} for $k = j$. To prove~\ref{prop2} for $k > j$, we use induction on $k$. If~\ref{prop2} holds for $k-1 \ge j$, then
    \begin{align}
        (\Usim_{1:k})^\top \vsim_j = (\Usim_k)^\top (\Usim_{1:(k-1)})^\top \vsim_j = (\Usim_{1:(k-1)})^\top \vsim_j = \vssim_j.
    \end{align}
    Above, the middle equality holds by the first property of \eqref{eq:uk_func}, since $\Usim_k$ fixes $(\Usim_{1:(k-1)})^\top \vsim_j$ because $j \le k-1$. The final equality holds by our inductive hypothesis. So, \ref{prop2} holds for $k$.
    
    Finally, \ref{prop3} and \ref{prop4} then follow from~\ref{prop2}, since $k-1 \ge j$ if $j \in H_{k-1}$.
\end{proof}

We highlight the importance of~\ref{prop2} for $k = K$, which roughly states that $(\Usim_{1:K})^\top$ actually sends each block Krylov-generated vector $\vsim_j$ to the simulated vector $\vssim_j$.

Before proving Lemma~\ref{lem:chen_block_krylov}, we must make one more basic definition.

\begin{definition}[queries and data]
    For $k \ge 2$, given the matrix $\Lambda$ and a set $\{v_j\}_{1 \le j \le k-1}$, define $\mathfrak{C}_k$ as the function that satisfies $\mathfrak{C}_k(\Lambda, \{v_j\}_{1 \le j \le k-1}) = \{\Lambda^i v_j\}_{H_{k-1}}$. In addition, define $\mathfrak{D}_{k} = v_k \circ \mathfrak{C}_k$.
\end{definition}

We are now ready to prove Lemma~\ref{lem:chen_block_krylov}.
Although the proof is notationally burdensome, the message is that we can show the equality of distributions inductively by repeatedly invoking the conditioning lemma (Lemma~\ref{lem:conditioning_lemma}), which is designed precisely for the present situation.

\medskip

\begin{proof}[Proof of Lemma~\ref{lem:chen_block_krylov}]
    For $1 \le k \le K$, let $\Lambda_{k} \deq (\Usim_{1:k})^\top \Lambda \Usim_{1:k}$.
    Since we can write $(\Usim_{1:k})^\top \Lambda^i \vsim_j = (\Usim_{1:k})^\top \Lambda^i (\Usim_{1:k}) \vssim_j = \Lambda_k^i \vssim_j$ for any $k \ge j$ by~\ref{prop2} of Lemma \ref{lem:krylov_aux}, it suffices to inductively prove that for all $1 \le k \le K$, 
    \begin{align} \label{eq:goal2}
        (\Lambda_k, \{\vssim_j\}_{1 \le j \le k}) \eqdist (\Lambda, \{\valg_j\}_{1 \le j \le k})\,.
    \end{align}

    For the base case of $k = 1$, it suffices to show that $(\Lambda_1, \vssim_1) \eqdist (\Lambda, \valg_1)$. Note, however, that $\vssim_1 = \valg_1 = v_1$, and $\Lambda_1 = (\Usim_1)^\top \Lambda (\Usim_1) = U_1(\vsim_1, v_1)^\top \Lambda U_1(\vsim_1, v_1)$.
    Since $v_1$ is a deterministic vector, $\vsim_1$ is independent of $\Lambda$, and the distribution of $\Lambda$ is rotationally invariant, the claim follows.

    For the inductive step, assume we know $(\Lambda_k, \{\vssim_j\}_{1 \le j \le k}) \eqdist (\Lambda, \{\valg_j\}_{1 \le j \le k})$. 
    Then, note that $\valg_{k+1} = v_{k+1}(\{\Lambda^i \valg_j\}_{H_k})$ and $\vssim_{k+1} = v_{k+1}(\{\Lambda_k^i \vssim_j\}_{H_k})$. Thus, we have $\valg_{k+1} = \mathfrak{D}_{k+1}(\Lambda, \{\valg_j\}_{1 \le j \le k})$ and $\vssim_{k+1} = \mathfrak{D}_{k+1}(\Lambda_k, \{\vssim_j\}_{1 \le j \le k})$. In addition, because $\Usim_{k+1}$ fixes $\Lambda_k^i \vssim_j$ for all $(i, j) \in H_k$ by~\ref{prop4}, we also have that $\Lambda_{k+1}^i \vssim_j = \Lambda_k^i \vssim_j$ for all $(i, j) \in H_k$, which means $\vssim_{k+1} = \mathfrak{D}_{k+1}(\Lambda_{k+1}, \{\vssim_j\}_{1 \le j \le k})$. Therefore, it suffices to show 
    \begin{align} \label{eq:goal3}
        (\Lambda_{k+1}, \{\vssim_j\}_{1 \le j \le k})
        \eqdist (\Lambda, \{\valg_j\}_{1 \le j \le k})\,,
    \end{align}
    as this implies $(\Lambda_{k+1}, \{\vssim_j\}_{1 \le j \le k+1}) \eqdist (\Lambda, \{\valg_j\}_{1 \le j \le k+1})$, which completes the inductive step.

    Next, we show that $\Usim_{k+1}$ sends $\vsim_{k+1}$ to a random unit vector orthogonal to the simulated queries so far.
    Note that $\Lambda_{k+1} = (\Usim_{k+1})^\top \Lambda_k (\Usim_{k+1})$, where,  by~\ref{prop4},
    \begin{equation} \label{eq:Usim_k+1_redefined}
        \Usim_{k+1} = U_{k+1}(\{\Lambda_k^i \vssim_j\}_{H_k}, (\Usim_{1:k})^\top\tilde{v}_{k+1}, \vssim_{k+1})\,.
    \end{equation}
   Note that $\tilde{v}_{k+1}$ has the law of a random unit vector conditional on being orthogonal to $\{\Lambda^i z_j\}_{H_k}$, or equivalently, it is a random unit vector orthogonal to $\{\Lambda^i \vsim_j\}_{H_k}$. Since
   \begin{align*}
       (\Usim_{1:k})^\top \Lambda^i \vsim_j = (\Usim_{1:k})^\top \Lambda^i(\Usim_{1:k}) \vssim_j = \Lambda_k^i \vssim_j
   \end{align*}
   for all $(i, j) \in H_k$ (by~\ref{prop2}), this means that $(\Usim_{1:k})^\top \vsim_{k+1}$ is orthogonal to $\{\Lambda_k^i \vssim_j\}_{H_k}$.
   The random direction of $\tilde{v}_{k+1}$ has no dependence on $\{\Lambda^i \tilde{v}_j\}_{H_k}$ apart from being orthogonal to them, which means by~\ref{prop1}, $(\Usim_{1:k})^\top \vsim_{k+1}$ is a \emph{uniformly random} unit vector orthogonal to $\{\Lambda_k^i \vssim_j\}_{H_k}$.
    
    Recalling that $\vssim_{k+1} = \mathfrak{D}_{k+1}(\Lambda_{k}, \{\vssim_j\}_{1 \le j \le k})$, this means that we can rewrite \eqref{eq:Usim_k+1_redefined} as
    \begin{align}
        \Usim_{k+1} =& \hspace{0.1cm} U_{k+1}\left(\{\Lambda_k^i \vssim_j\}_{H_k}, \hat{v}^{\msf{sim}}, \mathfrak{D}_{k+1}(\Lambda_{k}, \{\vssim_j\}_{1 \le j \le k})\right)\,, \label{eq:uk+1-sim-new}\\
    \intertext{where $\hat{v}^{\msf{sim}}$ is a random unit vector orthogonal to $\{\Lambda_k^i \vssim_j\}_{H_k}$.
    As a result, if we define}
        U_{k+1}^{\msf{alg}} \deq& \hspace{0.1cm} U_{k+1}\bigl(\{\Lambda^i \valg_j\}_{H_k}, \hat{v}^{\msf{alg}}, \mathfrak{D}_{k+1}(\Lambda, \{\valg_j\}_{1 \le j \le k})\bigr)\,, \label{eq:uk+1-alg}
    \end{align}
    where $\hat{v}^{\msf{alg}}$ is a random unit vector orthogonal to $\{\Lambda^i \valg_j\}_{H_k}$, then 
    \begin{align*}
        (\Lambda_{k+1}, \{\vssim_j\}_{1 \le j \le k})
        &= \bigl((\Usim_{k+1})^\top \Lambda_k (\Usim_{k+1}), \{\vssim_j\}_{1 \le j \le k}\bigr) \\
        &\eqdist \bigl((U_{k+1}^{\msf{alg}})^\top \Lambda (U_{k+1}^{\msf{alg}}), \{\valg_j\}_{1 \le j \le k}\bigr)\,.
    \end{align*}
    Above, the first equality follows by definition, and the second follows from our inductive hypothesis that $(\Lambda_k, \{\vssim_j\}_{1 \le j \le k}) \eqdist (\Lambda, \{\valg_j\}_{1 \le j \le k})$, along with \eqref{eq:uk+1-sim-new} and \eqref{eq:uk+1-alg}.

    We are now in a position to apply the conditioning lemma (Lemma~\ref{lem:conditioning_lemma}).
    Note that $U_{k+1}^{\msf{alg}}$ only depends on $\{\Lambda^i \valg_j\}_{H_k}$ (as well as some randomness in $\hat{v}^{\msf{alg}}$, but the randomness is independent of everything else given $\{\Lambda^i \valg_j\}_{H_k}$, so we can safely condition on it). Hence, we can apply the conditioning lemma with $U_{k+1}^{\msf{alg}}$, to obtain that 
    \begin{align*}
        (\Lambda_{k+1}, \{\vssim_j\}_{1 \le j \le k}) \eqdist \bigl((U_{k+1}^{\msf{alg}})^\top \Lambda (U_{k+1}^{\msf{alg}}), \{\valg_j\}_{1 \le j \le k}\bigr) \eqdist \bigl(\Lambda, \{\valg_j\}_{1 \le j \le k}\bigr) \,,
    \end{align*}
    which establishes \eqref{eq:goal3} and thereby concludes the proof.
\end{proof}

With the block Krylov reduction in hand, we can now establish our second lower bound for sampling from Gaussians.

\begin{theorem}[second lower bound for sampling from Gaussians] \label{thm:lower_bound_main}
There is a universal constant $\epsilon_0 > 0$ such that the query complexity of sampling from Gaussian distributions $\mc N(0, \Sigma)$ in $\R^d$, where the condition number $\kappa$ of $\Sigma$ satisfies $\kappa \le d^{1/5-\delta}$, with accuracy $\epsilon_0$ in total variation distance is at least $\Omega_\delta(\sqrt\kappa \log d)$.
\end{theorem}

\begin{proof}
    Let $U$ be a random orthogonal matrix, and let $\Lambda = U^\top D U$, $\Lambda' = U^\top D' U$ be as in Lemma \ref{lem:block_krylov}.
    We first show that if $\kappa \le d^{1/5-\delta}$ and $c$ is a sufficiently small constant, no adaptive algorithm that makes less than $c_\delta \sqrt{\kappa} \log d$ queries to the extended oracle can distinguish between $\Lambda$ and $\Lambda'$, with $\Omega(1)$ probability. 

    First we assume that the algorithm is deterministic, so its behavior is characterized by functions $v_1, v_2(\cdot), \dotsc, v_{K}(\cdot)$, as in Lemma~\ref{lem:chen_block_krylov}. The algorithm then proceeds to make queries $\valg_1, \valg_2, \dotsc, \valg_{K}$, where $\valg_k = v_k(\{\Lambda^i \valg_j\}_{H_{k-1}})$. Lemma~\ref{lem:chen_block_krylov} shows that the output of the algorithm $\{\Lambda^i \valg_j\}_{H_K}$ can be entirely simulated by a block Krylov algorithm, which receives $\{\Lambda^i z_k\}_{H_K}$, where $z_1, \dotsc, z_{K}$ are i.i.d. standard Gaussians. Lemma~\ref{lem:block_krylov} says that a block Krylov algorithm that makes $K = c_\delta\sqrt{\kappa}\log d$ queries, where $c_\delta$ is a small constant depending on $\delta$ and $\kappa \le d^{1/5-\delta}$, cannot distinguish between $\Lambda$ and $\Lambda'$ with $\Omega(1)$ advantage, which then implies the same for any deterministic algorithm.
    
    If the algorithm is randomized, then it uses a random seed $\xi$ that is independent of $\Lambda$ and $\Lambda'$. So conditional on the random seed, the algorithm will not be able to distinguish $\Lambda$ and $\Lambda'$ with $\Omega(1)$ advantage, so the overall probability that the randomized algorithm successfully distinguishes $\Lambda$ and $\Lambda'$ also cannot be $\Omega(1)$.
    
    Finally, we note that a sample from $\mc N(0, \Lambda^{-1})$ versus $\mc N(0, \Lambda'^{-1})$ can distinguish between the two cases. This means that even if we were able to draw a sample that was $\frac{1}{3}$-far in total variation distance, we could output the correct answer with probability at least $\frac{2}{3}$. This implies that any sampling algorithm must require at least $\Omega_\delta(\sqrt \kappa \log d)$ queries to the extended oracle, and hence at least same number of queries to the standard oracle.
\end{proof}

\section*{Acknowledgments}
The authors thank Ainesh Bakshi, Patrik R.\ Gerber, Piotr Indyk, Thibaut Le Gouic, Philippe Rigollet, Adil Salim, Terence Tao, and Kevin Tian for helpful conversations. 
SC was supported by the NSF TRIPODS program (award DMS-2022448).
JD was supported by a UCLA dissertation year fellowship.
CL was supported by the Eric and Wendy Schmidt Center at the Broad Institute of MIT and Harvard.
SN was supported by a Google Fellowship, the NSF TRIPODS program (award DMS-2022448), and the NSF Graduate Fellowship.

\printbibliography

\appendix

\section{Upper bound for log-concave sampling in constant dimension}\label{sec:low_d_upper}

In this section we give a simple proof that in constant dimension, one can approximately generate a sample from a log-concave distribution with condition number $\kappa$, in $O(\log \kappa)$ queries. Our query dependence also has a polylogarithmic dependence on $\frac{1}{\eps}$, if we wish to generate a sample that is $\eps$-close in TV distance to the true distribution. (We do not attempt to optimize the dependence on dimension $d$ or the polylogarithmic dependence on $\frac{1}{\eps}$.)

Let $V$ be a convex function that is $1$-strongly convex and $\kappa$-smooth, such that $V$ is minimized at the origin and $V(0) = 0$. For any real value $y \ge 0$, define $B_V(y)$ to be the set of points $x$ such that $V(x) \le y$.

First, we note the following basic facts that follow immediately from our convexity assumptions.

\begin{proposition}[basic facts about log-concavity] \label{prop:convexity_stuff}
\begin{enumerate}
    \item $B_V(y)$ is a convex body for any $y > 0$, and contains $0$.
    \item $B_V(y)$ is contained in the ball of radius $\sqrt{2 y}$ and contains the ball of radius $\sqrt{2 y/\kappa}$.
    \item For any $0 < y < y'$, $B_V(y') \subset \frac{y'}{y} \, B_V(y)$.
\end{enumerate}
\end{proposition}

Next, we show how to obtain a crude $d^{O(1)}$-approximation for $B_V(1)$ using $d^{O(1)} \log \kappa$ first-order queries. The proof is essentially folklore and follows from the ellipsoid method.

\begin{proposition}[ellipsoid method] \label{prop:ellipsoid}
    Let $B$ be a convex body that contains $B(0, r)$ and is contained in $B(0, R)$, along with a membership and separation oracle.
    Using $d^{O(1)} \log \frac{R}{r}$ adaptive queries to the membership and separation oracle, we can find an ellipsoid $E$ centered around some point $z$ such that $E \subset B \subset E'$, where $E'$ is $E$ dilated by an $O(d^{3/2})$ factor about $z$.
\end{proposition}

We can apply the above proposition to the convex body $B_V(1)$.

\begin{corollary}[sublevel set approximation] \label{cor:ellipsoid}
    Using $d^{O(1)} \log \kappa$ adaptive queries to $V$ and $\nabla V$, we can find an ellipsoid $E$ centered around some point $z$ such that $E \subset B_V(1) \subset E'$, where $E'$ is $E$ dilated by an $O(d^{3/2})$ factor about $z$.
\end{corollary}

\begin{proof}
    It suffices to show that from a single first-order query at a point $x$, we can generate a membership and separation oracle for $B_V(1)$. Indeed, the  membership part is straightforward as we just check whether $V(x) \le 1$ (which is equivalent to $x \in B_V(1)$). The separation oracle is also simple, and can be done using the gradient. Specifically, suppose that $V(x) > 1$; then,
    $V(x') \ge V(x) + \langle \nabla V(x), x'-x \rangle$, which means that every $x'$ with $\langle \nabla V(x), x' \rangle \ge \langle \nabla V(x), x \rangle$ is such that $V(x') \ge V(x) > 1$, i.e., $\nabla V(x)$ is a separation oracle for $B_V(1)$ at $x$.
\end{proof}

We are able to prove our sampling upper bound, using a rejection sampling approach.

\begin{theorem}[upper bound for log-concave sampling]\label{thm:low_d_upper}
    For any constant $d \ge 2$ and any $1$-strongly convex and $\kappa$-smooth function $V$ with minimum at $0$, we can approximately sample from $\pi\propto \exp(-V)$ to total variation error at most $\eps$ using $O(\log \kappa + \log^{O(1)}(1/\varepsilon))$ adaptive queries to $V$ and $\nabla V$ (here we emphasize that the asymptotic notation treats $d$ as constant).
\end{theorem}

\begin{proof}
    Given $V$ and any integer $t \ge 1$, let $p_t$ be the probability that a sample from $\pi$ lies in $t B_V(1)$. The normalizing constant is $Z \ge \int_{B_V(1)} \exp(-V) \ge e^{-1} \vol(B_V(1))$, but integral over $(t+1) B_V(1) \backslash t B_V(1)$ is
    \begin{align*}
        \int_{(t+1) B_V(1) \backslash t B_V(1)} \exp(-V) \le \exp(-t) \vol\bigl((t+1) B_V(1)\bigr) = \exp(-t) \,(t+1)^d \vol\bigl(B_V(1)\bigr)\,,
    \end{align*}
    using Proposition \ref{prop:convexity_stuff} which implies that $V(x) \ge t$ for any $x \not\in t B_V(1)$. 
    Therefore, the probability of $(t+1) B_V(1) \backslash t B_V(1)$ under $\pi$ is at most
    \begin{align*}
        \pi\bigl((t+1) B_V(1) \backslash t B_V(1)\bigr)
        &\le \frac{\exp(-t)\, {(t+1)}^d \vol(B_V(1))}{e^{-1} \vol(B_V(1))} = \exp(-(t-1)) \, {(t+1)}^d\,.
    \end{align*}
    By summing this quantity for all integers greater than $t$, the probability of the complement of $t B_V(1)$ is at most $\sum_{u \ge t} \exp(-(u-1)) \,(u+1)^d = \sum_{u\ge t} \exp(-u+d\log(u+1)+1)$. Note that for $t \ge \Omega(d \log d)$, this quantity is at most $O(\exp(-t/2))$.
    Taking $t = C \, (d\log d + \log(1/\varepsilon))$ for a large constant $C$, we obtain $\pi(\R^d\backslash tB_V(1)) \le \varepsilon/2$.
    
    The algorithm now works as follows. We use Corollary \ref{cor:ellipsoid} to find $E \subset B_V(1) \subset E'$. We pick a uniformly random point $X$ in $t E'$ for $t = C\,(d \log d + \log(1/\eps))$. We then accept the point $X$ with probability $\exp(-V(X))$, and if we reject we restart the procedure. First, note that this algorithm, upon termination, samples exactly from $\pi$ conditioned on $tE'$, which is at most $\frac{\eps}{2}$ away from $\pi$ in total variation distance. In addition, each rejection sampling step succeeds with probability at least $\vol(E)/(e\vol(t E'))$, since with probability $\vol(E)/\vol(t E')$ we choose a point in $E$ in which case $V(X) \le 1$ so we accept with probability at least $e^{-1}$. This is equal to $1/(t\, O(d^{3/2}))^d = d^{-O(d)} \, t^{-d} = d^{-O(d)} \, (\log \frac{1}{\eps})^{-d}.$ So, after $(d \log \frac{1}{\eps})^{O(d)}$ rounds of rejection sampling, each of which only needs one query to $V$, we accept the sample with probability at least $1-\frac{\eps}{2}$, which means that overall we have generated a sample which is $\eps$-close in distribution to $\pi$ in total variation distance.
    
    The overall query complexity is a combination of finding $E, E'$ and then running the rejection sampling, for a total complexity of $d^{O(1)} \log \kappa + (d \log \frac{1}{\eps})^{O(d)}$. So, for any fixed dimension $d$ and error probability $\eps$, the query complexity for log-concave sampling is $O(\log \kappa)$. In addition, the dependence on the error probability is polylogarithmic for any fixed $d$.
\end{proof}

\begin{remark}
    We briefly note that the exponential dependence on $d$ is not necessary: using more sophisticated tools developed for sampling from convex bodies one should be able to obtain a complexity of $\log(\kappa)\, (d \log \frac{1}{\eps})^{O(1)}$. However, we choose to not optimize the dimension dependence in this result for the sake of simplicity, and since we are focused on the setting of $d = O(1)$.
\end{remark}

\section{Upper bound for sampling from Gaussians}\label{sec:gaussian_upper}

Finally, we show a simple proof that, using only $O(\min(\sqrt{\kappa} \log d, d))$ gradient queries, one can generate an approximate sample from a Gaussian $\mathcal{N}(0, \Sigma)$ in $d$ dimensions. Note that the density evaluated at $x$, up to an additive constant, equals $-\frac{1}{2}\, x^\top \Lambda x$ for $\Lambda = \Sigma^{-1}$, which means that a gradient query at $x$ amounts to receiving the matrix-vector product $\Lambda x$.

First, we require a well-known proposition from approximation theory.

\begin{proposition}[{\cite[Theorem 3.3]{sachdeva_survey}}] \label{prop:approx_theory_1}
    For any positive integer $s$ and $0 < \delta < 1,$ there exists a polynomial $p_{s, \delta}$ of degree $\lceil \sqrt{2s \ln (2/\delta)} \rceil$ such that $|p_{s, \delta}(x) - x^s| \le \delta$ for all $x \in [-1, 1]$.
\end{proposition}

As a corollary, we have the following result.

\begin{proposition}[polynomial approximation of inverse square root] \label{prop:approx_theory_2}
    For any $\kappa \ge 2$ and $\delta < \frac{1}{2}$, there exists a polynomial $q_{\kappa, \delta}$ of degree $O(\sqrt{\kappa} \log \frac{\kappa}{\delta})$ such that $|q_{\kappa, \delta}(x) - x^{-1/2}| \le \delta/\sqrt\kappa$ for all $1 \le x \le \kappa$.
\end{proposition}

\begin{proof}
    First, consider the function $(1+x)^{-1/2}$. For $|x| \le 1-\frac{1}{\kappa} < 1$, we can use the Taylor series to write
    \[(1+x)^{-1/2} = 1 + \sum_{t=1}^\infty \frac{(\frac{1}{2}-1)\,(\frac{1}{2}-2)\,(\frac{1}{2}-3) \dotsm (\frac{1}{2}-t)}{t!} \,x^t = 1 + \sum_{i=1}^{\infty} c_t x^t\,,\]
    where $|c_t| \le 1$ for all $t \ge 1$.
    
    Note that for $|x| \le 1-\frac{1}{\kappa}$, $\left|\sum_{t > T} c_t x^t\right| \le \sum_{t > T} |x|^t \le \frac{|x|^T}{1-\abs x}$. For $T = O(\kappa \log \frac{\kappa}{\delta}),$ we can bound this by $\frac{(1-1/\kappa)^T}{1/\kappa} \le \frac{\delta}{2}$. Therefore, for all such $x$,
\[ \Bigl\lvert (1+x)^{-1/2} - \sum_{t=0}^{T} c_t x^t\Bigr\rvert \le \frac{\delta}{2}\,,\]
    where we have set $c_0  \deq 1$.
    
    Next, using Proposition \ref{prop:approx_theory_1}, we can replace each $x^t$ with $p_{t, \delta}(x)$ where $p_{t,\delta}$ is a polynomial of degree $O(\sqrt{t \log (t/\delta)})$ such that $|p_{t, \delta}(x) - x^t| \le \delta/(4t^2)$ for all $|x| \le 1$. (We also let $p_{0, \delta}$ simply be the constant function $1$.) Therefore, 
\[\Bigl\lvert (1+x)^{-1/2} - \sum_{t=0}^{T} c_t p_{t, \delta}(x)\Bigr\rvert \le \frac{\delta}{2} + \sum_{t = 1}^{T} |c_t| \, \frac{\delta}{4t^2} \le \delta\,.\]
    In addition, the polynomial $\hat p \deq \sum_{t=0}^{T} c_t p_{t, \delta}$ has degree at most $O(\sqrt{T \log (T/\delta)}) = O(\sqrt{\kappa} \log \frac{\kappa}{\delta})$.
    
    To finish, $|\hat{p}(x-1)-x^{-1/2}| \le \frac{\delta}{\kappa}$ for all $\frac{1}{\kappa} \le x \le 1$, which means that
    \begin{align*}
        \Bigl\lvert\hat{p}\bigl(\frac{x}{\kappa}-1\bigr) \, \frac{1}{\sqrt{\kappa}} - x^{-1/2}\Bigr\rvert \le \frac{\delta}{\sqrt\kappa} \qquad\text{for all}~1\le x \le \kappa\,.
    \end{align*}
    So, there exists a polynomial $q_{\kappa, \delta}$ with $q_{k,\delta}(x) = \hat{p}(\frac{x}{\kappa}-1) \, \frac{1}{\sqrt{\kappa}}$, such that $q_{\kappa, \delta}$ has degree $O(\sqrt{\kappa} \log \frac{\kappa}{\delta})$ and $|q_{\kappa, \delta}(x) - x^{-1/2}| \le \delta/\sqrt\kappa$ for all $1 \le x \le \kappa$.
\end{proof}

We are now ready to prove our query complexity upper bound.

\begin{theorem}[optimal algorithm for sampling from Gaussians]\label{thm:gaussian_upper}
    Let $\Lambda = \Sigma^{-1}$ be an unknown positive definite matrix with all eigenvalues between $1$ and $\kappa$. Then, using $O(\min(\sqrt{\kappa} \log \frac{d}{\eps}, d))$ adaptive matrix-vector queries to $\Lambda$, we can produce a sample from a distribution $\hat\pi$ such that $\KL(\hat\pi \mmid \mathcal{N}(0, \Sigma)) \le \varepsilon^2$.
\end{theorem}

\begin{proof}
    Choose $X \sim \mathcal{N}(0, I_d)$, define $R = O(\sqrt{\kappa} \log \frac{\kappa}{\delta})$ be the degree of $q_{\kappa, \delta}$, and for simplicity write $q(x) \deq q_{\kappa, \delta}(x) \deq \sum_{i=0}^R a_i x^i$.
    The algorithm works as follows. Using the power method, we compute $X, \Lambda X, \Lambda^2 X, \dotsc, \Lambda^R X$. We output $Y = \sum_{i=0}^R a_i \,\Lambda^i X$.
    Note that $Y \sim \mc N(0, \hat\Sigma)$, where we set $\hat\Sigma \deq (\sum_{i=0}^R a_i \Lambda^i)^2$.
    If $\lambda_1,\dotsc,\lambda_d$ denote the eigenvalues of $\Lambda$, then the eigenvalues of $\hat\Sigma$ are ${q(\lambda_1)}^2,\dotsc, {q(\lambda_d)}^2$.
    The KL divergence is given by
    \begin{align*}
        \KL\bigl(\mc N(0,\hat\Sigma) \bigm\Vert \mc N(0,\Sigma)\bigr)
        &\lesssim \sum_{k=1}^d {\abs{q(\lambda_k)^2 \,\lambda_k - 1}^2}
        \lesssim \sum_{k=1}^d {\abs{q(\lambda_k)\,\lambda_k^{1/2} - 1}^2}
        \lesssim \sum_{k=1}^d \lambda_k\, {\abs{q(\lambda_k) - \lambda_k^{-1/2}}^2} \\
        &\lesssim d\kappa \, \frac{\delta^2}{\kappa}\,.
    \end{align*}
    If we set $\delta \asymp \varepsilon/\sqrt d$, then we obtain a KL divergence of at most $\varepsilon^2$.
    
    Finally, we can also learn $\Lambda$ by querying $\Lambda e_i$ for each unit basis vector $e_1, \dots, e_d$. So, we can thus learn $\Sigma$, and then generate a perfect random sample from $\mathcal{N}(0, \Sigma)$. Hence, the query complexity of generating a sample from $\mathcal{N}(0, \Sigma)$ is at most $O(\min(\sqrt{\kappa} \log \frac{\kappa d}{\eps}, d)) = O(\min(\sqrt{\kappa} \log \frac{d}{\eps}, d))$.
\end{proof}

\begin{remark}
    If $\pi$ is an $\alpha$-strongly log-concave distribution, then from Pinsker's inequality and Talagrand's transport inequality,
    \begin{align*}
        \max\{\norm{\mu-\pi}_{\rm TV}^2,\; \alpha \, W_2^2(\mu,\pi)\}
        &\lesssim \KL(\mu \mmid \pi)\,.
    \end{align*}
    Hence, this algorithmic result for Gaussians complements the two lower bounds in Corollary~\ref{cor:wishart_lower_bd}.
\end{remark}

\end{document}